\documentclass[12pt,reqno]{amsart}
\usepackage{amssymb}
\usepackage{enumerate}
\usepackage[T1]{fontenc}
\usepackage{mathrsfs}
\usepackage[all]{xy}
\usepackage[usenames]{color}

\newcommand{\aaa}{\mathbf{\mathfrak{a}}}
\newcommand{\bbb}{\mathbf{\mathfrak{b}}}
\newcommand{\bfC}{\textit{\textbf{C}}}  
\newcommand{\bfD}{\textit{\textbf{D}}}  
\newcommand{\bfR}{\textit{\textbf{R}}}  
\newcommand{\bfa}{\textit{\textbf{a}}}
\newcommand{\bfb}{\textit{\textbf{b}}}

\newcommand{\OP}[1]{\textup{(OP$#1$)}}
\newcommand{\omres}[2][]{($\Omega_{#1}$-#2)}
\newcommand{\dn}{{\downarrow}}

\definecolor{darkgreen}{RGB}{0,120,0}
\definecolor{darkred}{RGB}{200,0,0}
\definecolor{bluegreen}{RGB}{0,150,150}
\long\def\new#1{{\color{darkred}#1}}

\newcommand{\mynote}[1]{{\color{blue}\noindent\textbf{\textup{[#1]}}}}

\newcommand{\3}[1]{\underline{#1}}
\newcommand{\4}[1]{\widebar{#1}}
\newcommand{\5}[1]{\widehat{#1}}

\newcommand{\RR}{R}
\newcommand{\kk}{k}

\let\xto=\xrightarrow
\let\too=\longrightarrow

\let\fromm=\longleftarrow

\let\Gamma=\varGamma

\newcommand{\higherlim}[2]{\displaystyle\setbox1=\hbox{\rm lim}
	\setbox2=\hbox to \wd1{\leftarrowfill} \ht2=0pt \dp2=-1pt
	\setbox3=\hbox{$\scriptstyle{#1}$}
	\def\test{#1}\ifx\test\empty
	\mathop{\mathop{\vtop{\baselineskip=5pt\box1\box2}}}\nolimits^{#2}
	\else
	\ifdim\wd1<\wd3
	\mathop{\hphantom{^{#2}}\vtop{\baselineskip=5pt\box1\box2}^{#2}}_{#1}
	\else
	\mathop{\mathop{\vtop{\baselineskip=5pt\box1\box2}}_{#1}}%
	\nolimits^{#2}
	\fi\fi}

\let\oldcirc=\circ
\renewcommand{\circ}{\mathchoice
    {\mathbin{\scriptstyle\oldcirc}}{\mathbin{\scriptstyle\oldcirc}}
    {\mathbin{\scriptscriptstyle\oldcirc}}
    {\mathbin{\scriptscriptstyle\oldcirc}}}

\SelectTips{cm}{10} \UseTips   

\numberwithin{equation}{section}

\mathcode`\:="003A
\mathchardef\cdot="0201

\renewenvironment{enumerate}[1][]
{\begin{enumerat}[#1]\setlength{\itemsep}{6pt}}{\end{enumerat}}

\newenvironment{enuma}{\begin{enumerate}[{\rm(a) }]}{\end{enumerate}}
\newenvironment{enumi}{\begin{enumerate}[{\rm(i) }]}{\end{enumerate}}

\renewenvironment{itemize}
{\begin{itemiz}\setlength{\itemsep}{6pt}} 
{\end{itemiz}}

\def\beq#1\eeq{\begin{equation*}#1\end{equation*}}
\def\beqq#1\eeqq{\begin{equation}#1\end{equation}}
\newcommand{\boldd}[1]{{\mathversion{bold}\textbf{#1}}}

\let\emptyset=\varnothing

\DeclareMathAlphabet\EuR{U}{eur}{m}{n}
\SetMathAlphabet\EuR{bold}{U}{eur}{b}{n}
\newcommand{\curs}{\EuR}
\newcommand{\catdef}[2][]{\expandafter\newcommand\csname#2\endcsname%
{#1\curs{#2}}}
\catdef{Ab}    \catdef{Grp}   \catdef{Cat}   \catdef{Vect}   \catdef{Mod}
\catdef{Is}    \catdef{Top}   \catdef{hoTop}  \catdef{ho}   \catdef{FibSimp}

\renewcommand{\mod}{\textup{-}\curs{mod}}

\newcommand{\widebar}[1]{\overset{\mskip3mu\hrulefill\mskip3mu}{#1}
        \vphantom{#1}}

\renewcommand{\:}{\colon}

\newlength{\shorter}
\newlength{\shortest}
\newlength{\upto}\newlength{\dnto}
\newcounter{let} 
\loop\stepcounter{let}
\expandafter\edef\csname cal\alph{let}\endcsname%
{\noexpand\mathcal{\Alph{let}}}
\ifnum\thelet<26\repeat

\newcommand{\scrp}{\mathscr{P}}

\newcommand{\tdef}[2][]{\expandafter\newcommand\csname#2\endcsname%
{#1\textup{#2}}}
\tdef{Iso}   \tdef{Aut}    \tdef{Out}    \tdef{Inn}    \tdef{Hom}
\tdef{End}   \tdef{Inj}    \tdef{map}    \tdef{Ker}    \tdef{Ob}
\tdef{Mor}   \tdef{Res}    \tdef{Spin}   \tdef{Id}     \tdef{Fr}
\tdef{rk}    \tdef{conj}   \tdef{incl}   \tdef{proj}   \tdef{diag} 
\tdef{trf}   \tdef{Sol}    \tdef{He}     \tdef{Sz}     \tdef{cj}
\tdef{free}  \tdef{ev}     \tdef{Map}    \tdef{Rep}    \tdef{pr}
\tdef{supp}  \tdef{res}    \tdef{hofibre}  \tdef{hofib} \tdef{tors}
\tdef{Coker} \tdef{Tor}    \tdef{Ext}    \tdef{Ind}    \tdef{Sing}
\tdef[_]{fc}  \tdef[_]{fn}  \tdef[_]{typ} \tdef[^]{op}   \tdef[^]{ab}
\tdef[^]{dg}

\newcommand{\chr}{\textup{char}}

\newcommand{\colim}{\mathop{\textup{colim}}\limits}

\newcommand{\GL}{\textit{GL}}

\newcommand{\fdef}[1]{\expandafter\newcommand\csname#1\endcsname%
{\mathfrak{#1}}}
\fdef{X}  \fdef{red}  \fdef{foc}  \fdef{hyp}  \fdef{Sub}

\newcommand{\bbdef}[1]{\expandafter\newcommand%
\csname#1\endcsname{\mathbb{#1}}}
\bbdef{C} \bbdef{F} \bbdef{R} \bbdef{Z} \bbdef{Q} \bbdef{N}

\newcommand{\KK}[1][]{_{K#1}}

\renewcommand{\2}{\partial}

\newcommand{\gee}{\varepsilon}
\newcommand{\SFL}[1][]{(S#1,\calf#1,\call#1)}

\newcommand{\gen}[1]{\langle{#1}\rangle}
\newcommand{\Gen}[1]{\bigl\langle{#1}\bigr\rangle}

\let\nsg=\normal

\newcommand{\syl}[2]{\textup{Syl}_{#1}(#2)}
\newcommand{\sylp}[1]{\syl{p}{#1}}

\newcommand{\sminus}{\smallsetminus}
\newcommand{\defeq}{\overset{\textup{def}}{=}}

\renewcommand{\Im}{\textup{Im}}

\newcommand{\pcom}{{}^\wedge_p}
\newcommand{\Rcom}{{}^\wedge_{\RR}}
\newcommand{\sd}[1]{\overset{{#1}}{\rtimes}}

\let\til=\widetilde

\newtheorem{Thm}[equation]{Theorem}
\newtheorem{Prop}[equation]{Proposition}
\newtheorem{Cor}[equation]{Corollary}
\newtheorem{Lem}[equation]{Lemma}

\newtheorem{Defi}[equation]{Definition}

\newtheorem{Th}{Theorem}

\newtheorem{Prp}[Th]{Proposition}

\theoremstyle{definition}
\newtheorem{Rmk}[equation]{Remark}
\newtheorem{Quest}[equation]{Question}
\newtheorem{Ex}[equation]{Example}

\newcommand{\longleft}[1]{\;{\leftarrow%
\count255=0 \loop \mathrel{\mkern-6mu}%
    \relbar\advance\count255 by1\ifnum\count255<#1\repeat}\;}
\newcommand{\longright}[1]{\;{\count255=0 \loop \relbar\mathrel{\mkern-6mu}%
    \advance\count255 by1\ifnum\count255<#1\repeat\rightarrow}\;}
\newcommand{\Right}[2]{\overset{#2}{\longright{#1}}}
\newcommand{\RIGHT}[3]{\mathrel{\mathop{\kern0pt\longright{#1}}
    \limits^{#2}_{#3}}}
\newcommand{\Left}[2]{{\buildrel #2 \over {\longleft{#1}}}}
\newcommand{\LEFT}[3]{\mathrel{\mathop{\kern0pt\longleft{#1}}\limits^{#2}_{#3}}
}
\newcommand{\longleftright}[1]{\;{\leftarrow\mathrel{\mkern-6mu}%
    \count255=0\loop\relbar\mathrel{\mkern-6mu}%
    \advance\count255 by1\ifnum\count255<#1\repeat\rightarrow}\;}

\newcommand{\onto}[1]{\;{\count255=0 \loop \relbar\joinrel
    \advance\count255 by1
    \ifnum\count255<#1 \repeat \twoheadrightarrow}\;}

\newcommand{\RLEFT}[3]{\mathrel{%
   \mathop{\vcenter{\baselineskip=0pt\hbox{$\kern0pt\longright{#1}$}%
   \hbox{$\kern0pt\longleft{#1}$}}}\limits^{#2}_{#3}}}

\usepackage[usenames]{color}

\usepackage[usenames]{color}                                      %
\newcommand{\xxto}[1]{\mathrel{\mathop{%
  \setbox0\hbox{$\ {\scriptstyle#1}\ $}%
  \hbox to \wd0{\rightarrowfill}}^{#1}}%
}

\newcommand{\xlto}[2][]{%
  \mathrel{\mathop{%
    \setbox0\vbox{
      \hbox{$\scriptstyle\;\;{#1}\;\;$}%
      \hbox{$\scriptstyle\;\;{#2}\;\;$}%
    }%
    \hbox to\wd0{\leftarrowfill}\displaystyle}%
  \limits^{#2}\ifx{#1}{}\else{_{#1}}\fi}%
}

\newcommand{\ctg}{(\calc,\theta)}

\newcommand{\Osystem}[4]{(#1,#2;#3,#4)}
\newcommand{\Osys}[3][\theta]{\Osystem{#2}{#3}{#1_*}{#1^*}}
\newcommand{\abth}[1][]{\Osystem{\cala#1}{\calb#1}{{\theta#1}_*}{{\theta#1}^*}}
\newcommand{\cpth}[1][\calc]{\Osystem{\RR#1\mod}{\RR\pi\mod}{\theta_*}{\theta^*}}

\newcommand{\mxtwo}[4]{\left(\begin{smallmatrix}#1&#2\\#3&#4
\end{smallmatrix}\right)}

\newcommand{\Bobj}[1][]{{\oldcirc}_{#1}}

\author{C. Broto}
\address{Departament de Matem\`atiques, Edifici Cc, Universitat Aut\`onoma de
Barcelona, E--08193 Cerdanyola del Vall\`es (Barcelona), Spain
\newline\indent 
Centre de Recerca Matem\`atica, Edifici Cc, Campus de Bellaterra, 08193
Cerdanyola del Vall\`es (Barcelona), Spain} 
\email{broto@mat.uab.cat}

\thanks{C. Broto  acknowledges financial support from the Spanish Ministry 
of Economy through the ``Mar\'ia de Maeztu'' Programme for Units of 
Excellence in R\&D (MDM-2014-0445) and FEDER-MINECO Grant MTM2016-80439-P 
and from the Generalitat de Catalunya through AGAUR Grant 2017SGR1725.}

\author{R. Levi}
\address{Department of Mathematical Sciences, University of Aberdeen,
Fraser Noble 138, Aberdeen AB24 3UE, U.K.}
\email{r.levi@abdn.ac.uk}

\author{B. Oliver}
\address{Universit\'e Paris 13, Sorbonne Paris Cit\'e, LAGA, UMR 7539 du CNRS, 
99, Av. J.-B. Cl\'ement, 93430 Villetaneuse, France.}
\email{bobol@math.univ-paris13.fr}
\thanks{B. Oliver is partially supported by UMR 7539 of the CNRS}

\thanks{R. Levi and B. Oliver were partly supported by FEDER-MINECO 
Grant MTM2016-80439-P during several visits to the Universitat Aut\`onoma 
de Barcelona.}

\thanks{The authors also thank the University of Aberdeen 
for its support during visits by two of us.}

\subjclass{Primary 55R35. Secondary 55R40, 20D20}

\keywords{Classifying space, $p$-completion, finite groups, fusion.}

\title{Loop space homology of a small category}
\date{}                                

\subjclass{Primary 55R35. Secondary 55R40, 20D20}

\keywords{Classifying space, Loop space, Small category, $p$-completion, 
Finite groups, Fusion.}

\begin{document}

\begin{abstract} 
In a 2009 paper, Dave Benson gave a description in purely algebraic terms 
of the mod $p$ homology of $\Omega(BG\pcom)$, when $G$ is a finite group, 
$BG\pcom$ is the $p$-completion of its classifying space, and 
$\Omega(BG\pcom)$ is the loop space of $BG\pcom$. The main purpose of this 
work is to shed new light on Benson's result by extending it to a more 
general setting. As a special case, we show that if $\calc$ is a small 
category, $|\calc|$ is the geometric realization of its nerve, $R$ is a 
commutative ring, and $|\calc|^+_R$ is a ``plus construction'' for 
$|\calc|$ in the sense of Quillen (taken with respect to $R$-homology), 
then $H_*(\Omega(|\calc|^+_R);R)$ can be described as the homology of a 
chain complex of projective $R\calc$-modules satisfying a certain list of 
algebraic conditions that determine it uniquely up to chain homotopy. 
Benson's theorem is now the case where $\calc$ is the category of a finite 
group $G$, $R=\F_p$ for some prime $p$, and $|\calc|^+_R=BG\pcom$. 
\end{abstract}

\maketitle


\section*{Introduction}

Let $G$ be a finite group, and let $BG\pcom$ denote its classifying space 
after $p$-completion in the sense of Bousfield and Kan \cite{BK}. In 
general, the higher homotopy groups $\pi_i(BG\pcom)$ for $i\ge2$ can be 
nonvanishing, and hence the loop space $\Omega(BG\pcom)$ is interesting in 
its own right. These spaces are the subject of several papers by the second 
author (e.g., \cite[Theorem 1.1.4]{L-memoirs}). In particular, the homology 
of $\Omega(BG\pcom)$ with $p$-local coefficients is known to have some very 
interesting properties, as described in \cite[\S\,2]{CL}. This helped to 
motivate the question of whether the homology of $\Omega(BG\pcom)$ admits a 
purely algebraic definition (e.g., in  \cite[\S\,2.6]{CL}). 

In \cite{B}, Benson answered this question by showing that 
$H_*(\Omega(BG\pcom);\kk)$, for a field $\kk$ of characteristic $p$, is 
isomorphic to the homology of what he called a ``left $\kk$-squeezed 
resolution for $G$'': a chain complex of projective $\kk G$-modules 
satisfying certain axioms. He also showed that any two such complexes are 
chain homotopy equivalent, and hence have the same homology. The 
$\kk$-homology of $\Omega(BG\pcom)$ is thus determined by the axioms of a 
squeezed resolution.

Our original aim in this paper was to check whether Benson's concept of a 
squeezed resolution can be formulated in a more categorical context. This 
was motivated in part by the problem of identifying $p$-compact groups in 
the sense of Dwyer and Wilkerson: loop spaces with finite mod $p$ homology 
and $p$-complete classifying spaces (see Section 
\ref{s:examples} for more discussion). When doing this, we discovered that 
in fact, squeezed resolutions can be defined in a much more general 
setting, where we call them $\Omega$-resolutions to emphasize the 
connection to loop spaces. In this setting, Benson's result can be 
generalized to a statement about plus constructions (in the sense of 
Quillen) on nerves of small categories.

When $\calc$ is a small category, we let $|\calc|$ denote the geometric 
realization of the nerve of $\calc$. If $\RR$ is a commutative ring, then 
an \emph{$\RR\calc$-module} is a (covariant) functor $\calc\too\RR\mod$, 
and a morphism of $\RR\calc$-modules is a natural transformation of 
functors. When $\pi$ is a group, we let $\calb(\pi)$ be the category with 
one object $\Bobj[\pi]$ and $\End(\Bobj[\pi])=\pi$.

As usual, a group $G$ is called \emph{$\RR$-perfect} if $H_1(G;\RR)=0$. We 
say that $G$ is \emph{strongly $\RR$-perfect} if it is $\RR$-perfect and 
$\Tor(H_1(G;\Z),\RR)=0$. Clearly, all $\RR$-perfect groups are strongly 
$\RR$-perfect whenever $\RR$ is flat as a $\Z$-module. 

If $X$ is a connected CW complex and $H\nsg\pi_1(X)$, then a \emph{plus 
construction} for $X$ with respect to $\RR$ and $H$ means a space $X_\RR^+$ 
together with a map $\kappa\:X\too X_\RR^+$ such that $\pi_1(\kappa)$ is 
surjective with kernel $H$ and $H_*(\kappa;M)$ is an isomorphism for each 
$\RR[\pi_1(X)/H]$-module $M$. In Proposition \ref{plus}, we modify Quillen's 
construction to show that $R$-plus constructions exist if and only if 
$\chr(\RR)\ne0$ and $H$ is $\RR$-perfect, or $\chr(\RR)=0$ and $H$ is 
strongly $\RR$-perfect.

\begin{Th} \label{ThA}
Fix a commutative ring $\RR$, a small 
connected category $\calc$, a group $\pi$, and a functor 
$\theta\:\calc\too\calb(\pi)$ such that 
$\pi_1(|\theta|)\:\pi_1(|\calc|)\too\pi$ is surjective. Set 
$H=\Ker(\pi_1(|\theta|))$. Assume that $\chr(\RR)\ne0$ and $H$ is 
$\RR$-perfect, or $\chr(\RR)=0$ and $H$ is strongly $\RR$-perfect. Then there 
is an $\varOmega$-resolution
	\[ \cdots \Right2{\2_3} C_2 \Right2{\2_2} C_1 \Right2{\2_1} 
	C_0 \Right2{\gee} \theta^*(\RR\pi) \too 0 \]
(a chain complex of $\RR\calc$-modules satisfying conditions 
listed in Definition \ref{d:Omega-resolution} or Lemma \ref{O-res}), 
and $H_*(C_*,\2_*) \cong H_*(\Omega(|\calc|^+_{\RR});\RR)$ for each such 
$(C_*,\2_*)$ and each plus construction $|\calc|_\RR^+$ for $|\calc|$ with 
respect to $\RR$ and $H$.
\end{Th}

Theorem \ref{ThA} will be stated in a more precise form as Theorem 
\ref{Omega_C0}. Upon restricting to the case where $\RR=\F_p$ for a prime 
$p$, $\calc=\calb(G)$ for some finite group $G$, and 
$\pi=G/O^p(G)\cong\pi_1(BG\pcom)$ (the largest $p$-group quotient of $G$), 
we recover Benson's theorem, since $BG\pcom$ is a plus construction on 
$BG=|\calb(G)|$ with respect to the ring $\F_p$ and the subgroup $O^p(G)$. 

As another special case of Theorem \ref{ThA}, let $(S,\calf,\call)$ be a 
$p$-local compact group in the sense of \cite[Definition 4.2]{BLO3}. Thus 
$S$ is a discrete $p$-toral group (an extension of $(\Z/p^\infty)^r$ by a 
finite $p$-group), $\calf$ is a saturated fusion system over $S$, and 
$\call$ is a centric linking system associated to $\calf$. Set 
$\pi=\pi_1(|\call|\pcom)$: a finite $p$-group by \cite[Proposition 
4.4]{BLO3}. By Theorem \ref{Omega_C} or \ref{Omega_C0} applied with $\call$ 
in the role of $\calc$, $H_*(\Omega(|\call|\pcom);\F_p)$ can be described 
in terms of $\Omega$-resolutions. As noted above, our original motivation for this work was 
the search for new conditions sufficient to guarantee that 
$\Omega(|\call|\pcom)$ has finite homology, and hence that $|\call|\pcom$ 
is a $p$-compact group in the sense of Dwyer and Wilkerson \cite{DW}. We 
did not succeed in doing this, but our attempt to do so is what led to this 
more general setting. Also, we do construct some 
examples in Propositions \ref{torus-res}, \ref{Ores:Sullivan}, and 
\ref{finite-res} of explicit $\Omega$-resolutions of finite length (in 
fact, of minimal length) for certain $p$-compact groups. 

It turns out that $\Omega$-resolutions can be defined in much greater 
generality than that needed in Theorem \ref{ThA}. Let 
$\bigl(\cala\RLEFT2{\theta_*}{\theta^*}\calb\bigr)$ be an 
\emph{$\varOmega$-system}: a pair of abelian categories and additive 
functors such that $\theta_*$ is left adjoint to $\theta^*$, 
$\theta_*\theta^*\cong\Id_\calb$, and $\theta^*$ is exact (Definition 
\ref{d:Omega-system}). In this situation, for a projective object $X$ in 
$\calb$, an \emph{$\varOmega$-resolution} of $X$ is a chain complex of 
projective objects in $\cala$ augmented by a morphism to $\theta^*(X)$ 
which satisfies certain axioms listed in Definition 
\ref{d:Omega-resolution}. In particular, these axioms are minimal 
conditions needed to ensure the uniqueness of $\Omega$-resolutions up to 
chain homotopy equivalence (Proposition \ref{p2:functoriality}). However, 
while each such $X$ has at most one $\Omega$-resolution up to homotopy, 
we have examples that show that it 
need not have any in this very general situation. The examples in Theorem 
\ref{ThA} are the special case where $\cala=\RR\calc\mod$, 
$\calb=\RR\pi\mod$, and $\theta_*$ is left Kan extension with respect to 
the functor $\theta$. Another large family of examples, where we show 
that $\Omega$-resolutions exist but haven't yet found a geometric 
interpretation of their homology, is described in the following 
proposition (and in more detail in Proposition \ref{ex3}).

\begin{Prp} \label{PrB}
Let $\theta\:\calc\too\cald$ be a functor 
between small categories that is bijective on objects and surjective on 
morphism sets, and which has the following property: 
	\beq \parbox{\shortest}{for each $c,c'\in\Ob(\calc)$ and 
	each $\varphi,\varphi'\in\Mor_\calc(c,c')$ such that 
	$\theta_{c,c'}(\varphi)=\theta_{c,c'}(\varphi')$, there is 
	$\alpha\in\Aut_\calc(c')$ such that 
	$\theta_c(\alpha)=\Id_{\theta(c')}$ and 
	$\varphi=\alpha\varphi'$.} \eeq
Then for each commutative ring $\RR$, 
$\Bigl(\RR\calc\mod\RLEFT3{\theta_*}{\theta^*}\RR\cald\mod\Bigr)$ is an 
$\varOmega$-system, where $\theta_*$ is defined by left Kan extension. 
Furthermore, projective objects in $\RR\cald\mod$ all have 
$\varOmega$-resolutions if and only if 
$\Ker[\theta_c\:\Aut_\calc(c)\too\Aut_\cald(\theta(c))]$ is 
$\RR$-perfect for each $c\in\Ob(\calc)$.
\end{Prp}

We begin in Section \ref{s:Omega} by defining $\Omega$-resolutions in our 
most general setting and proving their uniqueness. In Section 
\ref{s:existence}, we find some necessary conditions, and some sufficient 
conditions, for their existence. We then restrict in Section \ref{s:kC-mod} 
to the special case of $\RR\calc$-modules, and construct examples where 
$\Omega$-resolutions do or do not exist (Propositions \ref{ex3} and 
\ref{ex4}). Our results connecting the homology of certain 
$\Omega$-resolutions to the homology of loop spaces are shown in Section 
\ref{s:loops}, where Theorem \ref{ThA} is stated and proved in a slightly 
more precise form as Theorem \ref{Omega_C0}. Afterwards, we look in Section 
\ref{s:examples} at some detailed examples of $\Omega$-resolutions arising 
from $p$-local compact groups in which the maximal torus is normal. 

All three authors would like to thank the referee for carefully 
reading the paper and making many helpful suggestions.

\textbf{Notation:} For each small category $\calc$, $|\calc|$ denotes its 
geometric realization. When $\calc$ is a small category and $\RR$ is a 
commutative ring, we let $\RR\calc\mod$ denote the category of 
``$\RR\calc$-modules'': covariant functors from $\calc$ to $\RR\mod$. When 
$\calc$ is an abelian category, we write $\scrp(\calc)$ to denote the class 
of projective objects in $\calc$. For a group $G$, we write $G\ab=G/[G,G]$ 
for the abelianization, and let $\calb(G)$ denote the category with one 
object $\Bobj[G]$ and $\End_{\calb(G)}(\Bobj[G])\cong G$.


\section{$\Omega$-systems and $\Omega$-resolutions}
\label{s:Omega}

We begin by defining  $\Omega$-resolutions and proving their uniqueness in 
a very general setting. We do not prove any results about the existence of 
$\Omega$-resolutions in this section, but leave that for Sections 
\ref{s:existence} and \ref{s:loops}. 

\begin{Defi} \label{d:Omega-system}
An \emph{$\Omega$-system} $\abth$ consists of a 
pair of abelian categories $\cala$ and $\calb$, together with additive 
functors 
	\[ \cala \RLEFT5{\theta_*}{\theta^*} \calb \]
such that
\begin{enumerate}[\rm(OP1) ]

\item $\theta_*$ is left adjoint to $\theta^*$; \label{def:op1}

\item $\theta_*$ is a retraction in the sense that the counit of the adjunction 
$\bbb\:\theta_*\circ\theta^*\too\Id_\calb$ is an isomorphism; and 
\label{def:op2}

\item $\theta^*$ sends epimorphisms in $\calb$ to epimorphisms in $\cala$. 
\label{def:op3}

\end{enumerate}
\end{Defi}

It will be important, in the situation of Definition \ref{d:Omega-system}, 
to know that $\theta^*(\calb)$ is a full subcategory of $\cala$. In fact, 
this holds without assuming condition \OP3.

\begin{Lem} \label{l:adj-pair}
Let $\cala$ and $\calb$ be a pair of categories, together with functors
	\[ \cala \RLEFT5{\theta_*}{\theta^*} \calb \]
such that $\theta_*$ is left adjoint to $\theta^*$, and such that the counit 
$\bbb\: \theta_*\circ\theta^*\too\Id_\calb$ of the adjunction is an isomorphism 
of functors. Then the image $\theta^*(\calb)$ is a full subcategory of $\cala$.
\end{Lem}

\begin{proof} For each $B,B'\in\Ob(\calb)$, 
	\[ \theta^*_{B,B'} \: \Mor_\calb(B,B') \Right4{} 
	\Mor_\cala(\theta^*(B),\theta^*(B'))\cong 
	\Mor_\calb(\theta_*\theta^*(B),B') \]
is a bijection since $\theta_*\theta^*\cong\Id_\calb$. (Our thanks to 
the editor for pointing out this simple argument.)
\end{proof}

The following is one family of $\Omega$-systems to which we will frequently 
refer. More examples will be given in Section 3.

\begin{Ex} \label{ex1}
Fix a commutative ring $\RR$, a pair of groups $G$ and $\pi$, and a 
surjective homomorphism $\theta\:G\too\pi$. Let $\RR G\mod$ and 
$\RR\pi\mod$ be the categories of (left) $\RR G$- and $\RR\pi$-modules, 
respectively, and let 
	\[ \RR G\mod \RLEFT6{\theta_*}{\theta^*} \RR \pi\mod \]
be the functors defined as follows. For each $\RR G$-module $M$, set 
$\theta_*(M)=\RR\pi\otimes_{\RR G}M$ where $\RR\pi$ is regarded as a right 
$\RR G$-module via $\theta$. For each $\RR\pi$-module $N$, set 
$\theta^*(N)=N$ regarded as an $\RR G$-module via $\theta$. Then $\cpth[G]$ 
is an $\Omega$-system. 
\end{Ex}

\begin{proof} If $M$ and $N$ are $\RR G$- and $\RR\pi$-modules, 
respectively, then there is an obvious natural bijection 
$\Hom_{\RR G}(M,\theta^*(N))\cong\Hom_{\RR\pi}(\theta_*(M),N)$. Thus \OP1 
holds: $\theta_*$ is left adjoint to $\theta^*$. 
Conditions \OP2 and \OP3 are clear.
\end{proof}

The following properties of $\Omega$-systems follow easily from the basic 
properties of abelian categories and adjoint functors (See, e.g., 
\cite[\S\,IV-V]{MacLane-cat}.).

\begin{Lem} \label{l:Omega-props}
The following hold for each $\Omega$-system $\abth$.
\begin{enuma} 

\item The functor $\theta^*$ is exact, and $\theta_*$ is right exact.

\item The functor $\theta_*$ sends projectives to projectives.

\item A sequence in $\calb$ is exact if and only if its image under 
$\theta^*$ is exact in $\cala$. A morphism in $\calb$ is an isomorphism, an 
epimorphism, or a monomorphism if and only if the same is true in $\cala$ 
of its image under $\theta^*$.

\end{enuma}
\end{Lem}

\begin{proof} 
\noindent\textbf{(a,b)} Since $\theta_*$ is left adjoint to $\theta^*$, 
$\theta_*$ is right exact and $\theta^*$ is left exact. By \OP3, $\theta^*$ 
also preserves epimorphisms, and hence is exact. Since $\theta_*$ has a 
right adjoint that is exact, it sends projectives to projectives.

\smallskip

\noindent\textbf{(c) } The exactness of $\theta^*$ implies that it sends 
the kernel, cokernel, and image of each morphism $\psi$ in $\calb$ to the 
kernel, cokernel, and image in $\cala$ of $\theta^*(\psi)$. Also, 
if 
$\varphi\in\Iso_\cala(\theta^*(N),\theta^*(N'))$ for $N,N'$ in $\calb$, 
then 
since 
$\theta^*(\calb)$ is a full subcategory of $\cala$ and $\theta_*\theta^*$ 
is naturally isomorphic to the identity, 
$\varphi=\theta^*(\psi)$ for some $\psi\in\Iso_\calb(N,N')$. Hence a 
sequence in $\calb$ whose image under $\theta^*$ is exact in $\cala$ is 
also exact in $\calb$, and if $\theta^*(\varphi)$ is a monomorphism or 
epimorphism in $\cala$, then $\varphi$ is a monomorphism or epimorphism, 
respectively, in $\calb$. This proves the ``if'' statements, and the 
converse in all cases holds by the exactness of $\theta^*$. 
\end{proof}

We are now ready to define $\Omega$-resolutions.

\begin{Defi}\label{d:Omega-resolution}
Let $\abth$ be an $\Omega$-system. For a projective object $X$ in $\calb$, 
an \emph{$\Omega$-resolution} of $X$ with respect to $\abth$ is a 
chain complex 
	\beqq \bfR = \bigl( \cdots \Right3{\2_3} P_2 \Right3{\2_2} P_1 
	\Right3{\2_1} P_0 \Right3{\gee} \theta^*(X) \Right2{} 0 \bigr) 
	\label{e2:sqres} \eeqq
in $\cala$ such that 
\begin{enumerate}[\rm($\Omega$-1) ]
\item $P_n$ is projective in $\cala$ for all $n\geq 0$; 
\label{sqres-1}

\item $\theta_*(\bfR)$ is exact; and 
\label{sqres-2}

\item $H_n(P_*,\2_*)$ is isomorphic to an object in $\theta^*(\calb)$ for 
each $n\ge0$, and $\gee$ induces an isomorphism 
$H_0(P_*,\2_*)\cong\theta^*(X)$. \label{sqres-3}

\end{enumerate} 
If an $\Omega$-resolution $\bfR$ exists as above, then we set 
$H_*^\Omega(\cala,\calb;X)=\theta_*\bigl(H_*(P_*,\2_*)\bigr)$: the image 
under $\theta_*$ of the homology of the complex $(P_*,\2_*)$. 
\end{Defi}

There are, in fact, many $\Omega$-systems for which $\Omega$-resolutions do 
not exist. In the situation of Example \ref{ex1}, when $\RR$ is a 
field and $\theta\:G\too\pi$ is a surjection of groups, we will 
see in Example \ref{ex1a} that a nonzero projective object in $\calb$ has 
an $\Omega$-resolution if and only if $H_1(\Ker(\theta);\RR)=0$. 
However, whenever $X$ does have at least one 
$\Omega$-resolution, the next proposition implies that 
$H_*^\Omega(\cala,\calb;X)$ is unique up to natural isomorphism.

\begin{Prop}\label{p2:functoriality}
Let $\abth$ be an $\Omega$-system. Let $X$ and $Y$ be projective objects in 
$\calb$, and let $f\in\Mor_\calb(X,Y)$ be a morphism. Let 
	\[ \cdots \xto{~\2_2~} P_1 \xto{~\2_1~} P_0 
	\xto{~\gee~} \theta^*(X) \too 0 
	\quad\textup{and}\quad
	\cdots \xto{~\2'_2~} P'_1 \xto{~\2'_1~} P'_0 
	\xto{~\gee'~} \theta^*(Y) \too 0 \]
be chain complexes in $\cala$, where the first satisfies conditions 
\textup{($\Omega$-\ref{sqres-1})} and \textup{($\Omega$-\ref{sqres-2})} in Definition 
\ref{d:Omega-resolution} and the second satisfies condition 
\textup{($\Omega$-\ref{sqres-3})}. Then 
there are morphisms $f_n\in\Mor_\cala(P_n,P'_n)$ which make the following 
diagram commute:
	\[ \vcenter{\xymatrix@C=40pt@R=30pt{
	\cdots \ar[r]^{\2_3} & P_2 \ar[r]^{\2_2} \ar[d]_{f_2} & P_1 \ar[r]^{\2_1} 
	\ar[d]_{f_1} & P_0 \ar[r]^{\gee} \ar[d]_{f_0} & \theta^*(X) \ar[r] 
	\ar[d]_{\theta^*(f)} & 0 \\
	\cdots \ar[r]^{\2'_3} & P'_2 \ar[r]^{\2'_2} & P'_1 \ar[r]^{\2'_1} 
	& P'_0 \ar[r]^{\gee'} & \theta^*(Y) \ar[r] & 0 \\
	}} \]
Furthermore, $\{f_n\}_{n\in\N}$ is unique up to chain homotopy.
\end{Prop}

\begin{proof} For each $i\ge0$, $\theta_*(P_i)$ is projective in $\calb$ by 
Lemma \ref{l:Omega-props}(b) and since $P_i$ is projective. Also, 
$\theta_*\theta^*(X)\cong X$ by \OP2, and $X$ is projective in $\calb$ by 
assumption. So by \omres{\ref{sqres-2}}, 
$\theta_*(P_*)\too\theta_*\theta^*(X)\to0$ is an exact sequence of 
projective objects in $\calb$, and hence splits in each degree.

\smallskip

\noindent\textbf{Existence of $f_*$: } Since $P_0$ is projective, and the 
augmentation $\gee'\:P'_0\Right2{}\theta^*(Y)$ is onto by 
\omres{\ref{sqres-3}}, $\theta^*(f)\circ\gee$ lifts to a homomorphism 
$f_0\:P_0\Right2{}P'_0$.

Assume, for some $n\ge0$, that $f_i$ has been constructed for all $0\le 
i\le n$, where $f_{n-1}\circ \2_n=\2'_n\circ f_n$.  Then $f_n\circ \2_{n+1}$ 
sends $P_{n+1}$ into $\Ker(\2'_n)$, and hence induces a homomorphism 
$\chi\:P_{n+1}\Right2{}H_n(P'_*,\2'_*)$. By assumption 
\omres{\ref{sqres-3}}, 
$H_n(P'_*,\2'_*)\cong\theta^*(B)$ for some $B$ in $\calb$, and hence there 
are natural bijections 
	\[ \Mor_\cala(P_i,H_n(P'_*,\2'_*))\cong \Mor_\cala(P_i,\theta^*(B)) 
	\cong \Mor_\calb(\theta_*(P_i),B) \]
for $i=n,n+1,n+2$. Since the complex 
$(\theta_*(P_*),\theta_*(\2_*))\too\theta_*(\theta^*X)\to0$ is exact and 
split and $\chi\circ \2_{n+2}=0$ by construction, $\theta_*(\chi)$ factors 
through $\Im(\theta_*(\2_{n+1}))$ and extends to $\theta_*(P_n)$. By 
adjointness, there is $\varphi\:P_n\Right2{}H_n(P'_*,\2'_*)$ such that 
$\chi=\varphi\circ \2_{n+1}$.  

Since $P_n$ is projective, $\varphi$ lifts to a morphism 
$\til\varphi\:P_n\Right2{}\Ker(\2'_n)$.  An easy diagram chase now shows 
that $\Im((f_n-\til\varphi)\circ \2_{n+1})\le\Im(\2'_{n+1})$. Hence, upon 
replacing $f_n$ by $f_n-\til\varphi$, $f_{n-1}\circ \2_n=\2'_n\circ f_n$ 
still holds (where $f_{-1}=f$ if $n=0$) and 
$\Im(f_n\circ\2_{n+1})\le\Im(\2'_{n+1})$. Upon using the projectivity of 
$P_{n+1}$ again, one can lift $f_n\circ \2_{n+1}$ to a homomorphism 
$f_{n+1}\:P_{n+1}\Right2{}P'_{n+1}$ such that $f_n\circ\2_{n+1} = 
\2'_{n+1}\circ f_{n+1}$. We now continue inductively.

\smallskip

\noindent\textbf{Uniqueness of $f_*$: } Let $f'_*$ and $f''_*$ be two 
homomorphisms covering $f$, and set $t_*=f_*'-f''_*$.  Thus 
$t_*\:(P_*,\2_*)\Right2{}(P'_*,\2'_*)$ is a homomorphism covering 
$X\Right2{0}Y$, and we must construct a chain homotopy 
$D\:P_*\Right2{}P'_{*}$ of degree $+1$ such that $D\circ \2 + \2'\circ D = 
t_*$.

Set $D_{-1}=0\:\theta^*(X)\Right2{}P'_0$.  Since $\gee'\circ t_0=0$, and the 
sequence 
	\[ P'_1\Right5{\2'_1}P'_0\Right5{\gee'}
	\theta^*(Y) \Right2{} 0 \] 
is exact by condition \omres{\ref{sqres-3}}, $t_0$ lifts to a homomorphism 
$D_0\:P_0\Right2{}P'_1$. 
The rest of the proof is carried out using arguments similar to those used 
to show existence.
\end{proof}

When $\abth$ is an $\Omega$-system and $X\in\scrp(\calb)$ has an 
$\Omega$-resolution, there is a spectral sequence that links the 
$\Omega$-homology of $X$ to higher derived functors of $\theta_*$.

\begin{Prop} \label{Omega-sp.seq.}
Let $\abth$ be an $\Omega$-system, and assume $\cala$ has enough 
projectives. Let $X$ be a projective object in $\calb$ that has an 
$\Omega$-resolution. Then there is a first quadrant spectral sequence 
$E^r_{*,*}$ in $\calb$ such that 
	\[ E^2_{i,j} \cong 
	(L_i\theta_*)\bigl(\theta^*(H_j^\Omega(\cala,\calb;X))\bigr) 
	\qquad\textup{and}\qquad
	E^\infty_{i,j}\cong \begin{cases} 
	X & \textup{if $(i,j)=(0,0)$} \\
	0 & \textup{if $(i,j)\ne(0,0)$.}
	\end{cases} \]
\end{Prop}

\begin{proof} Let $(P_*, \2_*)$ be an $\Omega$-resolution of $X$. Let 
$\bigl\{Q_{ij}\bigr\}_{i,j\ge0}$ be a \emph{proper projective resolution} 
of $(P_*,\2_*)$; i.e., a double complex of projective objects in $\cala$, 
where for each $j$, the sequences 
\begin{enumerate}[(i) ]

\item $0\fromm P_j\fromm Q_{0,j}\fromm Q_{1,j}\fromm\cdots$,

\item $0\fromm H_j(P_*)\fromm H_j(Q_{0,*})\fromm H_j(Q_{1,*})
\fromm\cdots$, and 

\item $0\fromm Z_j(P_*)\fromm Z_j(Q_{0,*})\fromm Z_j(Q_{1,*})
\fromm\cdots$

\end{enumerate}
are all projective resolutions (see \cite[Definition 3.6.1]{Benson2}). Proper 
projective resolutions exist by \cite[Proposition XII.11.6]{MacLane} (see 
also, \cite[Lemma 3.6.2]{Benson2}). 

Consider the two spectral sequences associated to the double complex 
$\theta_*(Q_{*,*})$. Since each row $Q_{*,j}$ is a projective 
resolution of the projective object $P_j$, $\theta_*(Q_{*,j})$ is a resolution 
of $\theta_*(P_j)$, and thus $\bar{E}^1_{0,j}\cong\theta_*(P_j)$, while 
$\bar{E}^1_{i,j}=0$ if $i\ge1$. Since $P_*$ is an $\Omega$ resolution of $X$, 
we now obtain $\bar{E}^2_{0,0}\cong X$, while $\bar{E}^2_{i,j}=0$ if 
$(i,j)\ne(0,0)$. 

Now consider the other spectral sequence $E^r_{i,j}$, where we first take 
homology of the columns. By (ii), $H_j(Q_{i,*})$ and $Z_j(Q_{i,*})$ are 
projective for each $i,j\ge0$, so $B_j(Q_{i,*})$ is also projective, and 
all sequences involved in the homology of $Q_{i,*}$ split. In other words, 
	\[ E^1_{i,j} = H_j(\theta_*(Q_{i,*}))\cong \theta_*(H_j(Q_{i,*})) \]
for all $i,j\ge0$. By (ii) again, the $j$-th row in the $E^1$-page is 
obtained by applying $\theta_*$ to a projective resolution of $H_j(P_*)$, and 
so 
	\[ E^2_{i,j}\cong(L_i\theta_*)(H_j(P_*,\2_*))\cong 
	(L_i\theta_*)\bigl(\theta^*(H_j^\Omega(\cala,\calb; X))\bigr). \]

Since $\bar{E}^\infty_{i,j}=0$ for all $(i,j)\ne(0,0)$, the two spectral 
sequences have isomorphic $E^\infty$-pages, and this proves the proposition.
\end{proof}

We finish the section with the following observation.

\begin{Rmk} \label{Q10}
If $\Osys{\cala}{\calb}$ and $\Osys[\eta]{\calb}{\calc}$ are two 
$\Omega$-systems, then their composite 
$\Osystem{\cala}{\calc}{\eta_*\theta_*}{\theta^*\eta^*}$ is easily seen to 
be an $\Omega$-system. In other words, there is a category whose objects 
are the small abelian categories and whose morphisms are isomorphism 
classes of $\Omega$-systems. One obvious question is whether there is a 
natural way to construct $\Omega$-resolutions for the composite 
$\Omega$-system from $\Omega$-resolutions for the two factors, and if so, 
what connection there is (if any) between the homology groups of these three 
complexes.
\end{Rmk}

\section{The existence of $\Omega$-resolutions}
\label{s:existence}

We saw in the last section that $\Omega$-resolutions, when they exist, are 
unique up to chain homotopy. The question of when they do exist is more 
complicated, and in this section, we give some necessary conditions and 
some sufficient conditions for that to happen. When doing this, the 
following more general form of Definition \ref{d:Omega-resolution} 
will be needed. 

\begin{Defi}\label{d:part.res.}
Let $\abth$ be an $\Omega$-system. For $X\in\scrp(\calb)$ and 
$1\le n\le\infty$, an \emph{$\Omega_n$-resolution} of $X$ is a 
chain complex 
	\beqq 
	\bfR_n = \begin{cases} 
	\bigl( P_n\Right3{\2_n} \cdots \Right3{\2_3} P_2 \Right3{\2_2} 
	P_1 \Right3{\2_1} P_0 
	\Right3{\gee} \theta^*(X) \Right2{} 0 \bigr) & \textup{if $n<\infty$} \\
	\bigl( \cdots \Right3{\2_3} P_2 \Right3{\2_2} 
	P_1 \Right3{\2_1} P_0 
	\Right3{\gee} \theta^*(X) \Right2{} 0 \bigr) & \textup{if $n=\infty$} \\
	\end{cases}
	\label{e3:sqres} \eeqq
in $\cala$ such that 
\begin{enumerate}[\rm($\Omega_n$-1) ]
\item $P_i\in\scrp(\cala)$ for all $0\le i\le n$ (for all $i\ge0$ if 
$n=\infty$); 
\label{sqres-1n}

\item $\theta_*(\bfR_n)$ is exact; \label{sqres-2n}

\item $H_i(P_*,\2_*)$ is isomorphic to an object in $\theta^*(\calb)$ for 
each $0\le i<n$, and $\gee$ induces an isomorphism 
$H_0(P_*,\2_*)\cong\theta^*(X)$; and \label{sqres-3n}

\item if $n<\infty$, the inclusion $\Im(\2_n)\le P_{n-1}$ induces a 
monomorphism $\theta_*(\Im(\2_n))\too\theta_*(P_{n-1})$. \label{sqres-4n}

\end{enumerate} 
\end{Defi}

In particular, an $\Omega_\infty$-resolution is the same as an 
$\Omega$-resolution (Definition \ref{d:Omega-resolution}).

\begin{Lem} \label{Omega_n4}
Condition \textup{\omres[n]{\ref{sqres-4n}}} can be replaced by the 
following equivalent condition:
\begin{enumerate}[($\Omega_n$-1) ]
	
\item[$(\Omega_n$-$4')$ ] If $n<\infty$, the sequence 
	$\theta_*(\Ker(\2_n)) \Right4{\theta_*(\incl)} \theta_*(P_n) 
	\Right4{\theta_*(\2_n)} \theta_*(P_{n-1}) $
is exact.
\end{enumerate}
\end{Lem}

\begin{proof} The sequence $\theta_*(\Ker(\2_n)) \too \theta_*(P_n) \too 
\theta_*(\Im(\2_n))\to0$ is exact since $\theta_*$ is right exact by Lemma 
\ref{l:Omega-props}(a). Hence the 
sequence in \omres[n]{4$'$} is exact if and only if the inclusion of 
$\Im(\2_n)$ in $P_{n-1}$ induces a monomorphism $\theta_*(\Im(\2_n))\too 
\theta_*(P_{n-1})$.
\end{proof}

\begin{Lem} \label{l:part.res.} 
Fix an $\Omega$-system $\abth$, and a projective object $X\in\scrp(\calb)$. Let 
$P_m\too\cdots\too P_0\too X\to0$ be an $\Omega_m$-resolution of $X$ for 
some $1<m\le\infty$. Then for each $1\le n<m$, the truncation 
$P_n\too\cdots\too P_0\too X\to0$ is an $\Omega_n$-resolution of $X$.
\end{Lem}

\begin{proof} Conditions 
\omres[n]{\ref{sqres-1n}}--\omres[n]{\ref{sqres-3n}} in Definition 
\ref{d:part.res.} follow immediately from 
\omres[m]{\ref{sqres-1n}}--\omres[m]{\ref{sqres-3n}}, so we need only prove 
that \omres[n]{\ref{sqres-4n}} holds. Consider the following commutative 
diagram:
	\beqq \vcenter{\xymatrix@C=40pt@R=25pt{
	\theta_*(P_{n+1}) \ar[r]^-{\theta_*(\2_{n+1})} \ar@/_1pc/[drr]_{0} 
	& \theta_*(P_n) \ar[rr]^-{\theta_*(\2_n)} 
	\ar@{->>}[dr]_(0.4){\theta_*(\2_n^*)} && \theta_*(P_{n-1}) \\ 
	&& \theta_*(\Im(\2_n)) \ar[ur]_{\theta_*(\incl)}
	}} \label{e:factors} \eeqq
where $\2_n^*\:P_n\too\Im(\2_n)$ is the corestriction of $\2_n$. The row in 
\eqref{e:factors} is exact by \omres[m]{\ref{sqres-2n}} and since $m>n$, and 
$\theta_*(\2_n^*)$ is an epimorphism since $\theta_*$ is right exact. Hence 
$\Ker(\theta_*(\2_n^*))=\Ker(\theta_*(\2_n))$ and 
$\theta_*(\incl)$ is a monomorphism, and \omres[n]{\ref{sqres-4n}} holds.
\end{proof}

\begin{Defi} \label{d:closedunder}
When $\abth$ is an $\Omega$-system, we say that \emph{$\theta^*(\calb)$ is 
closed under subobjects in $\cala$} if for each monomorphism $A_1\too A_2$ 
in $\cala$, $A_1$ is isomorphic to an object of $\theta^*(\calb)$ if $A_2$ 
is. Similarly, we say that \emph{$\theta^*(\calb)$ is closed under 
extensions in $\cala$} if for each short exact sequence $0\to M'\too 
M\too M''\to0$ in $\cala$, $M$ is isomorphic to an object in 
$\theta^*(\calb)$ if $M'$ and $M''$ are isomorphic to objects in 
$\theta^*(\calb)$.
\end{Defi}

We will show that for each $\Omega$-system $\abth$ in which $\theta^*(\calb)$ is 
closed under subobjects and extensions, all projectives in $\calb$ have 
$\Omega$-resolutions (Proposition \ref{exists-res}).

\begin{Lem} \label{l:closed}
Let $\abth$ be an $\Omega$-system, where $\cala$ has enough projectives.
\begin{enuma} 

\item The following are equivalent:\smallskip

\begin{enumerate}[\rm({a}.i) ]

\item $\theta^*(\calb)$ is closed under subobjects in $\cala$.

\item For each $M$ in $\cala$, the unit morphism 
$\aaa_M\:M\too\theta^*\theta_*(M)$ is an epimorphism.

\end{enumerate}

\item If either condition \textup{(a.i)} or \textup{(a.ii)} holds, then the 
following two  conditions are equivalent:\smallskip

\begin{enumerate}[\rm(b.i) ]

\item $\theta^*(\calb)$ is closed under extensions in $\cala$.

\item For each $N$ in $\calb$, $(L_1\theta_*)(\theta^*(N))=0$.
\end{enumerate}

\item If $0\too M'\too M\too M''\too0$ is an extension in $\cala$, 
where $M'$ and $M''$ are in $\theta^*(\calb)$ but $M$ is not isomorphic to 
an object in $\theta^*(\calb)$, then $(L_1\theta_*)(M'')\ne0$. 

\end{enuma}
\end{Lem}

\begin{proof} \textbf{(a.i$\implies$a.ii) } Fix an object $M$ in $\cala$, 
and consider the unit morphism $\aaa_M\:M\too\theta^*\theta_*(M)$. Since 
$\theta^*(\calb)$ is closed under subobjects, $\Im(\aaa_M)\cong\theta^*(B)$ 
for some $B$ in $\calb$. Since $\theta^*(\calb)$ is a full subcategory of 
$\cala$ by Lemma \ref{l:adj-pair}, each morphism in 
$\Mor_\cala(\theta^*(B),\theta^*\theta_*(M))$ lies in $\theta^*(\calb)$. 
Thus $\aaa_M$ factors as a composite
	\[ \aaa_M \: M \Right4{g} \theta^*(B) \Right4{\theta^*(\psi)} 
	\theta^*\theta_*(M) \]
for some $\psi\in\Mor_\calb(B,\theta_*(M))$, 
where $g$ is surjective and $\theta^*(\psi)$ is injective.

Let $\gamma\in\Mor_\calb(\theta_*(M),B)$ be the morphism adjoint to $g$. 
Then $\psi\circ\gamma=\Id_{\theta_*(M)}$ since $\aaa_M$ is adjoint to the 
identity, and thus $\psi$ is surjective. So $\theta^*(\psi)$ is also 
surjective by \OP3, and hence $\aaa_M$ is surjective.

\textbf{(a.ii$\implies$a.i) } Now assume that 
$M\Right2{\aaa_M}\theta^*\theta_*(M)$ is an epimorphism for each $M$ in 
$\cala$. Let $M_1\Right2{f}M_2$ be a monomorphism in $\cala$, where 
$M_2$ is in $\theta^*(\calb)$ and hence $\aaa_{M_2}$ is an isomorphism. 
From the commutative square
	\beq \vcenter{\xymatrix@C=40pt@R=25pt{
	M_1 \ar[r]^{f} \ar[d]_{\aaa_{M_1}} & M_2 
	\ar[d]_{\aaa_{M_2}}^{\cong} \\
	\theta^*\theta_*(M_1) \ar[r]^{\theta^*\theta_*(f)} & \theta^*\theta_*(M_2) 
	\rlap{\,,}
	}} \eeq
we see that $\theta^*\theta_*(f)\circ\aaa_{M_1}=\aaa_{M_2}\circ f$ is a 
monomorphism, and hence that $\aaa_{M_1}$ is a monomorphism. Since 
$\aaa_{M_1}$ is also an epimorphism, we have 
$M_1\cong\theta^*\theta_*(M_1)\in\Ob(\theta^*(\calb))$.

\smallskip
	
\textbf{(c) } Let $0\to M'\xto{\alpha}M\xto{\beta}M''\to0$ be a short 
exact sequence in $\cala$, where $M',M''\in\Ob(\theta^*(\calb))$ and 
$M$ is not isomorphic to any object in $\theta^*(\calb)$. Consider the 
following commutative diagram with exact rows:
	\beq \vcenter{\xymatrix@C=40pt@R=25pt{
	0 \ar[r] & M' \ar[r]^{\alpha} \ar[d]_{\aaa_{M'}}^{\cong} 
	& M \ar[r]^{\beta} \ar[d]_{\aaa_{M}} & M'' \ar[r] 
	\ar[d]_{\aaa_{M''}}^{\cong} & 0 \\
	& \theta^*\theta_*(M') \ar[r]^{\theta^*\theta_*(\alpha)} & 
	\theta^*\theta_*(M) \ar[r]^{\theta^*\theta_*(\beta)} 
	& \theta^*\theta_*(M'') \ar[r] & 0
	\rlap{\,.}
	}} \eeq
Here, $\aaa_{M'}$ and $\aaa_{M''}$ are isomorphisms since $M'$ and 
$M''$ are in $\theta^*(\calb)$, while $\aaa_{M}$ is not an isomorphism since 
$M$ is not isomorphic to any object in $\theta^*(\calb)$. Thus 
$\theta^*\theta_*(\alpha)$ is not injective in $\cala$, so 
$\theta_*(\alpha)$ is not injective in $\calb$ (Lemma 
\ref{l:Omega-props}(c)), and $(L_1\theta_*)(M'')\ne0$.

\smallskip

\textbf{(b) } The implication (b.ii$\implies$b.i) follows immediately from 
(c), and it remains to prove the converse. So assume that (a.ii) holds, and 
that $\theta^*(\calb)$ is closed under extensions in $\cala$. Fix $M$ in 
$\theta^*(\calb)$, and let 
	\[ 0 \Right2{} K \Right4{\alpha} P \Right4{\beta} M \Right2{} 0 \]
be a short exact sequence in $\cala$ where $P$ is projective. Set 
$K_0=\Ker(\aaa_K)$, and consider the following commutative diagram:
	\beqq \vcenter{\xymatrix@C=35pt@R=25pt{
	& \qquad\qquad\quad0 \ar[r] & K/K_0 \ar[r]^{\5\alpha} 
	\ar[d]_{\5\aaa_K}^{\cong} & P/\alpha(K_0) \ar[r]^{\5\beta} 
	\ar@{->>}[d]_{\5\aaa_P} & M \ar[r] \ar[d]_{\aaa_M}^{\cong} & 0 \\
	0 \ar[r] & \theta^*\bigl((L_1\theta_*)(M)\bigr) \ar[r]
	& \theta^*\theta_*(K) \ar[r]^{\theta^*\theta_*(\alpha)} & 
	\theta^*\theta_*(P) \ar[r]^{\theta^*\theta_*(\beta)} 
	& \theta^*\theta_*(M) \ar[r] & 0
	\rlap{\,.}
	}} \label{e:2.5(b)} \eeqq
Here, $\aaa_M$ is an isomorphism since $M$ is in $\theta^*(\calb)$, 
$\5\aaa_K$ and $\5\aaa_P$ are epimorphisms since $\aaa_K$ and $\aaa_P$ are 
surjective by (a.ii), and $\5\aaa_K$ is injective by construction. The top 
row is exact by construction, and the bottom row since $(L_1\theta_*)(P)=0$ 
($P$ is projective) and $\theta^*$ is exact.

Now, $K/K_0\cong\theta^*\theta_*(K)$ and $M$ are both isomorphic to 
objects of $\theta^*(\calb)$, and the same holds for $P/\alpha(K_0)$ since 
$\theta^*(\calb)$ is closed under extensions. Thus there is an object $N$ 
in $\calb$, a surjective morphism $f\:P\too\theta^*(N)$ with kernel 
$\alpha(K_0)$, and a morphism $\nu\:\theta^*(N)\too\theta^*\theta_*(P)$ 
such that $\nu\circ f=\aaa_P$. Since $\theta^*(\calb)$ is a full 
subcategory of $\cala$ (Lemma \ref{l:adj-pair}), $\nu=\theta^*(\chi)$ for 
some $\chi\in\Mor_\calb(N,\theta_*(P))$. 

Let $\varphi\in\Mor_\calb(\theta_*(P),N)$ be adjoint to $f$. Then 
	\[ \vcenter{\xymatrix@C=25pt@R=25pt{
	& P \ar[dl]_{\aaa_P} \ar[dr]^{f} \\
	\theta^*\theta_*(P) \ar@/_/[rr]_{\theta^*(\varphi)} && \theta^*(N) 
	\ar@/_/[ll]_{\nu=\theta^*(\chi)}
	}}
	\qquad\textup{is adjoint to}\qquad
	\vcenter{\xymatrix@C=25pt@R=25pt{
	& \theta_*(P) \ar[dl]_{\Id_{\theta_*(P)}} \ar[dr]^{\varphi} \\
	\theta_*(P) \ar@/_/[rr]_{\varphi} && N \ar@/_/[ll]_{\chi}
	\rlap{\,,}
	}} \]
so $\theta^*(\varphi)\circ\aaa_P=f$ and $\nu\circ f=\aaa_P$. Since $f$ 
and $\aaa_P$ are both surjective, $\nu$ and $\theta^*(\varphi)$ are 
isomorphisms (and inverses to each other). So in diagram \eqref{e:2.5(b)}, 
$\5\aaa_P$ is an isomorphism, $\theta^*\theta_*(\alpha)$ is injective, and 
thus $(L_1\theta_*)(M)=0$. 
\end{proof}

The next proposition provides one tool for showing that 
$\Omega$-resolutions do \emph{not} exist in certain cases. Recall that 
by Lemma \ref{l:part.res.}, if a projective object has no 
$\Omega_1$-resolution, then it has no $\Omega$-resolution.

\begin{Prop} \label{exists-Omega1}
For each $\Omega$-system $\abth$ for which $\cala$ has enough projectives, 
and each $X\in\scrp(\calb)$, there is an $\Omega_1$-resolution of $X$ if 
and only if $(L_1\theta_*)(\theta^*(X))=0$.
\end{Prop}

\begin{proof} Assume that $(L_1\theta_*)(\theta^*(X))=0$. Let 
$\bfR_1=\bigl(P_1\xto{\2_1}P_0\xto{\gee}\theta^*(X)\to0\bigr)$ be an exact 
sequence in $\cala$, where $P_0, P_1\in \scrp(\cala)$. Then the sequence 
	\[0\too\Im(\2_1)\too P_0\too \theta^*(X)\too 0\]
is short exact, and since $(L_1\theta_*)(\theta^*(X))=0$, the induced 
morphism $\theta_*(\Im(\2_1))\too\theta_*(P_0)$ is a monomorphism. So 
$\bfR_1$ is an $\Omega_1$-resolution of $X$, where \omres[1]{\ref{sqres-2n}} 
holds since $\theta_*$ is right exact. 

Conversely, if $\bfR_1=\bigl(P_1\xto{\2_1}P_0\xto{\gee}\theta^*(X)\to0\bigr)$ 
is an $\Omega_1$-resolution of $X$, then $\bfR_1$ is exact 
by \omres[n]{\ref{sqres-3n}}. Since $P_0\in\scrp(\cala)$, the
sequence 
	\[ 0 \Right2{} (L_1\theta_*)(\theta^*(X)) \Right4{} \theta_*(\Im(\2_1)) 
	\Right4{\theta_*(\incl)} \theta_*(P_0) \Right4{\theta_*(\gee)} 
	\theta_*(\theta^*(X)) \Right2{} 0 \]
is exact. Since $\theta_*(\incl)$ is a monomorphism by condition 
\omres[1]{\ref{sqres-4n}}, $(L_1\theta_*)(\theta^*(X))=0$.
\end{proof}

The following lemma gives conditions for extending an $\Omega_n$-resolution 
to an $\Omega_{n+1}$-resolution.

\begin{Lem} \label{extend-res}
Fix an $\Omega$-system $\abth$, where $\cala$ has enough projectives, and 
let $X$ be a projective object in $\calb$. Let  
$\bfR_n=\bigl(P_n\xto{\2_n}P_{n-1}\too\cdots\too 
P_0\xto{\gee}\theta^*(X)\to0\bigr)$ be an $\Omega_n$-resolution of $X$, for 
some $1\le n<\infty$.
\begin{enuma} 

\item The natural morphism 
$f_0\:\theta_*(\Ker(\2_n))\Right2{}\Ker(\theta_*(\2_n))$ is a split 
epimorphism. 

\item If $P_{n+1}\in\scrp(\cala)$, and 
$\2_{n+1}\in\Mor_{\cala}(P_{n+1},P_n)$ are such that $\2_n\circ\2_{n+1}=0$, 
then the complex $\bfR_{n+1}=\bigl(P_{n+1}\xto{\2_{n+1}}P_{n}\too\cdots 
\xto{\gee}\theta^*(X)\to0\bigr)$ is an $\Omega_{n+1}$-resolution of $X$ if 
and only if \smallskip

\begin{enumerate}[\rm({b}.i) ]
\item $\Ker(\2_n)/\Im(\2_{n+1})$ is in $\calb$; and

\item the composite $\theta_*(\Im(\2_{n+1}))\Right4{\theta_*(\incl)} 
\theta_*(\Ker(\2_n)) \Right2{f_0}\Ker(\theta_*(\2_n))$ is an isomorphism.

\end{enumerate}

\item The resolution $\bfR_n$ extends to an $\Omega_{n+1}$-resolution 
if and only if for some splitting $s$ of $f_0$, the composite 
	\[ \aaa_{\Ker(\2_n)}^{[s]} \: 
	\Ker(\2_n)\Right6{\aaa_{\Ker(\2_n)}}\theta^*\theta_*(\Ker(\2_n))
	\Right4{\theta^*(\chi^{[s]})} 
	\theta^*\bigl(\theta_*(\Ker(\2_n))/\Im(s)\bigr) \] 
(where $\chi^{[s]}$ is the natural projection) and the induced map 
	\[ (L_1\theta_*)(\aaa_{\Ker(\2_n)}^{[s]}) \: 
	(L_1\theta_*)\bigl(\Ker(\2_n)\bigr) \Right5{} 
	(L_1\theta_*)\bigl(\theta^*(\theta_*(\Ker(\2_n))/\Im(s))\bigr) \]
are both epimorphisms. 

\item If $\bfR_n$ does extend to $\bfR_{n+1}$ as in (b), 
then for some splitting $s$ of $f_0$, 
$\Im(\2_{n+1})=\Ker(\aaa_{\Ker(\2_n)}^{[s]})$, and 
	\beqq \theta_*\bigl(H_n(P_*,\2_*)\bigr) \cong 
	\theta_*(\Ker(\2_n))/\Im(s)\cong\Ker(f_0) \cong 
	(L_1\theta_*)(\Im(\2_n)) \,. \label{e:diag2a} \eeqq

\end{enuma}
\end{Lem}

\begin{proof} \textbf{(a) } Since $\theta_*$ is right exact, we have the following 
commutative diagram in $\calb$ 
	\beqq \vcenter{\xymatrix@C=40pt@R=25pt{
	& \theta_*(\Ker(\2_n)) \ar[r] \ar[d]_{f_0} & \theta_*(P_n) \ar[r] \ar@{=}[d] 
	& \theta_*(\Im(\2_n)) \ar[r] \ar[d]_{f_1}^{\cong} & 0 \\
	0 \ar[r] & \Ker(\theta_*(\2_n)) \ar[r] & \theta_*(P_n) \ar[r] & 
	\Im(\theta_*(\2_n)) \ar[r] & 0 
	}} \label{e:diag3a} \eeqq
with exact rows. By condition \omres[n]{\ref{sqres-4n}}, $f_1$ is a 
monomorphism (hence an isomorphism), and so $f_0$ is an epimorphism. Also, 
$\Ker(\theta_*(\2_n))\in\scrp(\calb)$ since the sequence $\theta_*(P_*)\too 
\theta^*(X)\to0$ is an exact sequence of projective objects, and hence 
$f_0$ splits. 

\smallskip

\noindent\textbf{(b) } Assume that 
$\bfR_{n+1}=\bigl(P_{n+1}\xto{\2_{n+1}}P_n\too\cdots\bigr)$ is an 
$\Omega_{n+1}$-resolution of $X$, and set $J=\Im(\2_{n+1})\le\Ker(\2_n)$. 
Consider the following commutative diagram
	\beq \vcenter{\xymatrix@C=30pt@R=25pt{
	\theta_*(P_{n+1}) \ar[r]^-{\theta_*(\2^*_{n+1})} \ar@{->>}[dr]^{f_3} & 
	\theta_*(J) \ar[d]_(0.45){f_4} \ar[r]^-{\theta_*(\incl)} 
	& \theta_*(\Ker(\2_n)) \ar[d]_{f_0} \\
	& \Im(\theta_*(\2_{n+1})) \ar@{=}[r] & \Ker(\theta_*(\2_n))
	\rlap{\,,}
	}} \eeq 
where $\2_{n+1}^*\:P_{n+1}\too J$ is the corestriction of $\2_{n+1}$ and is 
surjective, and $f_3$ is the corestriction of 
$\theta_*(\2_{n+1})$. By condition \omres[n+1]{\ref{sqres-4n}} on $\bfR_{n+1}$, 
the morphism $\theta_*(J)\too\theta_*(P_n)$ is a monomorphism. Hence $f_4$ is a 
monomorphism, and is an isomorphism since $f_3$ is an epimorphism. This 
proves (b.ii), and (b.i) follows from condition 
\omres[n+1]{\ref{sqres-3n}}.

Conversely, assume that (b.i) and (b.ii) hold. In particular, $\bfR_{n+1}$ 
satisfies \omres[n+1]{\ref{sqres-3n}}, and it satisfies 
\omres[n+1]{\ref{sqres-1n}} ($P_{n+1}$ is projective) by assumption. 
Condition \omres[n+1]{\ref{sqres-4n}} (that $\theta_*(\Im(\2_{n+1}))$ injects 
into $\theta_*(P_n)$) follows from (b.ii). 

It remains to prove \omres[n+1]{\ref{sqres-2n}}; i.e., the exactness of 
$\theta_*(\bfR_{n+1})$. Since $\theta_*(\bfR_n)$ is exact, we need only show that 
$\Im(\theta_*(\2_{n+1}))=\Ker(\theta_*(\2_n))$. Consider the following diagram:
	\beq \vcenter{\xymatrix@C=55pt@R=30pt{ 
	P_{n+1} \ar@{->>}[r]^-{\2^*_{n+1}} \ar[d]^{\aaa_{P_{n+1}}} & \Im(\2_{n+1}) 
	\ar[r]^-{i} \ar[d]^{\aaa_{\Im(\2_{n+1})}} & P_n 
	\ar[d]^{\aaa_{P_{n}}} \\ 
	\theta^*\theta_*(P_{n+1}) \ar@{->>}[r]^-{\theta^*\theta_*(\2^*_{n+1})} 
	\ar@/_1.8pc/[rr]^{\theta^*\theta_*(\2_{n+1})} & 
	\theta^*\theta_*(\Im(\2_{n+1})) \ar[r]^-{\theta^*\theta_*(i)} 
	& \theta^*\theta_*(P_n) \\
	}} \eeq
where $\2^*_{n+1}$ is surjective by definition and 
$\theta^*\theta_*(\2^*_{n+1})$ is surjective since $\theta^*\theta_*$ is 
right exact. Hence 
$\Im(\theta^*\theta_*(\2_{n+1}))=\Im(\theta^*\theta_*(i))$, and so 
$\Im(\theta_*(\2_{n+1}))=\Im(\theta_*(i))$ by Lemma \ref{l:Omega-props}(c). 
Finally, $\Im(\theta_*(i))=\Ker(\theta_*(\2_n))$ by the following diagram
	\beq \vcenter{\xymatrix@C=45pt@R=30pt{ 
	\theta_*(\Im(\2_{n+1})) \ar[r]^-{\theta_*(i')} \ar[drr]_{\theta_*(i)} & 
	\theta_*(\Ker(\2_n)) \ar[r]^{f_0} \ar[dr]^{\theta_*(i'')} & 
	\Ker(\theta_*(\2_n)) \ar[d]^{\incl} \\
	& & \theta_*(P_n) 
	}} \eeq
and since $f_0\circ\theta_*(i')$ is an isomorphism by (b.ii).

\smallskip

\noindent\textbf{(c,d) } Assume first that $\bfR_n$ does extend to an 
$\Omega_{n+1}$-resolution 
	\[ \bfR_{n+1}=\bigl(P_{n+1} \Right3{\2_{n+1}} P_n \Right2{} 
	\cdots\bigr)\,. \] 
Set $J=\Im(\2_{n+1})$. By (b.i) and (b.ii), $\Ker(\2_n)/J$ is 
isomorphic to an object in $\theta^*(\calb)$, and the composite 
	\[ \xymatrix@C=50pt{
	\theta_*(J) \ar[r]_-{\theta_*(i_2)} \ar@/^1.6pc/[rr]^-{\cong} 
	& \theta_*(\Ker(\2_n)) \ar[r]^{f_0} & \Ker(\theta_*(\2_n)) 
	\ar@{-->}@/^0.7pc/[l]^{s}
	} \] 
is an isomorphism. Set $s=\theta_*(i_2)\circ(f_0\circ\theta_*(i_2))^{-1}$: 
a splitting for $f_0$.

Consider the following commutative diagram:
	\beqq \vcenter{\xymatrix@C=40pt@R=25pt{
	0 \ar[r] & J \ar[r]^-{i_2} \ar[d]_{\aaa_J} & \Ker(\2_n) 
	\ar[r]^-{\pr_2} \ar[d]_{\aaa_{\Ker(\2_n)}} \ar[dr]^{\omega} & 
	\Ker(\2_n)/J \ar[r] \ar[d]_{f_2}^{\cong} & 0 \\
	0 \ar[r] & \theta^*\theta_*(J) \ar[r]^-{\theta^*\theta_*(i_2)} & 
	\theta^*\theta_*(\Ker(\2_n)) \ar[r]^-{\theta^*\theta_*(\pr_2)} & 
	\theta^*\theta_*(\Ker(\2_n)/J) \ar[r] & 0 \rlap{\,.}
	}} \label{e:diag4a} \eeqq
Here, $f_2=\aaa_{\Ker(\2_n)/J}$ is an isomorphism since $\Ker(\2_n)/J$ is 
isomorphic to an object in $\theta^*(\calb)$ by (b.i). The bottom row of 
\eqref{e:diag4a} is exact since $\theta^*\theta_*$ is right exact and 
$\theta_*(i_2)$ is a monomorphism by (b.ii) (and $\theta^*$ is left exact). 
Also, $\Im(\theta_*(i_2))=\Im(s)$, and hence 
	\begin{align*} 
	\Bigl( \Ker(\2_n) \Right6{\aaa_{\Ker(\2_n)}^{[s]}} 
	\theta^*\bigl(\theta_*(\Ker(\2_n))/\Im(s)\bigr) \Bigr) &\cong 
	\Bigl( \Ker(\2_n) \Right3{\omega} \theta^*\theta_*(\Ker(\2_n)/J) \Bigr) \\
	&\cong \Bigl( \Ker(\2_n) \Right3{\pr_2} \Ker(\2_n)/J \Bigr) \,.
	\end{align*}
Thus $\aaa_{\Ker(\2_n)}^{[s]}$ is an epimorphism, and 
$(L_1\theta_*)(\aaa_{\Ker(\2_n)}^{[s]})\cong(L_1\theta_*)(\pr_2)$ is also an 
epimorphism since $\theta_*(i_2)$ is injective. This also proves that 
$\Im(\2_{n+1})=J=\Ker(\aaa_{\Ker(\2_n)}^{[s]})$, and proves 
\eqref{e:diag2a} except for the isomorphism 
$\Ker(f_0)\cong(L_1\theta_*)(\Im(\2_n))$ which follows from \eqref{e:diag3a}. 
This finishes the proof of (d), and the proof of the ``only if'' part 
of (c).

Conversely, assume, for some splitting $s$ of $f_0$, that 
$\aaa_{\Ker(\2_n)}^{[s]}$ and $(L_1\theta_*)(\aaa_{\Ker(\2_n)}^{[s]})$ are 
both epimorphisms. Set $J=\aaa_{\Ker(\2_n)}^{-1}(\theta^*(\Im(s))) 
\le\Ker(\2_n)$, and  consider the following commutative diagram:
	\beqq \vcenter{\xymatrix@C=30pt@R=25pt{
	0 \ar[r] & J \ar[r]^-{i_2} \ar[d] & \Ker(\2_n) \ar[r]^-{\pr_2} 
	\ar[d]_{\aaa_{\Ker(\2_n)}} & 
	\Ker(\2_n)/J \ar[r] \ar[d]_{f_3}^{\cong} & 0 \\
	0 \ar[r] & \theta^*(\Im(s)) \ar[r] & \theta^*\theta_*(\Ker(\2_n)) 
	\ar[r]^-{\chi^{[s]}} & 
	\theta^*\bigl(\theta_*(\Ker(\2_n))/\Im(s)\bigr) \ar[r] & 0 \rlap{\,.}
	}} \label{e:diag5a} \eeqq
The left square in \eqref{e:diag5a} is a pullback square by definition of 
$J$, so $f_3$ is a monomorphism, and $f_3$ is an epimorphism since 
$\chi^{[s]}\circ\aaa_{\Ker(\2_n)}=\aaa_{\Ker(\2_n)}^{[s]}$ is an epimorphism. 

In particular, $\Ker(\2_n)/J$ is isomorphic to an object in 
$\theta^*(\calb)$, and (b.i) holds. 
Also, $f_2$ is an isomorphism in \eqref{e:diag4a}, and the bottom row in 
\eqref{e:diag4a} is exact since $(L_1\theta_*)(\aaa_{\Ker(\2_n)}^{[s]})$ is 
an epimorphism. Upon comparing \eqref{e:diag4a} and \eqref{e:diag5a}, we 
see that $\Im(s)=\theta_*(i_2)(\theta_*(J))$, and (b.ii) now follows since $\Im(s)$ 
is the image of a splitting of $f_0$. So $\bfR_n$ extends to an 
$\Omega_{n+1}$-resolution by (b).
\end{proof}

The next proposition is our most general result on the existence of 
$\Omega$-resolutions.

\begin{Prop} \label{exists-res}
Fix an $\Omega$-system $\abth$, where $\cala$ has enough projectives. 
Assume that $\theta^*(\calb)$ is closed under subobjects and extensions in 
$\cala$. Then each $X\in\scrp(\calb)$ admits an $\Omega$-resolution. 
Furthermore, for $n\ge0$, each $\Omega_n$-resolution of $X$ extends to an 
$\Omega$-resolution of $X$.
\end{Prop}

\begin{proof} By Lemma \ref{l:closed}(a,b) and since $\theta^*(\calb)$ is 
closed under subobjects and extensions in $\cala$, $\aaa_M$ is an epimorphism 
for each $M$ in $\cala$ and $(L_1\theta_*)(\theta^*(N))=0$ for each $N$ in 
$\calb$. In particular, $X\in\scrp(\calb)$ has an $\Omega_1$-resolution by 
Proposition \ref{exists-Omega1}.

Assume, for some $n\ge1$, that $\bfR_n=(P_i,\2_i)_{i\le n}$ is an 
$\Omega_n$-resolution of $\theta^*(X)$. By Lemma \ref{extend-res}(a), the morphism 
$f_0\:\theta_*(\Ker(\2_n))\too\Ker(\theta_*(\2_n))$ is a split epimorphism. 
Let $s$ be a splitting of $f_0$, and let $\chi^{[s]}$ be the natural 
epimorphism from $\theta_*(\Ker(\2_n))$ to $\theta_*(\Ker(\2_n))/\Im(s)$. 
Then the morphisms 
	\[ \aaa_{\Ker(\2_n)}^{[s]} \: \Ker(\2_n) \Right6{\aaa_{\Ker(\2_n)}} 
	\theta^*\theta_*(\Ker(\2_n)) \Right6{\theta^*(\chi^{[s]})} 
	\theta^*\bigl(\theta_*(\Ker(\2_n))/\Im(s)\bigr) \]
are epimorphisms: the first since $\aaa_M$ is an epimorphism for all $M$ 
and the second since $\theta^*$ is exact. Also, 
$(L_1\theta_*)(\aaa_{\Ker(\2_n)}^{[s]})$ is an epimorphism since 
$(L_1\theta_*)(\theta^*(N))=0$ for all $N$. So by Lemma 
\ref{extend-res}(c), $\bfR_n$ extends to an $\Omega_{n+1}$-resolution 
$\bfR_{n+1}$. Since this argument applies for all $n\ge1$, it follows that 
$\bfR_n$ extends to an $\Omega$-resolution of $\theta^*(X)$. In particular, 
$\theta^*(X)$ has $\Omega$-resolutions since it has $\Omega_1$-resolutions.
\end{proof}

Under the assumptions of Example \ref{ex1}, we can now describe exactly 
under what conditions there are $\Omega$-resolutions. Recall, for a 
commutative ring $\RR$, that a group $G$ is $\RR$-perfect if 
$H_1(G;\RR)=0$. 

\begin{Ex} \label{ex1a}
Fix a commutative ring $\RR$ and a surjective homomorphism 
$\theta\:G\too\pi$ of groups. Let $\cpth[G]$ be the $\Omega$-system of 
Example \ref{ex1}. Then $\theta^*(\RR\pi\mod)$ is closed under subobjects. 
\begin{enuma} 

\item If $\Ker(\theta)$ is $\RR$-perfect, then $\theta^*(\RR\pi\mod)$ is 
closed under extensions in $\RR G\mod$. So by Proposition \ref{exists-res}, 
$\Omega$-resolutions exist of all projective objects in $\RR\pi\mod$.

\item If $\Ker(\theta)$ is not $\RR$-perfect, then $\theta^*(\RR\pi\mod)$ is not 
closed under extensions, and for each nonzero object $X$ in $\RR\pi\mod$ 
that is free as an $\RR$-module, $(L_1\theta_*)(\theta^*(X))\ne0$. So by 
Proposition \ref{exists-Omega1}, no nonzero projective object in 
$\RR\pi\mod$ that is free as an $\RR$-module has an $\Omega$-resolution.

\end{enuma}
\end{Ex}

\begin{proof} Set $K=\Ker(\theta)$, and note that an $\RR G$-module $M$ is 
isomorphic to an object in $\theta^*(\RR\pi\mod)$ if and only if $K$ acts 
trivially on $M$. Thus $\theta^*(\RR\pi\mod)$ is closed under subobjects. 
If $H_1(K;\RR)=0$, then $\theta^*(\RR\pi\mod)$ is closed under extensions 
by Lemma \ref{l:p-perf}, and the existence of $\Omega$-resolutions follows 
from Proposition \ref{exists-res}.

If $H_1(K;\RR)\ne0$, then for each nonzero object $X$ in $\RR\pi\mod$ that 
is free as an $\RR$-module, $(L_1\theta_*)(\theta^*(X))\cong 
H_1(K;\theta^*(X))\ne0$ since $H_1(K;\RR)\ne0$ and $K$ acts trivially on 
$\theta^*(X)$. Thus $\theta^*(\RR\pi\mod)$ is not closed under extensions 
in $\RR G\mod$ by Lemma \ref{l:closed}(b). If in addition, $X$ is 
projective in $\RR\pi\mod$, Proposition \ref{exists-Omega1} implies that it 
has no $\Omega$-resolution. 
\end{proof}

Note that the ``squeezed resolutions'' defined and studied by Benson 
\cite{B} are $\Omega$-resolutions in the context of Example 
\ref{ex1a}(a), when $G$ is a finite group and $K=O^p(G)$. 

\begin{Rmk} \label{Q1-2}
Proposition \ref{exists-res} gives some general conditions for the 
existence of $\Omega$-resolu\-tions: conditions which are satisfied by the 
$\Omega$-systems of Example \ref{ex1} (as just seen), and also by the much 
larger family of examples to be described in Proposition \ref{ex3}(b). 
However, they do not hold for the family of examples constructed in 
Proposition \ref{ex4}(a), even though $\Omega$-resolutions are shown to exist 
in those cases (at least for certain projective objects) in Proposition 
\ref{C*Enu}. This suggests that there should 
be a more general existence result that covers all of these cases.

In fact, there are two questions of this type that one can ask. First, of 
course, we want to find conditions as general as possible on an 
$\Omega$-system $\Osys\cala\calb$ that imply the existence of 
$\Omega$-resolutions of all projectives in $\calb$. But we will see in 
Example \ref{ex3a} that there are $\Omega$-systems for 
which some nonzero projectives have $\Omega$-resolutions and others do not, 
and so we would also like to find more general conditions on a pair 
$\bigl(\Osys\cala\calb,X\bigr)$, for $X\in\scrp(\calb)$, that imply the 
existence of an $\Omega$-resolution of $X$. 

\end{Rmk}

\section{$\Omega$-systems of functor categories} 
\label{s:kC-mod}

We next look at a large family of examples of $\Omega$-systems and 
$\Omega$-resolutions involving functor categories; especially 
categories of $\RR\calc$-modules for a small category $\calc$ and a 
commutative ring $\RR$. At the end of the section, in Propositions 
\ref{ex3} and \ref{ex4}, we give two large families of examples of 
$\Omega$-systems where we can say fairly precisely in which cases 
$\Omega$-resolutions exist. 

We refer to \cite[\S\,II.6 and \S\,X.3]{MacLane-cat} for the definitions and 
properties of overcategories and left Kan extensions. As usual, when 
$\cala$ and $\calc$ are categories and $\calc$ is small, $\cala^\calc$ 
denotes the functor category whose objects are the functors 
$\calc\too\cala$, and whose morphisms are the natural transformations of 
functors.

To simplify the statement of the next proposition, we define a 
functor $\theta\:\calc\too\cald$ between small categories to be 
\emph{quasisurjective} if it is surjective on objects and 
$\cald$ is generated 
as a category by the image of $\theta$ together with inverses of 
isomorphisms in the image of $\theta$. As far as we know, this concept 
has not been defined earlier, and does seem to be designed for this 
very specialized situation.

\begin{Prop} \label{AC->AD}
Let $\cala$ be an abelian category with colimits, and let 
$\theta\:\calc\too\cald$ be a quasisurjective functor between small 
categories. For each $d$ in $\cald$, let $\cali(\theta\dn d)$ be the full 
subcategory of $\theta\dn d$ with objects $(c,\varphi)$ for 
$\varphi\in\Iso_\cald(\theta(c),d)$, and assume that all objects 
$(c,\Id_d)$ for $c\in\theta^{-1}(d)$ lie in the same connected component of 
$\cali(\theta\dn d)$. Let $\theta^*\:\cala^\cald\too\cala^\calc$ be 
composition with $\theta$, and let $\theta_*\:\cala^\calc\too\cala^\cald$ 
be left Kan extension along $\theta$. Then 
$\Osys{\cala^\calc}{\cala^\cald}$ is an $\Omega$-system. 
\end{Prop}

\begin{proof} Conditions \OP1 and \OP3 are clear. So the only difficulty is 
to show that condition \OP2 holds: that the counit 
$\bbb\:\theta_*\theta^*\too\Id_{\cala^\cald}$ associated to the adjunction 
is an isomorphism. 

Fix a functor $\alpha\:\cald\too\cala$ and an object $d$ in $\cald$, and 
let $\3{\alpha}(\3{d})\:\theta\dn d\too\cala$ be the 
constant functor sending all objects to $\alpha(d)$. Let 
$\alpha_d\:\theta\dn d\too\cala$ be the functor that sends an object 
$(c,\varphi)$, where $\varphi\in\Mor_\cald(\theta(c),d)$, to 
$(\theta^*\alpha)(c)=\alpha(\theta(c))$. Let 
$\alpha_*\:\alpha_d\too\3{\alpha}(\3{d})$ be the natural 
transformation of functors that sends $(c,\varphi)$ to 
$\alpha(\varphi)\in\Mor_\cala(\alpha_d(c,\varphi),\alpha(d))$. Then 
	\[ \bbb(\alpha)(d) = \colim_{\theta\dn d}(\alpha_*) \:
	(\theta_*\theta^*(\alpha))(d) =
	\colim_{\theta\dn d}(\alpha_d) \Right5{} \alpha(d) \,, \]
and we must show that this is an isomorphism (for all $\alpha$ and $d$). 

To see this, choose $\5d\in\Ob(\calc)$ such that $\theta(\5d)=d$. For each 
$(c,\varphi)$ in $\theta\dn d$, let $\iota_{(c,\varphi)}$ be the natural 
morphism from $\alpha_d(c,\varphi)=\alpha(\theta(c))$ to the colimit. Set 
	\[ \beta=\iota_{(\5d,\Id_d)}\:\alpha(d)=\alpha_d(\5d,\Id_d)
	\Right5{} \colim(\alpha_d) = (\theta_*\theta^*(\alpha))(d). \]
Then $\bbb(\alpha)(d)\circ\beta=\Id_{\alpha(d)}$, and it remains to show that 
$\beta\circ\bbb(\alpha)(d)$ is also the identity. This means showing, for each 
object $(c,\varphi)$ in $\theta\dn d$, that 
	$ \beta\circ\bbb(\alpha)(d)\circ\iota_{(c,\varphi)}=\iota_{(c,\varphi)} 
	\,. $ 
Since 
	\[ \bbb(\alpha)(d)\circ\iota_{(c,\varphi)} = \alpha(\varphi) \: 
	\alpha_d(c,\varphi) = \alpha(\theta(c)) \Right5{} 
	\alpha(d) = \alpha_d(\5d,\Id_d), \]
we are reduced to showing, for each $(c,\varphi)$, that 
	\beqq \iota_{(\5d,\Id_d)}\circ\alpha(\varphi)=\iota_{(c,\varphi)} 
	\,. \label{e:bai=i} \eeqq

We now claim the following:
\begin{enumi} 

\item Equation \eqref{e:bai=i} holds for $(\5d,\Id_d)$.

\item If there is $\chi\in\Mor_{\theta\dn d}((c,\varphi),(c',\varphi'))$, then 

\begin{enumerate}[({ii.}1)] \medskip
\item if \eqref{e:bai=i} holds for $(c',\varphi')$, then it also holds for 
$(c,\varphi)$; and 

\item if $\theta(\chi)$ is an isomorphism, then \eqref{e:bai=i} holds for 
$(c,\varphi)$ if and only if it holds for $(c',\varphi')$. 

\end{enumerate} \medskip

\item If $\theta(c)=\theta(c')$ and $\varphi\in\Mor_\cald(\theta(c),d)$, 
then \eqref{e:bai=i} holds for $(c,\varphi)$ if and only if it holds for 
$(c',\varphi)$.

\end{enumi}
Point (i) is clear. If $\chi\in\Mor_{\theta\dn 
d}((c,\varphi),(c',\varphi'))$, then 
$\iota_{(c',\varphi')}\circ\alpha(\theta(\chi))=\iota_{(c,\varphi)}$ by 
definition of colimits, and (ii) follows immediately from this. Point (iii) 
follows from (ii.2) and the assumption that $(c,\Id_{\theta(c)})$ and 
$(c',\Id_{\theta(c)})$ are in the same connected component of 
$\cali(\theta\dn\theta(c))$.

Now let $(c,\varphi)$ be arbitrary. Since $\theta$ is quasisurjective, for 
some $m\ge1$, there are objects $\theta(c)=d_0,d_1,\dots,d_m=d$ in $\cald$ 
and morphisms $\varphi_i\in\Mor_\cald(d_{i-1},d_i)$ for $1\le i\le m$ such 
that $\varphi=\varphi_m\circ\cdots\circ\varphi_2\circ\varphi_1$, and for 
each $i$, either $\varphi_i\in\theta(\Mor(\calc))$ or 
$\varphi_i\in\Iso(\cald)$ and $\varphi_i^{-1}\in\theta(\Mor(\calc))$. Since 
\eqref{e:bai=i} holds for $(\5d,\Id_d)$ by (i), it also holds for 
$(c,\Id_d)$ for all $c\in\theta^{-1}(d)=\theta^{-1}(d_m)$ by (iii). If 
$\varphi_m\in\theta(\Mor(\calc))$, then \eqref{e:bai=i} holds for 
$(c,\varphi_m)$ for some $c\in\theta^{-1}(d_{m-1})$ by (ii.1), while if 
$\varphi_m\in\Iso(\cald)$ and $\varphi_m^{-1}\in\theta(\Mor(\calc))$, then 
\eqref{e:bai=i} holds for $(c,\varphi_m)$ for some 
$c\in\theta^{-1}(d_{m-1})$ by (ii.2). In either case, \eqref{e:bai=i} holds 
for $(c,\varphi_m)$ for all $c\in\theta^{-1}(d_{m-1})$ by (iii). Upon 
continuing this argument, we see by downward induction that for each $1\le 
i\le m$, \eqref{e:bai=i} holds for $(c,\varphi_m\circ\cdots\circ\varphi_i)$ 
for all $c\in\theta^{-1}(d_{i-1})$. In particular, 
\eqref{e:bai=i} holds for $(c,\varphi)$ (the case $i=1$). 
\end{proof}

We now specialize to the case where $\cala=\RR\mod$: the category 
of modules over a commutative ring $\RR$.

\begin{Defi}\label{d:loctriv}
Let $\calc$ be a small category, and let $\RR$ be a commutative ring.
\begin{enuma} 

\item An $\RR\calc$-module is a \emph{covariant} functor $M\colon 
\calc\Right1{}\RR\mod$, and a morphism of $\RR\calc$-modules is a natural 
transformation of functors. Let $\RR\calc\mod$ denote the category of 
$\RR\calc$-modules.

\item Let $\theta\: \calc\to \cald$ be a functor between small 
categories. When $M$ is an $\RR\calc$-module, let $\theta_*(M)$ 
denote the left Kan extension of $M$ along $\theta$. When $N$ 
is an $\RR\cald$-module, let $\theta^*(N)=N\circ\theta$ be the 
$\RR\calc$-module induced by composition with $\theta$.

\item An $\RR\calc$-module $M$ is \emph{locally constant} on $\calc$ if it 
sends all morphisms in $\calc$ to isomorphisms of $\RR$-modules. 

\item An $\RR\calc$-module $M$ is \emph{essentially constant} if $M$ is 
isomorphic to a constant $\RR\calc$-module; i.e., isomorphic to a functor 
$\calc\too\RR\mod$ that sends each object to the same $\RR$-module $V$ 
and each morphism to $\Id_V$.

\end{enuma}
\end{Defi}

The next lemma characterizes essentially constant modules in terms of an 
action of $\pi_1(|\calc|)$. 

\begin{Lem} \label{ess.const.}
Assume $\calc$ is a small category, and let $\RR$ be a 
commutative ring. 
\begin{enuma} 

\item If $M$ is a locally constant $\RR\calc$-module, then 
for each object $c_0$ in $\calc$, there is a unique homomorphism
	\[ M_\#\: \pi_1(|\calc|,c_0) \Right5{} \Aut_{\RR}(M(c_0)) \]
satisfying the following condition: for each sequence
	\[ \sigma = \Bigl( c_0 \Right2{f_1} c_1 \Left2{f_2} c_2 
	\Right2{f_3} \cdots \Left2{f_{2m}} c_{2m}=c_0 \Bigr) \]
of morphisms in $\calc$ ($m\ge1$), beginning and ending at $c_0$, regarded 
as a loop in $|\calc|$, 
	\[ M_\#([\sigma]) = M(f_{2m})^{-1}\circ M(f_{2m-1})\circ \cdots
	\circ M(f_2)^{-1}\circ M(f_1) \in \Aut_{\RR}(M(c_0)) \,. \]

\smallskip

\item If $\calc$ is connected, then a locally constant $\RR\calc$-module 
$M$ is essentially constant if and only if $M_\#$ (as defined in (a)) is 
the trivial homomorphism for some object $c_0$ in $\calc$.

\end{enuma}
\end{Lem}

\begin{proof} \textbf{(a) } Let $\Is(\RR\mod)$ be the category of 
$\RR$-modules with only isomorphisms as morphisms, and regard $M$ as a 
functor $M\:\calc\too\Is(\RR\mod)$. This induces a map between the 
geometric realizations, and hence a homomorphism of fundamental groups 
	\[ M_\#\: \pi_1(|\calc|,c_0) \Right5{} 
	\pi_1(|\Is(\RR\mod)|,M(c_0)) \cong \Aut_{\RR}(M(c_0))\,. \]
For each sequence $\sigma$ as described above, $M_\#$ sends the class 
$[\sigma]\in\pi_1(|\calc|,c_0)$ to 
	\[ M(c_0) \Right3{M(f_1)} M(c_1) \Left3{M(f_2)} M(c_2) 
	\Right3{M(f_3)} \cdots \Left4{M(f_{2m)}} M(c_{2m})=M(c_0)\,, \]
regarded as a loop in $|\Is(\RR\mod)|$, and this is homotopic to the 
composite
	\[ M(f_{2m})^{-1}\circ M(f_{2m-1})\circ \cdots\circ M(f_1)^{-1}\circ 
	M(f_0) \in \Aut_{\RR}(M(c_0)) \]
when also regarded as a loop in $|\Is(\RR\mod)|$. 

\smallskip

\noindent\textbf{(b) } If $M$ is isomorphic to a constant functor, then 
it clearly sends all morphisms to isomorphisms, and sends a loop $\sigma$ 
as above to a sequence whose composite is the identity. Thus for each $c_0$ 
in $\calc$, the homomorphism $M_\#$ defined in (a) is trivial. It remains 
to prove the converse. 

Assume that $M$ is locally constant, and that for some object $c_0$ in 
$\calc$, the homomorphism $M_\#$ defined in (a) is trivial. Set 
$M_\calc=\colim_\calc(M)$. We claim that the natural morphism 
$\iota_c\:M(c)\too M_\calc$ is an isomorphism for each object $c$. Once 
this has been shown, the $\iota_c$ define an isomorphism of functors from 
$M$ to the constant functor with value $M_\calc$. 

For each pair of objects $c,d$ and each $\varphi\in\Mor_{\calc}(c,d)$, we 
have $\iota_c=\iota_d\circ M(\varphi)$, where $M(\varphi)$ is an isomorphism 
since $M$ is locally constant. Thus $\Im(\iota_c)=\Im(\iota_d)$ whenever 
$\Mor_{\calc}(c,d)\ne\emptyset$, and so $\Im(\iota_c)=\Im(\iota_d)$ for each 
pair of objects $c,d$ since $\calc$ is connected. So $\iota_c$ is 
surjective for each $c$ in $\calc$.

For each object $d$ in $\calc$, since $\calc$ is connected, there is a 
sequence 
	\[ c_0 \Right2{f_1} c_1 \Left2{f_2} c_2 \Right2{f_3} \cdots 
	\Left2{f_{2m}} c_{2m}=d \]
($m\ge1$) of morphisms in $\calc$ connecting $c_0$ to $d$. Set 
	\[ \eta_d = M(f_1)^{-1}\circ M(f_2) \circ \cdots\circ 
	M(f_{2m-1})^{-1}\circ M(f_{2m}) 
	\: M(d) \Right6{\cong} M(c_0) \,. \]
Then $\eta_d$ is independent of the choice of the $f_i$ since $M_\#=1$. 
This independence of the choice of sequence of morphisms also implies that 
for each pair of objects $d$ and $d'$ and each morphism 
$\varphi\in\Mor_{\calc}(d,d')$, we have $\eta_{d'}\circ M(\varphi)=\eta_d$. 
We thus get a natural morphism $\eta\:M_\calc\too M(c_0)$ such that 
$\eta\circ\iota_d=\eta_d$ for each $d$, and $\iota_d$ is injective for 
each $d$ in $\calc$ since $\eta_d$ is. We already showed that $\iota_d$ is 
surjective for each $d$, so $\iota$ is a natural isomorphism of functors 
from $M$ to the constant functor $\3{M_\calc}$.
\end{proof}

The following description of certain projective $\RR\calc$-modules will be 
needed later.

\begin{Lem} \label{kC-proj}
Let $\RR$ be a commutative ring, and let $\calc$ be a small category. For 
each object $c$ in $\calc$, let $F^{\RR\calc}_c$ be the $\RR\calc$-module that 
sends an object $d$ to $\RR(\Mor_\calc(c,d))$ (the free 
$\RR$-module with basis $\Mor_\calc(c,d)$); and sends a morphism 
$\varphi\in\Mor_\calc(d,d')$ to composition with $\varphi$. Then 
$F_c^{\RR\calc}$ is projective, and for each $\RR\calc$-module $M$, evaluation at 
$\Id_c\in F^{\RR\calc}_c(c)$ defines a bijection 
$\Mor_{\RR\calc}(F^{\RR\calc}_c,M)\cong M(c)$.
\end{Lem}

\begin{proof} The bijection $\Mor_{\RR\calc}(F^{\RR\calc}_c,M)\cong M(c)$ 
holds by Yoneda's lemma. In particular, 
$\Mor_{\RR\calc}(F^{\RR\calc}_c,-)$ is an exact functor, and so 
$F_c^{\RR\calc}$ is projective. 
\end{proof}

We now restrict further to two different cases: one where 
$\theta\:\calc\too\cald$ is bijective on objects, and the second where 
$\cald$ is the category of a group. In each of these cases, we are able 
to get much more precise results about the existence of 
$\Omega$-resolutions. 

\subsection{Functors bijective on objects}
\leavevmode

We begin with the case where $\theta$ is bijective on objects. When $\RR$ 
is a commutative ring and one additional technical assumption holds, we can 
say quite precisely in which cases there always exist $\Omega$-resolutions.

\begin{Prop} \label{ex3}
Fix a commutative ring $\RR$. Let $\theta\:\calc\too\cald$ be a functor 
between small categories that is bijective on objects and surjective on 
morphism sets. Then 
\begin{enuma} 

\item $\Osys{\RR\calc\mod}{\RR\cald\mod}$ is an $\Omega$-system, and the 
subcategory $\theta^*(\RR\cald\mod)$ is closed under subobjects in 
$\RR\calc\mod$.

\end{enuma}
For each object $c$ in $\calc$, set 
	\[ K_c = \Ker[\theta_c\:\Aut_\calc(c)\Right2{}
	\Aut_\cald(\theta(c)) \bigr], \]
and assume that $\theta$ has the following property: 
	\beqq \parbox{\shorter}{for each pair of objects $c,c'$ in $\calc$, and 
	each pair of morphisms $\varphi,\varphi'\in\Mor_\calc(c,c')$ such that 
	$\theta_{c,c'}(\varphi)=\theta_{c,c'}(\varphi')$, there is some 
	$\alpha\in K_{c'}$ such that 
	$\varphi=\alpha\varphi'$.} \label{e:mor-quot} \eeqq
Then the following hold.
\begin{enuma} \stepcounter{enumi}

\item If $K_c$ is $\RR$-perfect for each $c\in\Ob(\calc)$, then 
$\theta^*(\RR\cald\mod)$ is closed under extensions, and hence all 
projectives in $\RR\cald\mod$ have $\Omega$-resolutions.

\item If $K_c$ is not $\RR$-perfect for some $c\in\Ob(\calc)$, then 
$\theta^*(\RR\cald\mod)$ is not closed under extensions, and there is a 
projective object $X$ in $\RR\cald\mod$ that does not have an 
$\Omega$-resolution. 

\end{enuma}
\end{Prop}

\begin{proof} \textbf{(a) } Since $\theta$ is surjective on objects and 
morphisms, it is quasisurjective. Since it is bijective on objects, the 
condition on $\cali(\theta\dn d)$ in Proposition \ref{AC->AD} holds for all 
objects $d$ in $\cald$, and so $\Osys{\RR\calc\mod}{\RR\cald\mod}$ is an 
$\Omega$-system by that proposition. 

Since $\theta$ is bijective on objects and surjective on morphisms, an 
$\RR\calc$-module $M$ is isomorphic to an object in 
$\theta^*(\RR\cald\mod)$ if and only if it has the following property: if 
$\varphi,\psi\in\Mor_\calc(c,c')$ are such that 
$\theta(\varphi)=\theta(\psi)$ (some $c,c'\in\Ob(\calc)$), then 
$M_{c,c'}(\varphi)=M_{c,c'}(\psi)$. In particular, $\theta^*(\RR\cald\mod)$ 
is closed under subobjects.

\smallskip

\noindent\textbf{(b,c) } Now assume that \eqref{e:mor-quot} holds. 
For each $\RR\calc$-module $M$, let $M\KK$ be the $\RR\cald$-module 
defined by setting, for each $d\in\Ob(\cald)$ and $c\in\theta^{-1}(d)$, 
	\[ (M\KK)(d) \cong M(c)\KK[_c] \defeq M(c)\big/ \Gen{\alpha(x)-x 
	\,\big|\, x\in M(c),~ \alpha\in K_c }. \]
For each morphism $\varphi\in\Mor_\calc(c,c')$ and each $\alpha\in K_c$, 
$\theta_{c,c'}(\varphi\circ\alpha)=\theta_{c,c'}(\varphi)$, so by 
\eqref{e:mor-quot}, there is $\beta\in K_{c'}$ such that 
$\varphi\circ\alpha=\beta\circ\varphi$. Hence for each $x\in M(c)$, 
$\varphi_*(x)$ and $\varphi_*(\alpha(x))$ are in the same orbit of 
$K_{c'}$. It follows that $\varphi_*\in\Mor_\RR(M(c),M(c'))$ induces a 
homomorphism between the quotient modules $M\KK(c)$ and $M\KK(c')$. So by 
\eqref{e:mor-quot} and since $\theta$ is surjective on morphisms, there is 
a unique functor $M\KK$ on $\cald$ such that the natural surjections 
$M(c)\too M\KK(\theta(c))$ define a morphism of $\RR\calc$-modules 
$M\too\theta^*(M\KK)$, and hence a morphism of $\RR\cald$-modules 
$\theta_*(M)\too M\KK$. 

By \eqref{e:mor-quot} and the surjectivity of $\theta$ again, we have a 
natural bijection 
$\Mor_{\RR\calc}(M,\theta^*N)\cong\Mor_{\RR\cald}(M\KK,N)$ for each 
$\RR\calc$-module $M$ and each $\RR\cald$-module $N$, and thus 
$M\KK\cong\theta_*(M)$. We have now shown that 
	\beqq 
	\parbox{\shorter}{for each $\RR\calc$-module $M$ and each 
	$c\in\Ob(\calc)$, the natural morphism 
	\[\aaa_M\:M(c)\Right5{}(\theta^*\theta_*M)(c)=(\theta_*M)(\theta(c))\] 
	induces an isomorphism $M(c)_{K_c}\cong(\theta_*M)(\theta(c))$.}
	\label{e:iso-theta*} \eeqq

\smallskip

\noindent\textbf{(b) } Assume $K_c$ is $\RR$-perfect for each $c\in\Ob(\calc)$. Let 
$0\too M'\too M\too M''\too0$ be an extension of $\RR\calc$-modules such 
that $M'$ and $M''$ are in $\theta^*(\RR\cald\mod)$. For each 
$c\in\Ob(\calc)$, $K_c$ acts trivially on $M'(c)$ and on 
$M''(c)$, and hence also acts trivially on $M(c)$ by Lemma \ref{l:p-perf}. 
So $M\cong\theta^*(\theta_*(M))$ by \eqref{e:iso-theta*}. 

This proves that $\theta^*(\RR\cald\mod)$ is closed under extensions in 
$\RR\calc\mod$, and hence by Proposition \ref{exists-res} that 
$\Omega$-resolutions exist of all projectives in $\RR\cald\mod$.

\smallskip

\noindent\textbf{(c) } Assume, for some object $c_0$ in $\calc$, that $K_{c_0}$ is 
not $\RR$-perfect, and set $d_0=\theta(c_0)$. Let $F_{c_0}^{\RR\calc}$ and 
$F_{d_0}^{\RR\cald}$ be the projective $\RR\calc$- and $\RR\cald$-modules 
defined in Lemma \ref{kC-proj}; thus 
	\[ F_{c_0}^{\RR\calc}(c)=\RR(\Mor_\calc(c_0,c)) \qquad\textup{and}\qquad
	F_{d_0}^{\RR\cald}(d)=\RR(\Mor_\cald(d_0,d)) \]
for each $c\in\Ob(\calc)$ and $d\in\Ob(\cald)$. 
Since $\theta$ is surjective on morphisms, there is a natural surjection of 
$\RR\calc$-modules $\chi\:F_{c_0}^{\RR\calc}\too\theta^*(F_{d_0}^{\RR\cald})$ 
that sends $\varphi\in F_{c_0}^{\RR\calc}(c)$ to $\theta(\varphi)\in 
\theta^*(F_{d_0}^{\RR\cald})(c)=F_{d_0}^{\RR\cald}(\theta(c))$. 

Set $Q_0=\Ker(\chi)$, and consider the exact sequence 
	\[ 0 \too (L_1\theta_*)(\theta^*(F_{d_0}^{\RR\cald})) \Right3{} 
	\theta_*(Q_0) \Right3{} \theta_*(F_{c_0}^{\RR\calc}) \Right3{} 
	F_{d_0}^{\RR\cald} \too 0.  \]
Here, $\theta_*(F_{c_0}^{\RR\calc})\cong F_{d_0}^{\RR\cald}$ by 
\eqref{e:iso-theta*} and since $\Mor_\calc(c_0,c)/K_c\cong 
\Mor_\cald(d_0,\theta(c))$ for each $c$ in $\calc$ by \eqref{e:mor-quot}. 
Thus $(L_1\theta_*)(\theta^*F_{d_0}^{\RR\cald})\cong \theta_*(Q_0)$. We 
will show that $\theta_*(Q_0)(d_0)\ne0$; then 
$(L_1\theta_*)(\theta^*F_{d_0}^{\RR\cald})\ne0$, so $F_{d_0}^{\RR\cald}$ 
has no $\Omega$-resolution by Proposition \ref{exists-res}, and 
$\theta^*(\RR\cald\mod)$ is not closed under extensions by Lemma 
\ref{l:closed}(b).

Set $\End_\calc^{(1)}(c_0)=\Aut_\calc(c_0)$, and let 
$\End_\calc^{(2)}(c_0)$ be its complement (as a set) in $\End_\calc(c_0)$. 
Set $\End_\cald^{(i)}(d_0)=\theta_{c_0}(\End_\calc^{(i)}(c_0))$ for 
$i=1,2$. Thus 
	\[ \End_\calc(c_0) = \End_\calc^{(1)}(c_0) \amalg \End_\calc^{(2)}(c_0) 
	\qquad\textup{and}\qquad
	\End_\cald(d_0) = \End_\cald^{(1)}(d_0) \amalg \End_\cald^{(2)}(d_0) \,: \]
the first by definition, and the second by \eqref{e:mor-quot} and since 
$\theta_{c_0}$ is surjective. Set 
$U_\calc^{(i)}=\RR(\End_\calc^{(i)}(c_0))$ and 
$U_\cald^{(i)}=\RR(\End_\cald^{(i)}(d_0))$ ($i=1,2$), so that 
$F_{c_0}^{\RR\calc}(c_0)= U_\calc^{(1)}\oplus U_\calc^{(2)}$ and 
$F_{d_0}^{\RR\cald}(d_0)= U_\cald^{(1)}\oplus U_\cald^{(2)}$. Thus 
$Q_0(c_0)=Q_0^{(1)}\oplus Q_0^{(2)}$ where $Q_0^{(i)}$ is the kernel of the 
surjection $U_\calc^{(i)}\too U_\cald^{(i)}$. 

By \eqref{e:iso-theta*}, we must show that 
$Q_0(c_0)_{K_{c_0}}\cong\theta_*(Q_0)(d_0)\ne0$, and to do this, it suffices 
to show that $(Q_0^{(1)})_{K_{c_0}}\ne0$. Set $A=\RR[\Aut_\calc(c_0)]\cong 
U_\calc^{(1)}$, and identify it with the group ring. We can also identify 
$Q_0^{(1)}=I$: the 2-sided ideal in $A$ generated as an $\RR$-module by the 
elements $g-h$ for $g,h\in\Aut_\calc(c_0)$ such that $gh^{-1}\in K_{c_0}$. 
Then $X_{K_{c_0}}=X/IX$ for each $A$-module $X$. In particular, 
$(Q_0^{(1)})_{K_{c_0}}\cong I/I^2$.

Consider the short exact sequence $0\to I\to A\to A/I\to0$ of 
$\RR[K_{c_0}]$-modules. Since $A$ is projective, this induces an 
isomorphism $I/I^2\cong H_1(K_{c_0};A/I)$. Since $K_{c_0}$ is not 
$\RR$-perfect and acts trivially on the free $\RR$-module 
$A/I\cong\RR[\Aut_\cald(d_0)]$, we now conclude that $I/I^2\ne0$. 
\end{proof}

\begin{Ex} \label{ex3a}
In the situation of Proposition \ref{ex3}(c), there can also be nonzero 
$\RR\calc$-modules that do have $\Omega$-resolutions. For example, fix a  
prime $p$, set $\RR=\F_p$, and assume that $\Ob(\calc)=\Ob(\cald)=\{x,y\}$, 
where $\End_\cald(x)=\End_\calc(y)=\End_\cald(y)=\{\Id\}$ and 
$\End_\calc(x)\cong C_p$, and each category has a unique morphism from $x$ 
to $y$ and none from $y$ to $x$. Then the unique functor 
$\theta\:\calc\too\cald$ satisfies the hypotheses of Proposition \ref{ex3}, 
and $K_x\cong C_p$ is not $\RR$-perfect while $K_y=1$ is. Set 
$X=F_y^{\RR\cald}$; then $\theta^*(X)\cong F_y^{\RR\calc}$ is projective as 
an $\RR\calc$-module, so $X$ has as $\Omega$-resolution the sequence $0\too 
F_y^{\RR\calc}\xto{\quad\gee\quad} \theta^*(F_y^{\RR\cald})\too0$.
\end{Ex}

\subsection{Categories over a group}
\leavevmode

The other large family of examples we consider are those where 
$\cald=\calb(\pi)$ (as defined in the introduction) for a 
group $\pi$. 

\begin{Defi} \label{d:overG}
A \emph{category over a group $\pi$} consists of a pair 
$(\calc,\theta)$, where $\calc$ is a nonempty small connected category, and 
$\theta\:\calc\Right2{}\calb(\pi)$ is a functor such that the homomorphism 
$\pi_1(|\calc|)\too\pi$ induced by $\theta$ is surjective.
\end{Defi}

For example, if $G$ and $\pi$ are groups, and 
$\theta\:\calb(G)\too\calb(\pi)$ is the functor induced by a surjective 
homomorphism $G\too\pi$, then $(\calb(G),\theta)$ is a category over 
$\pi$. 

As another example, one that helped motivate this work, let $\SFL$ be a 
$p$-local compact group as defined in \cite{BLO3}. Set 
$\pi=\pi_1(|\call|\pcom)$. Then $\pi$ is a finite $p$-group, and there is a 
natural functor $\theta\:\call \Right2{}\calb(\pi)$ whose restriction to 
$\calb(S)$ is surjective. It follows from properties of linking systems 
that $(\call, \theta)$ is a category over $\pi$. We refer to the 
introduction to Section \ref{s:examples} for more details.

\begin{Lem}\label{isos-over-G}
Let $\RR$ be a commutative ring, and let $(\calc,\theta)$ be a category 
over $\pi$. Then 
\begin{enuma} 
\item the overcategory $\theta\dn\Bobj[\pi]$ is connected, $\cpth$ is an 
$\Omega$-system, and the 
projection $|\theta\dn\Bobj[\pi]|\too|\calc|$ is a covering space with 
covering group $\pi$.

\end{enuma}
For an $\RR\calc$-module $M$, 
\begin{enuma} \setcounter{enumi}{1}

\item $\theta_*(M)\cong \colim_{\theta\dn\Bobj[\pi]}(M)$; and

\item $M\cong\theta^*(N)$ for some $\RR\pi$-module $N$ if and only if $M$ 
is locally constant on $\calc$ and essentially constant on 
$\theta\dn\Bobj[\pi]$. 

\end{enuma}
\end{Lem}

\begin{proof} \textbf{(a) } Since $\pi$ acts freely on 
$|\theta\dn\Bobj[\pi]|$ with orbit space $|\calc|$, the projection to 
$|\calc|$ is a covering space with covering group $\pi$. In particular, 
$|\theta\dn\Bobj[\pi]|$ is connected since $\pi_1(|\calc|)$ surjects onto 
$\pi$. Also, $\theta$ is quasisurjective since $\pi_1(|\calc|)$ surjects 
onto $\pi$, and so $\cpth$ is an $\Omega$-system by Proposition \ref{AC->AD}. 

\smallskip

\noindent\textbf{(b) } By definition of left Kan extension, 
$\theta_*(M)=\colim_{\theta\dn\Bobj[\pi]}(M)$.

\smallskip

\noindent\textbf{(c) } Assume $M$ is locally constant on $\calc$ and 
essentially constant on $\theta\dn\Bobj[\pi]$. We claim that the 
natural morphism $\aaa_M\:M\too\theta^*\theta_*(M)$ is an isomorphism. 
This means showing, for each $c$ in $\calc$, that the natural morphism 
from $M(c)$ to $\theta_*(M)$ is an isomorphism. By (b), this is 
equivalent to showing that the natural morphism 
$\xi_{c}\:M(c)\too \colim_{\theta\dn\Bobj[\pi]}(M)$ is an 
isomorphism for each $c$. But this holds since by assumption, the 
composite of $M$ with the forgetful functor 
$\theta\dn\Bobj[\pi]\too\calc$ is isomorphic to a constant functor.

Conversely, if $M\cong\theta^*(N)$, then $M$ is locally constant on 
$\calc$, and isomorphic to a constant functor on $\theta\dn\Bobj[\pi]$. 
\end{proof}

\begin{Lem} \label{H1(|C0|)}
Let $\RR$ be a commutative ring, let $(\calc,\theta)$ be a category 
over a group $\pi$, and set 
$H=\Ker\bigl[\pi_1(|\theta|)\:\pi_1(|\calc|)\too\pi\bigr]$.
\begin{enuma} 
\item If $H$ is $\RR$-perfect, then $(L_1\theta_*)(\theta^*X)=0$ 
for each $\RR\pi$-module $X$; and 
\item if $H$ is not $\RR$-perfect, then 
$(L_1\theta_*)(\theta^*X)\ne0$ for each $\RR\pi$-module $X\ne0$ that (as an 
$\RR$-module) contains $\RR$ as a direct summand.
\end{enuma}
\end{Lem}

\begin{proof} For each $\RR\pi$-module $X$, $(L_1\theta_*)(\theta^*X)\cong 
(L_1(\colim_{\theta\dn\Bobj[\pi]}))(\theta^*X)$ as $\RR$-modules by Lemma 
\ref{isos-over-G}(b), and 
$(L_1(\colim_{\theta\dn\Bobj[\pi]}))(\theta^*X)\cong 
H_1(|\theta\dn\Bobj[\pi]|;X)$ since the two sides are homology groups of 
the same chain complex by \cite[Appendix II, Proposition 3.3]{GZ}. Here, 
the homology is with untwisted coefficients since 
$H\cong\pi_1(|\theta\dn\Bobj[\pi]|)$ (Lemma \ref{isos-over-G}(a)) acts 
trivially on the $\RR\pi$-module $X$. Thus $(L_1\theta_*)(\theta^*X)=0$ if 
and only if $H_1(|\theta\dn\Bobj[\pi]|;X)=0$. Points (a) and (b) now follow 
since $H\cong\pi_1(|\theta\dn\Bobj[\pi]|)$ is $\RR$-perfect if and only if 
$H_1(|\theta\dn\Bobj[\pi]|;\RR)\cong H_1(H;\RR)=0$. 
\end{proof}

Since a category over a group $\pi$ gives rise to an $\Omega$-system, 
we can now work with $\Omega$-resolutions in this situation. 

\begin{Lem} \label{O-res}
Let $\ctg$ be a category over a group $\pi$, and let $\RR$ be a commutative 
ring. A complex of $\RR\calc$-modules 
	$$  \dots \Right2{} P_2 \Right2{\2_2} P_1 \Right2{\2_1} P_0 
	\Right2{\varepsilon} \theta^*(\RR \pi) \Right2{} 0 $$
is an $\Omega$-resolution of $\theta^*(\RR\pi)$ with respect to the 
$\Omega$-system $\cpth$ of Proposition \ref{AC->AD} if and only if 
\begin{enumerate}[\rm(1) ]

\item $P_n$ is a projective $\RR\calc$-module for each $n\geq 0$; 
\label{O-res-1}

\item the complex $\theta_*(P_*)$ is acyclic, and $\gee$ induces an 
isomorphism $H_0(P_*)(c)\cong\RR\pi$ for each $c\in\Ob(\calc)$; and 
\label{O-res-2}

\item for each $n\ge0$, $H_n(P_*)$ is locally constant on $\calc$ and 
essentially constant on $\theta\dn\Bobj[\pi]$. \label{O-res-3}

\end{enumerate} 
\end{Lem}

\begin{proof} By Lemma \ref{isos-over-G}(c), \eqref{O-res-3} is equivalent 
to the first statement in ($\Omega$-\ref{sqres-3}) (that $H_n(P_*,\2_*)$ is 
isomorphic to an object in $\theta^*(\RR\pi\mod)$). The equivalence of 
\eqref{O-res-1} with ($\Omega$-\ref{sqres-1}), and of \eqref{O-res-2} with 
($\Omega$-\ref{sqres-2}) and the second part of ($\Omega$-\ref{sqres-3}) 
(that $H_0(P_*,\2_*)\cong X$), is clear. 
\end{proof}

By Proposition \ref{p2:functoriality}, if $\RR$ is a commutative ring and 
$(\calc,\theta)$ is a category over a group $\pi$, and there is at 
least one $\Omega$-resolution of $\RR\pi$, then all $\Omega$-resolutions 
are chain homotopy equivalent to each other. This allows us to define 
``$\Omega$-homology'' in this situation.

\begin{Defi}\label{homology}
Let $\ctg$ be a category over a group $\pi$. For a commutative ring 
$\RR$, if there is an $\Omega$-resolution $(P_*,\2_*)$ of $\theta^*(\RR\pi)$ 
with respect to $\ctg$, then we define 
	\[ H^\Omega_*(\calc,\theta; \RR) = \theta_*\bigl(H_*(P_*,\2_*)\bigr). \] 
\end{Defi}

The following proposition is a first step towards determining for which 
categories over $\pi$ the free module $\RR\pi$ has an $\Omega$-resolution. 
In the next section, we will show that $\Omega$-resolutions of $\RR\pi$ do 
exist in many of the cases not excluded here. Recall that $\calc$ is an 
\emph{EI-category} if all endomorphisms of objects in $\calc$ are 
automorphisms.

\begin{Prop} \label{ex4}
Fix a commutative ring $\RR$. Let $(\calc,\theta)$ be a category over a 
group $\pi$, and set 
$H=\Ker\bigl[\pi_1(|\theta|)\:\pi_1(|\calc|)\too\pi\bigr]$. Thus 
$\cpth$ is an $\Omega$-system by Lemma \ref{isos-over-G}(a).
\begin{enuma} 

\item Assume $H$ is $\RR$-perfect. Then 
$(L_1\theta_*)(\theta^*(X))=0$ for each $X$ in $\RR\pi\mod$, and 
$\theta^*(\RR\pi\mod)$ is closed under extensions in $\RR\calc\mod$. If 
$\calc=\calb(G)$ for a group $G$, then $\theta^*(\RR\pi\mod)$ is closed 
under subobjects in $\RR\calc\mod$, and each projective $\RR\pi$-module 
has an $\Omega$-resolution. If $\calc$ is an EI-category with more than 
one isomorphism class, then $\theta^*(\RR\pi\mod)$ is not closed under 
subobjects in $\RR\calc\mod$. 

\item Assume $H$ is not $\RR$-perfect. Then 
$\theta^*(\RR\pi\mod)$ is not closed under extensions in $\RR\calc\mod$, 
and the projective $\RR\pi$-module $\RR\pi$ does not have an 
$\Omega$-resolution. More generally, if $X$ is a nonzero projective 
$\RR\pi$-module that is free as an $\RR$-module, then $X$ has no 
$\Omega$-resolution.

\end{enuma}
\end{Prop}

\begin{proof} \noindent\textbf{(a) } Assume that $H$ is $\RR$-perfect. Then 
$(L_1\theta_*)(\theta^*(X))=0$ for each $\RR\pi$-module $X$ by Lemma 
\ref{H1(|C0|)}(a). Hence by Lemma \ref{l:closed}(c), $\theta^*(\RR\pi\mod)$ 
is closed under extensions in $\RR\calc\mod$. If $\calc\cong\calb(G)$ for 
some $G$, then $\theta$ is surjective on morphisms, $H=\Ker[G\too\pi]$, and 
$\theta^*(\RR\pi\mod)$ is closed under subobjects in $\RR\calc\mod$ by 
Proposition \ref{ex3}(a). 

Assume $\calc$ is an EI-category with more than one isomorphism class, and 
let $x,y\in\Ob(\calc)$ be a pair of nonisomorphic objects. At least 
one of the sets $\Mor_\calc(x,y)$ and $\Mor_\calc(y,x)$ must be empty; we 
can assume that $\Mor_\calc(x,y)=\emptyset$. Let $\3{\RR}$ be the 
constant $\RR\calc$-module with value $\RR$, and let $M\le\3{\RR}$ 
be the submodule where $M(c)=0$ if $\Mor_\calc(x,c)=\emptyset$ and 
$M(c)=\RR$ otherwise. Then $M(c)=\RR$ and $M(c')=0$ imply that 
$\Mor_\calc(c,c')=\emptyset$; thus $M$ is well defined as a submodule of 
$\3{\RR}$. Also, $M(x)=\RR$ so $M\ne0$, and $M(y)=0$ so $M$ is 
properly contained in $\3{\RR}$. Since $\calc$ is connected, $M$ is 
not locally constant, and hence not isomorphic to an object in 
$\theta^*(\RR\pi\mod)$. So $\theta^*(\RR\pi\mod)$ is not closed under 
subobjects in $\RR\calc\mod$.

\smallskip

\noindent\textbf{(b) } Fix an object $c_0$ in $\calc$, and set 
$G=\pi_1(|\calc|,c_0)$ for short. Let $\eta\:G\too\pi$ be the homomorphism 
induced by $|\theta|\:|\calc|\too|\calb(\pi)|=B\pi$. Thus $\eta$ is 
surjective and $H=\Ker(\eta)$. 

Assume $H$ is not $\RR$-perfect. By Lemma \ref{H1(|C0|)}(b), for each 
nonzero $\RR\pi$-module $X$ that is free as an $\RR$-module, 
$(L_1\theta_*)(\theta^*X)\ne0$. So $X$ has no $\Omega$-resolution by 
Proposition \ref{exists-Omega1}, and it remains to show that 
$\theta^*(\RR\pi\mod)$ is not closed under extensions in $\RR\calc\mod$. 

Set $N_0=H\ab\otimes_{\Z}\RR\cong H_1(H;\RR)$, regarded as an $\RR$-module, 
and let $\chi\:H\too N_0$ be the homomorphism $\chi(h)=[h]\otimes1$. Since 
$H$ is not $\RR$-perfect, $N_0\ne0$, and $\chi$ is not the trivial 
homomorphism. Let $M_0$ be the $\RR H$-module with underlying $\RR$-module 
$N_0\times N_0$, where $h\in H$ acts via the matrix $\mxtwo1{\chi(h)}01$. 
Thus there is a submodule $M'_0=\{(x,0)\,|\,x\in\RR\}\le M_0$ such that $H$ 
acts trivially on $M'_0$ and on $M_0/M'_0$.

Now set $M=\RR G\otimes_{\RR H}M_0$. Thus $M$ is an $\RR G$-module, and 
contains a submodule $M'$ such that $M'$ and $M/M'$ are both isomorphic to 
$\eta^*(\RR\pi)$.

We now use this to construct a counterexample to $\theta^*(\RR\pi\mod)$ 
being closed under extensions. For each $c\in\Ob(\calc)$, choose a 
path $\til\phi_c$ in $|\theta\dn\Bobj[\pi]|$ from $(c_0,\Id)$ to $(c,\Id)$ 
($|\theta\dn\Bobj[\pi]|$ is connected by Lemma \ref{isos-over-G}(a)), and 
let $\phi_c$ be its image in $|\calc|$. In particular, let $\til\phi_{c_0}$ 
and $\phi_{c_0}$ be the constant paths at $(c_0,\Id)$ and $c_0$, 
respectively. Define a functor 
$\til\theta\:\calc\too\calb(G)$ by sending each object in $\calc$ to the 
unique object $\Bobj[G]$, and by sending each morphism 
$\omega\in\Mor_\calc(c,c')$ to the class of the loop 
$\phi_c\cdot\omega\cdot\phi_{c'}^{-1}$ (where we compose paths from left to 
right). We claim that 
\begin{enumi} 
\item $\pi_1(|\til\theta|)\:\pi_1(|\calc|,c_0) \too 
\pi_1(\calb(G),\Bobj[G])=G$ is the identity on $G$; and 
\item $\theta=\calb(\eta)\circ\til\theta$. 
\end{enumi}
Point (i) is immediate from the definition of $\til\theta$ (and since 
$\phi_{c_0}$ is the constant path). Point (ii) holds since the paths 
$\phi_c$ all lift to $|\theta\dn\Bobj[\pi]|$ and hence are sent to trivial 
loops in $\calb(\pi)$, and since $\eta\:G\too\pi$ is induced by $\theta$. 

Now, $\til\theta^*(M)$ is an $\RR\calc$-module with submodule 
$\til\theta^*(M')$, such that by (ii), 
	\[ \til\theta^*(M')\cong \til\theta^*(\eta^*(\RR\pi)) \cong 
	\theta^*(\RR\pi)
	\qquad\textup{and}\qquad
	\til\theta^*(M)\big/\til\theta^*(M')\cong
	\til\theta^*(M/M')\cong\theta^*(\RR\pi). \] 
Thus $\til\theta^*(M')$ and $\til\theta^*(M)\big/\til\theta^*(M')$ are 
both isomorphic to objects in $\theta^*(\RR\pi\mod)$. As for 
$\til\theta^*(M)$, by (i), the homomorphism 
	\[ (\til\theta^*(M))_\#\: G = \pi_1(|\calc|,c_0) \Right5{} 
	\Aut_{\RR}(\til\theta^*(M)(c_0)) = \Aut_{\RR}(M) \]
of Lemma \ref{ess.const.}(a) is just the given action of $G$ on the $\RR 
G$-module $M$. So its restriction to $H=\pi_1(|\theta\dn\Bobj[\pi]|)$ is 
nontrivial, and by Lemma \ref{ess.const.}(b), 
$\til\theta^*(M)$ is not essentially constant on $\theta\dn\Bobj[\pi]$.  
By Lemma \ref{isos-over-G}(c), it is not isomorphic to an object in 
$\theta^*(\RR\pi\mod)$, and thus $\theta^*(\RR\pi\mod)$ is not closed under 
extensions in $\RR\calc\mod$. 
\end{proof}

Note that Proposition \ref{exists-res} need not apply under 
the hypotheses of Proposition \ref{ex4} when $H$ is 
$\RR$-perfect, although $\Omega_1$-resolutions (at least) 
exist by Proposition \ref{exists-Omega1}. For example, if 
$(\calc,\theta)$ is a category over $\pi$ where $\calc$ is an 
EI-category with more than one isomorphism class of objects, then 
$\theta^*(\RR\pi\mod)$ is not closed under subobjects in $\RR\calc\mod$ 
by Proposition \ref{ex4}(a), and so Proposition \ref{exists-res} cannot 
be applied. In contrast, if $\calc$ is the category of a group, then 
$\Omega$-resolutions always exist by Proposition \ref{ex3}(b). We will 
show in Theorem \ref{Omega_C0} that at least with one extra condition 
on $\RR$ and $H$, $\Omega$-resolutions of $\RR\pi$ always exist when 
the hypotheses of Proposition \ref{ex4} hold and $H$ is 
$\RR$-perfect.

\begin{Ex} \label{ex4a}
In the situation of Proposition \ref{ex4}(a), if $\calc$ is 
not an EI-category, then $\theta^*(\RR\pi\mod)$ can fail to be closed under 
subobjects even when $\calc$ has only one object, and can be closed under 
subobjects even when $\calc$ has more than one isomorphism class of 
object: 
\begin{enuma} 

\item Set $\RR=\Z$, $\pi=\Z$, and $\calc=\calb(\N)$, and let 
$\theta\:\calb(\N)\too\calb(\pi)$ be the inclusion. Then $(\calc,\theta)$ 
is a category over $\pi$. Let $N$ be the $\RR\pi$-module with underlying 
group $\Q$, where $\pi=\Z$ acts via $n(x)=2^nx$. Let $M$ be the 
$\RR\N$-module with underlying group $\Z$, where $n\in\N$ acts in the same 
way. Thus $M$ is a submodule of $\theta^*(N)$, but is not isomorphic 
to an object in $\theta^*(\RR\pi\mod)$.

\item Let $\calc$ be a category with two objects $x$ and $y$, where 
$\End_\calc(x)=\{0_x,1_x\}$, $\End_\calc(y)=\{0_y,1_y\}$, and there are 
unique morphisms $0_{xy}\in\Mor_\calc(x,y)$ and $0_{yx}\in\Mor_\calc(y,x)$. 
Composition is defined by multiplication of the labels $0$ or $1$. Set 
$\pi=\Z$, and let $\theta\:\calc\too\calb(\Z)$ be the functor that sends 
all endomorphisms to $0$ and the other two morphisms to $1$ and $-1$, 
respectively. Via generators and relations, one checks that $\theta$ 
induces an isomorphism $\pi_1(|\calc|)\cong\Z$. We are thus in the 
situation of Proposition \ref{ex4}(a) with $H=1$. An $\RR\calc$-module $M$ 
is isomorphic to an object in $\theta^*(\RR\pi\mod)$ if and only if all 
endomorphisms induce the identity, in which case the other two morphisms 
induce inverse isomorphisms between $M(x)$ and $M(y)$. So 
$\theta^*(\RR\pi\mod)$ is closed under subobjects in this case.

\end{enuma}
\end{Ex}

\section{Homology of loop spaces of categories over groups}
\label{s:loops}

We next show, in the situation of Proposition \ref{ex4}(a), that 
$\Omega$-resolutions of $\RR\pi$ with respect to $\ctg$ and 
$H\nsg\pi_1(|\calc|)$ do exist, at least whenever $\RR$-plus 
constructions exist for $(|\calc|,H)$, and that the homology of an 
$\Omega$-resolution is the $\RR$-homology of the loop space of that 
$\RR$-plus construction (Theorem \ref{Omega_C0}). For example, when $\kk$ 
is a field of characteristic $p$ for some prime $p$ and $\pi$ is a finite 
$p$-group, the homology of the $\Omega$-resolution is isomorphic to 
$H_*(\Omega(|\calc|\pcom);\kk)$ (Theorem \ref{Omega_C}).

Throughout this section, we work mostly with simplicial sets and their 
realizations, referring to \cite[Chapter I]{GJ} and \cite{Curtis} for the 
definitions and basic properties that we use. In particular, Kan fibrations 
of simplicial sets (called ``fibre maps'' by Curtis) play an important role 
here, and we refer to \cite[\S\,I.3]{GJ} and \cite[Definition 2.5]{Curtis} 
for their definitions. We let $|K|$ denote the geometric realization of a 
simplicial set $K$, let $C_*(K)$ denote its simplicial chain complex, and 
write $H_*(K)=H_*(C_*(K))$ ($\cong H_*(|K|)$). Thus 
$|\calc|=|\caln(\calc)|$ when $\calc$ is a small category and 
$\caln(\calc)$ is its nerve. Note that if $f\:E\too K$ is a Kan fibration 
and $\mu\:L\too K$ is a simplicial map, then the pullback of $f$ along 
$\mu$ is also a Kan fibration. 

For a small category $\calc$, a \emph{$\calc$-diagram of simplicial sets} 
is a functor from $\calc$ to simplicial sets, and a 
morphism of $\calc$-diagrams is a natural transformation of such 
functors. Let $\3K$ denote the constant $\calc$-diagram that 
sends each object to the simplicial set $K$, and let $\3f\:\3K\too\3L$ 
denote the morphism induced by a map $f\:K\too L$ of simplicial sets.

Let $E\calc$  denote the $\calc$-diagram of simplicial sets where 
$E\calc(c)=\caln(\Id_\calc\dn c)$, and where a morphism 
$\varphi$ in $\calc$ induces a map between spaces $E\calc(-)$ by 
composition with $\varphi$. Then $|E\calc|$ is a free 
$\calc$-CW complex (see \cite[Definition 3.2]{DL}) and 
$|E\calc(c)|$ is contractible for each $c$ in $\calc$, so $|E\calc|$ is 
the ``$\calc$-CW-approximation'' of the trivial (point) $\calc$-space 
in the sense of \cite[Definitions 3.6 and 3.8]{DL}. The forgetful 
functors $\Id_\calc\dn c\to\calc$ induce a natural transformation 
$\eta\:E\calc\to\3{\caln(\calc)}$.

For each Kan fibration $f\: K\too \caln(\calc)$, let $\mu\:E_f\too E\calc$ 
denote the pullback of $K$ along $\eta$. Thus $E_f$ is the 
$\calc$-diagram of simplicial sets that sends an object $c$ in $\calc$ to the pullback 
$E_f(c)$ of the system
	\[ K \Right4{f} \caln(\calc) \Left4{\eta_c} E\calc(c). \]

\begin{Lem}\label{C*(Ef) projective}
Fix a commutative ring $\RR$ and a small category $\calc$, and let 
$f\colon K\to \caln(\calc)$ be a Kan fibration. Then for each $n\geq 0$, 
the $\RR\calc$-module $C_n(E_f;\RR)$ is  projective, and the morphism 
$\omega\:E_f\too \3K$ induces an isomorphism 
$\colim_{\calc}\bigl(C_*(E_f;\RR)\bigr)\cong C_*(K;\RR)$.
\end{Lem}

\begin{proof} For each $n\ge0$ and each object $c\in\calc$, 
$C_n(E\calc;\RR)(c)$ has as basis the set of all chains $(c_0\to c_1\to\cdots\to 
c_n\to c)$.  So in the notation of Lemma \ref{kC-proj}, the 
$\RR\calc$-module $C_n(E\calc;\RR)$ is the direct sum of one copy of 
$F_{c_n}^{\RR\calc}$ for each $n$-simplex $(c_0\to c_1\to\cdots\to c_n)$ in 
$\caln(\calc)$. In particular, it is projective, and since 
$\colim_\calc(F_{c_n}^{\RR\calc})\cong\RR$, the natural transformation 
$\eta\:E\calc\too\3{\caln(\calc)}$ induces an isomorphism 
$\colim_\calc(C_n(E\calc;\RR))\cong C_n(\caln(\calc);\RR)$.

This proves the lemma when $f$ is the identity fibration, and the general 
case is similar. An  $n$-simplex in the pullback $E_f(c)$ 
is a pair $(\sigma, c_0\to\cdots\to c_n\to c)$ where $\sigma\in K_n$ is 
such that $f(\sigma) = (c_0\to\cdots\to c_n)$. Hence the $\RR\calc$-module 
$C_n(E_f;\RR)$ is the direct sum of copies of $F_{c_n}^{\RR\calc}$, one for 
each  pair $(\sigma, c_0\to\cdots\to c_n)$ as above, hence is projective, 
and $\colim_\calc(C_n(E_f;\RR))\cong C_n(K;\RR)$. 
\end{proof}

We next define a generalized version of Quillen's plus construction, 
which plays a central role in this section. 

\begin{Defi} \label{d:plus}
Fix a commutative ring $\RR$, a connected CW complex $X$, and a normal 
subgroup $H\nsg\pi_1(X)$. An \emph{$\RR$-plus construction for $(X,H)$} 
consists of a CW complex $X^+_{\RR}$ together with a map $\kappa\:X\too 
X^+_{\RR}$, such that $\pi_1(\kappa)$ is surjective with kernel $H$, and 
$H_*(\kappa;N)$ is an isomorphism for each $\RR[\pi_1(X)/H]$-module $N$.
\end{Defi}

A different generalization of Quillen's plus construction, based on 
Bousfield localization with respect to a homology theory $h_*$, has been studied 
by Mislin and Peschke \cite{MP}, Jin-Yen Tai \cite{Tai}, and others. In the 
special case when $h_*=H_*(-;R)$ for a commutative ring $R$, Bousfield 
localization seems to be an example of a plus construction in our sense, 
although we have been unable to find references that prove this.

A few results about $\RR$-plus constructions are collected in the appendix. 
For example, we show there that $(X,H)$ has an $\RR$-plus construction if 
and only if $\chr(\RR)\ne0$ and $H$ is $\RR$-perfect, or $\chr(\RR)=0$ and 
$H$ is strongly $\RR$-perfect. 
(Recall that $H$ is strongly $\RR$-perfect if it is 
$\RR$-perfect and $\Tor(H_1(H;\Z),\RR)=0$.) 
Also, the $\RR$-completion of a 
space in the sense of Bousfield and Kan is an $\RR$-plus construction under 
certain hypotheses.

For $n\ge0$, let $\Delta^n$ denote the $n$-simplex as a simplicial 
set, and let $v_0,\dots,v_n$ be its vertices. For $0\le k\le n$, let 
$\Lambda^n_k\subseteq\Delta^n$ be the simplicial subset whose realization 
is the union of all proper (closed) faces in $\Delta^n$ containing $v_k$.  
Thus a Kan fibration is a simplicial map $f\:K\too L$ with the following 
lifting property: for each $0\le k\le n$, each $\sigma\:\Delta^n\too L$, and 
each $\tau\:\Lambda^n_k\too K$ such that 
$f\circ\tau=\sigma|_{\Lambda^n_k}$, there is a simplicial map 
$\til\sigma\:\Delta^n\too K$ such that $\til\sigma|_{\Lambda^n_k}=\tau$ and 
$f\circ\til\sigma=\sigma$. A \emph{Kan complex} is a simplicial set $K$ for 
which the (unique) map to $\Delta^0$ is a Kan fibration; equivalently, a 
simplicial set for which each simplicial map $\Lambda^n_k\too K$ extends to 
$\Delta^n$ (see \cite[\S\,I.3]{GJ}, \cite[Definition 1.12]{Curtis}, or 
\cite[\S\,IV.3]{GZ}). For 
example, for each space $X$, the singular simplicial set $S.(X)$ is a Kan 
complex \cite[Lemma I.3.3]{GJ}.

For any connected simplicial set $K$ with basepoint $x_0\in K_0$, let 
$\calp(K)=\calp(K,x_0)$ be the simplicial set of paths in $K$ based at 
$x_0$. Thus an $n$-simplex in $\calp(K)$ is a map of simplicial sets 
$\Delta^1\times\Delta^n\too K$ that sends $\{v_0\}\times\Delta^n$ to $x_0$ 
(more precisely, to the image of $x_0$ under the degeneracy map $K_0\too 
K_n$). Let $e=e_K\:\calp(K)\too K$ denote the path-loop fibration over $K$: 
the simplicial map that sends an $n$-simplex $\Delta^1\times\Delta^n\too K$ 
to the image of $\{v_1\}\times\Delta^n$. If $K$ is a Kan complex, then 
$e_K\:\calp(K)\too K$ is a Kan fibration and $|\calp(K)|$ is weakly 
contractible (see \cite[Lemma I.7.5]{GJ}). Thus the fibre of $e_K$ over 
$x_0$ is the loop simplicial set $\Omega(K,x_0)$ based at $x_0$ \cite[p. 
31]{GJ}. Using the fact that the realization of a Kan fibration is a Serre 
fibration (see \cite[Theorem I.10.10]{GJ}), one can show that 
$|\Omega(K,x_0)|$ is weakly equivalent to $\Omega(|K|,x_0)$.

If $f\:K\too L$ is a Kan fibration, and $\chi\:\5L\too L$ is an arbitrary 
simplicial map, then the pullback $\5f\:\5K\too \5L$ is defined levelwise: 
$\5K_n$ is the pullback (as a set) of $f_n\:K_n\too L_n$ along 
$\chi_n\:\5L_n\too L_n$. It is immediate from the definitions that $\5f$ is 
also a Kan fibration. By \cite[Theorem III.3.1]{GZ}, pullbacks commute with 
geometric realization; i.e., $|\5K|$ is the pullback of $|K|\too|L|$ along 
$|\5L|$. Note, however, that this requires that the pullbacks of 
realizations be taken in the category of compactly generated 
Hausdorff spaces (called ``Kelley spaces'' in \cite{GZ}).

\begin{Prop} \label{C*Enu}
Fix a group $\pi$ and a commutative ring $\RR$. Let $\ctg$ be a category 
over $\pi$, and set $H=\Ker[\pi_1(|\calc|)\xto{\pi_1(|\theta|)}\pi]$. 
Assume that $\kappa\:|\calc|\too|\calc|^+_{\RR}$ 
is an $\RR$-plus construction for $(|\calc|,H)$, and let 
$\5\kappa\:\caln(\calc)\too S.(|\calc|^+_{\RR})$ be the simplicial map 
adjoint to $\kappa$. Fix an object $c_0$ in $\calc$, regarded as a vertex 
in $\caln(\calc)$, set $x_0=\5\kappa(c_0)$, and let 
$e=e_{S.(|\calc|^+_{\RR})}$ be the path-loop fibration over 
$S.(|\calc|^+_{\RR})$ based at $x_0$. Let $\nu\:A\calc\too\caln(\calc)$ be 
the pullback of $e$ along $\5\kappa$, and let $\mu\:E_\nu\too E\calc$ 
denote the fibration of $\calc$-diagrams of simplicial sets obtained as the 
pullback of $\nu$ along $\eta$. We thus have, for each object $c$ in $\calc$, 
the following diagram of simplicial sets with pullback squares 
	\beqq \vcenter{\xymatrix@C=40pt@R=25pt{
	E_\nu(c) \ar[r] \ar[d]^{\mu_c} & A\calc \ar[d]^{\nu} \ar[r] & 
	\calp(S.(|\calc|^+_{\RR}),x_0) \ar[d]^{e} \\
	E\calc(c) \ar[r]^{\eta_c} & \caln(\calc) \ar[r]^{\5\kappa} 
	& S.(|\calc|^+_{\RR}) \rlap{\,.}
	}} \label{e:4.4pb}
	\eeqq
Then the following hold, where we regard the $\calc$-diagram $E_\nu$ as a 
$\theta\dn\Bobj[\pi]$-diagram via the forgetful functor 
$\theta\dn\Bobj[\pi]\too\calc$.
\begin{enumerate}[\rm(a) ]

\item For each $n\geq 0$, $C_n(E_\nu;\RR)$ is a projective $\RR\calc$-module. 
\label{C*Enu-1}

\item The complex $\theta_*(C_*(E_\nu;\RR))\cong 
\colim_{\theta\dn\Bobj[\pi]}(C_*(E_\nu;\RR))$ is acyclic, and 
$\gee$ induces an isomorphism $H_0\bigl(\theta_*(C_*(E_\nu;\RR))\bigr)\cong\RR\pi$. 
\label{C*Enu-2}

\item For each $n\ge0$, $H_n(E_\nu;\RR)$ is locally constant on $\calc$ and 
essentially constant on $\theta\dn\Bobj[\pi]$, and hence 
$H_n(E_\nu;\RR)\cong\theta^*(\5H_n)$ for some $\RR\pi$-module $N_n$. 
\label{C*Enu-3}

\item For each object $c$ in $\calc$, 
$|E_\nu(c)|$ is weakly equivalent to $\Omega(|\calc|^+_{\RR})$. 
\label{C*Enu-4}

\end{enumerate}
In particular, by \eqref{C*Enu-1}--\eqref{C*Enu-3}, $C_*(E_\nu;\RR)$ is an 
$\Omega$-resolution of $\RR\pi$ with respect to $\ctg$.
\end{Prop}

\begin{proof} We write $C_*(-)=C_*(-;\RR)$ and $H_*(-)=H_*(-;\RR)$ for 
short, and refer to diagram \eqref{e:4.4pb}, where by construction, 
$\mu_c$, $\nu$, and $e$ are all Kan fibrations with fibre 
$\Omega(S.(|\calc|^+_{\RR})) \cong S.(\Omega(|\calc|^+_{\RR}))$. Then 
$A\calc$ is $\RR$-acyclic since $\calp(S.(|\calc|^+_{\RR}),x_0)$ is 
contractible and $H_*(\5\kappa;N)$ is an isomorphism for each 
$\RR\pi$-module $N$. Point (\ref{C*Enu-1}) follows from Lemma \ref{C*(Ef) 
projective}, applied with $A\calc$ and $\nu$ in the roles of $K$ and $f$, 
and point (\ref{C*Enu-4}) holds since each $E\calc(c)$ is the nerve of 
a category with final object and hence contractible. 

Let $\sigma\:\theta\dn\Bobj[\pi]\too\calc$ be the forgetful functor, and 
consider the following cubical diagram (for each object $(c,g)$ in 
$\theta\dn\Bobj[\pi]$):
	\[ \xymatrix@C=20pt@R=8pt{
	E_{\til\nu}(c,g) \ar[rrr]_(.58){\cong} \ar[dr] \ar[ddd] &&& E_\nu(c) 
	\ar[dr] \ar[ddd]^(.55){\mu_c}|!{[dll];[dr]}\hole \\
	& \til{A\calc} \ar[rrr] \ar[ddd]^(.35){\til\nu} &&& A\calc \ar[ddd]^{\nu} \\ 
	\\
	E(\theta\dn\Bobj[\pi])(c,g) \ar[dr] 
	\ar[rrr]^(.65){E\sigma_{(c,g)}}_(.6){\cong} |!{[uur];[dr]}\hole 
	&&& E\calc(c) \ar[dr]^{\eta_c} \\
	& \caln(\theta\dn\Bobj[\pi]) \ar[rrr]^-{\caln(\sigma)} &&& \caln(\calc) 
	} \]
Here, $\til{A\calc}$, $\til\nu$, and $E_{\til\nu}$ are defined so that all 
of the ``vertical'' squares in this diagram are pullbacks. (Note that the 
bottom square, and hence also the top square, need not be pullbacks.)
Also, $E\sigma_{(c,g)}$ is an isomorphism of simplicial 
sets since for each morphism $\varphi\in\Mor_\calc(c',c)$ and each 
$g\in\pi$, there is a unique $g'\in\pi$ such that 
$\varphi\in\Mor_{\theta\dn\Bobj[\pi]}((c',g'),(c,g))$. 
Hence $E_{\til\nu}(c,g)\cong E_\nu(c)$. 
So by Lemma \ref{isos-over-G}(b), and Lemma \ref{C*(Ef) projective}
applied with $\theta\dn\Bobj[\pi]$ in the role of $\calc$, 
	\[ H_*\bigl(\theta_*(C_*(E_\nu))\bigr) \cong 
	H_*\bigl(\colim_{\theta\dn\Bobj[\pi]}\sigma^*(C_*(E_{\nu}))\bigr)
	\cong H_*\bigl(\colim_{\theta\dn\Bobj[\pi]}C_*(E_{\til\nu})\bigr)
	\cong H_*\bigl(C_*(\til{A\calc})\bigr) 
	\cong H_*(\til{A\calc})\,. \]
But $|\theta\dn\Bobj[\pi]|$ is the covering space of $|\calc|$ with 
fundamental group $H$ and covering group $\pi$ (Lemma 
\ref{isos-over-G}(a)), the image of $\pi_1(|A\calc|)$ in $\pi_1(|\calc|)$ 
is contained in $H=\Ker(\pi_1(|\kappa|))$ since it vanishes in 
$\pi_1(|\calc|^+_{\RR})$, and hence $|\til{A\calc}|\cong\pi\times|A\calc|$. 
Since $|A\calc|$ is $\RR$-acyclic, this proves \eqref{C*Enu-2}: 
$\theta_*(C_*(E_\nu))$ is acyclic and 
$H_0(\theta_*(C_*(E_\nu)))\cong\RR\pi$.

For each object $c$ in $\calc$, let $F(c)=\nu^{-1}(c)$ be the fibre of 
$\nu$ over the vertex $c$ in $\caln(\calc)$. Via homotopy lifting, this is 
extended to a homotopy functor $F$ from $\calc$ to simplicial sets, and 
this in turn defines a locally constant graded $\RR\calc$-module 
$M_*=H_*(F)$. For each $c$ in $\calc$, the action of $\pi_1(|\calc|,c)$ on 
$M_*(c)=H_*(F(c))$ described in Lemma \ref{ess.const.}(a) is the usual 
action of the fundamental group of the base on the homology of a fibre, and 
since $\nu$ is a pullback of $e$, this action factors through 
$\pi_1(|\calc|^+_{\RR})\cong\pi$. So $M_*$ is essentially constant on 
$\theta\dn\Bobj[\pi]$ by Lemma \ref{ess.const.}(b). Also, since each 
$E\calc(c)$ contracts to the vertex $(c,\Id_c)$ in a natural way, where 
$\eta_c(c,\Id_c)=c$, we have homotopy equivalences $E_\nu(c)\simeq F(c)$ 
natural in $\calc$ up to homotopy. So $H_*(E_\nu)\cong M_*$ as 
$\RR\calc$-modules. Hence by Lemma \ref{isos-over-G}(c), 
$H_*(E_\nu)\cong\theta^*(N_*)$ for some graded $\RR\pi$-module $N_*$, 
finishing the proof of \eqref{C*Enu-3}.


Since $\theta_*$ is right exact, 
	\[ N_0 \cong \theta_*\theta^*(N_0) \cong \theta_*(H_0(E_\nu)) 
	\cong H_0\bigl(\theta_*(C_*(E_\nu))\bigr) \cong \RR\pi, \]
and so $H_0(E_\nu)\cong\theta^*(\RR\pi)$. This defines a surjective 
homomorphism $\gee\:C_0(E_\nu)\too\theta^*(\RR\pi)$, and finishes the proof 
that $(C_*(E_\nu),\2_*)\too\theta^*(\RR\pi)\too0$ is an 
$\Omega$-resolution of $\RR\pi$. 
\end{proof}

Upon combining Proposition \ref{C*Enu} with Proposition \ref{plus}, we get the 
following theorem.

\begin{Thm} \label{Omega_C0}
Fix a group $\pi$ and a commutative ring $\RR$. Let $(\calc,\theta)$ be a category 
over $\pi$, set $H=\Ker[\pi_1(|\calc|)\xto{\pi_1(|\theta|)}\pi]$, and 
assume that $\chr(\RR)\ne0$ and $H$ is $\RR$-perfect, or $\chr(\RR)=0$ and 
$H$ is strongly $\RR$-perfect. Then 
\begin{enuma} 

\item $(|\calc|,H)$ admits an $\RR$-plus construction; 

\item the free $\RR\pi$-module $\RR\pi$ has an $\Omega$-resolution with 
respect to $(\calc,\theta)$; and 

\item for each $\RR$-plus construction $|\calc|^+_{\RR}$ for $(|\calc|,H)$, 
	$ H^\Omega_*(\calc,\theta;\RR) \cong 
	H_*(\Omega(|\calc|^+_{\RR});\RR)$.

\end{enuma}
\end{Thm}

\begin{proof} By Proposition \ref{plus}, $(|\calc|,H)$ admits an $\RR$-plus 
construction. Fix such a space $|\calc|^+_{\RR}$, and let $E_\nu$ be the 
functor from $\calc$ to simplicial sets constructed as a pullback in diagram 
\eqref{e:4.4pb} of Proposition \ref{C*Enu}. 

By Proposition \ref{C*Enu}, 
$C_*(E_\nu;\RR)$ is an $\Omega$-resolution of $\RR\pi$ with respect to 
$(\calc,\theta)$, and also $H_*(E_\nu;\RR)\cong\theta^*(N_*)$ for some graded 
$\RR\pi$-module $N_*$. By point (d) in the same proposition, for each $c$ in 
$\calc$, $|E_\nu(c)|$ is weakly equivalent to $\Omega(|\calc|^+_{\RR})$ 
and hence 
	\beq \begin{split} 
	H^\Omega_*(\calc,\theta;\RR) \defeq \theta_*(H_*(E_\nu;\RR)) 
	&\cong \theta_*\theta^*(N_*) \cong N_* \\ & \cong \theta^*(N_*)(c) 
	\cong H_*(|E_\nu(c)|;\RR)  \cong 
	H_*(\Omega(|\calc|^+_{\RR});\RR). \qedhere
	\end{split} \eeq
\end{proof}



In the special case where $\pi=\pi_1(|\calc|)$, this takes the 
form:

\begin{Cor} \label{Omega_no+}
Let $\calc$ be a small, connected category, and set $\pi=\pi_1(|\calc|)$. 
Then there is a functor $\theta\:\calc\too\calb(\pi)$ such that 
$\pi_1(|\theta|)$ is an isomorphism. For such $\theta$, and for any 
commutative ring $\RR$, the $4$-tuple 
$\cpth$ is an $\Omega$-system, the free module $\RR\pi$ has an 
$\Omega$-resolution with respect to $\cpth$, and 
	\[ H^\Omega_*(\calc,\theta;\RR) \cong H_*(\Omega(|\calc|);\RR). \]
\end{Cor}

The $\RR$-plus construction of $(\caln(\calc),H)$ as defined in Definition 
\ref{d:plus} is not in general unique, not even up to homotopy. 
However, in certain cases, we can choose it to be a completion or a 
fibrewise completion of $|\calc|$ in the sense of Bousfield and Kan.
Recall \cite[III.5.1]{BK} that for $\RR\subseteq\Q$, a group $\pi$ is 
$\RR$-nilpotent if it has a central series for which each quotient is an 
$\RR$-module.

\begin{Thm} \label{Omega_C}
Let $(\calc,\theta)$ be a category over a group $\pi$, and set 
$H=\Ker[\pi_1(|\calc|)\xto{\pi_1(|\theta|)}\pi]$.
\begin{enuma} 

\item Assume that $\RR$ is a subring of $\Q$ or $\RR=\F_p$ for some 
prime $p$, and that $H$ is $\RR$-perfect. Let $|\calc|^{\wedge}$ be 
the fibrewise $\RR$-completion of $|\calc|$ over $B\pi$. Then 
	\[ H^\Omega_*(\calc,\theta;\RR) \cong H_*(\Omega(|\calc|^{\wedge});\RR). \]

\item If $\RR\subseteq\Q$ is such that $H$ is $\RR$-perfect, and $\pi$ is 
$\RR$-nilpotent with nilpotent action on $H_i(|\theta\dn\Bobj[\pi]|;\RR)$ 
for each $i$, then 
	\[ H^\Omega_*(\calc,\theta;\RR) \cong 
	H_*(\Omega(|\calc|\Rcom);\RR), \]
where $|\calc|\Rcom$ is the $\RR$-completion of $|\calc|$.

\item If for some prime $p$, $\kk$ is a field of characteristic $p$, 
$\pi$ is a finite $p$-group, and $H$ is $p$-perfect, then 
	\[ H^\Omega_*(\calc,\theta;\kk) \cong H_*(\Omega(|\calc|\pcom);\kk) \]
where $|\calc|\pcom$ is the $p$-completion of $|\calc|$.

\end{enuma}
\end{Thm}

\begin{proof} By Lemma \ref{|X|pcom->Bpi}, the natural map from $|\calc|$ 
to $|\calc|^\wedge$, $|\calc|\Rcom$, or $|\calc|\pcom$ is an $\RR$- or 
$\kk$-plus construction for $(|\calc|,H)$ under the hypotheses of (a), (b), 
or (c), respectively. So this theorem follows as a special case of Theorem 
\ref{Omega_C0}(c).
\end{proof}

The following corollary includes the case proven by Benson in \cite{B}. 
Note that when $G$ is a finite group, its quotient by the maximal normal 
$p$-perfect subgroup is always a $p$-group.

\begin{Cor}\label{Omega_G}
Fix a prime $p$. 
Let $G$ be a (possibly infinite) discrete group, and let $O^p(G)$ be the 
maximal normal $p$-perfect subgroup of $G$. Set $\pi=G/O^p(G)$, let 
$\chi\:G\too\pi$ be the natural surjection, and assume 
that $\pi$ is a finite $p$-group. Then for each field $\kk$ of 
characteristic $p$, $H_*(\Omega(BG\pcom);\kk)\cong 
H^\Omega_*(\calb(G),\calb(\chi);\kk)$: the homology of an $\Omega$-resolution of 
$\kk\pi$ with respect to $(\calb(G),\calb(\chi))$ as a category over $\pi$.
\end{Cor}

\begin{proof} This is just Theorem \ref{Omega_C}(c) when $\calc=\calb(G)$.
\end{proof}

The results in this section lead in a natural way to the following 
question.

\begin{Quest} 
Are there more general conditions on an $\Omega$-system $\Osys\cala\calb$ 
and $X\in\scrp(\calb)$ under which $H^\Omega_*(\cala,\calb;X)$, or a 
functorial image, describes the homology of a space (e.g., of a loop 
space)? In particular, can the homology of the $\Omega$-resolutions of 
Proposition \ref{ex3}(b) be realized as the homology of some space 
determined by the $\Omega$-systems?
\end{Quest}


\section{Examples: $\Omega$-resolutions for some $p$-local compact groups}
\label{s:examples}

One problem that motivated this work was that of finding a way to 
characterize the $p$-compact groups among the more general $p$-local 
compact groups. As already noted in the introduction, we did not succeed in 
doing so. The aim of this section is to give some very simple examples that 
demonstrate how complicated this problem can be, for example, by analyzing 
some $p$-local compact groups that are not $p$-compact. We also give some 
results, and one explicit computation, that follow from knowing that 
$\Omega$-resolutions determine the homology of loop spaces without 
having to explicitly construct the resolutions themselves.

Throughout this section, we fix a prime $p$ and a field $k$ of 
characteristic $p$. We first recall some definitions. A \emph{$p$-compact 
group} consists of a loop space $X$ and its classifying space $BX$, such 
that $X\simeq\Omega(BX)$, $H_*(X;\F_p)$ is finite (in particular, 
$H_n(X;\F_p)=0$ for $n$ large enough), and $BX$ is $p$-complete. This 
concept was first introduced by Dwyer and Wilkerson \cite{DW}, and 
developed by them and others in several papers. If $G$ is a compact Lie 
group whose group of components $\pi_0(G)$ is a $p$-group, then 
$\Omega(BG\pcom)$ is a $p$-compact group, but this need not be the case if 
$\pi_0(G)$ is not a $p$-group. Every $p$-compact group contains a maximal 
torus with properties very similar to those of maximal tori in compact Lie 
groups. 

A \emph{$p$-local compact group} consists of a discrete $p$-toral group 
$S$ (i.e., an extension of a discrete $p$-torus $(\Z/p^\infty)^r$ for some 
$r\ge0$ by a finite $p$-group), together with a fusion system $\calf$ over 
$S$ and a linking system $\call$ associated to $\calf$. We refer to 
\cite[Definitions 2.2 and 4.1]{BLO3} for the precise definitions of fusion 
and linking systems in this context; here, we just note that $\calf$ and 
$\call$ are categories, $\Ob(\calf)$ is the set of subgroups of $S$, each 
morphism in $\calf$ is a homomorphism between subgroups, and there is a 
functor $\call\too\calf$ that is an inclusion on objects and surjective on 
each morphism set. The classifying space of such a triple $\SFL$ is the 
$p$-completed space $|\call|\pcom$. By \cite[\S\S\,9--10]{BLO3}, each 
compact Lie group $G$ or $p$-compact group $X$ has a maximal discrete 
$p$-toral subgroup $S$ (unique up to conjugacy), together with a fusion 
system $\calf$ and a linking system $\call$ such that $|\call|\pcom$ is 
homotopy equivalent to $BG\pcom$ or $BX$, respectively.

By \cite[Proposition 4.4]{BLO3}, for each $p$-local compact group $\SFL$, 
the fundamental group of the classifying space $|\call|\pcom$ is a finite 
$p$-group. So as a special case of Theorem \ref{Omega_C}(c), we get:

\begin{Thm} \label{Omega_L}
Let $\SFL$ be a $p$-local compact group, and set $\pi=\pi_1(|\call|\pcom)$. 
Then there is $\theta\:\call\too\calb(\pi)$ such that $(\call,\theta)$ is a 
category over $\pi$ and 
	\[ H_*(\Omega(|\call|\pcom);\F_p) \cong H_*^\Omega(\call,\theta;\F_p). 
	\]
\end{Thm}

If $\Gamma$ is an extension of a discrete $p$-torus by a finite $p$-group, 
then $B\Gamma\pcom$ is the classifying space of a $p$-compact group. In 
contrast, if $\Gamma$ is an extension of a discrete $p$-torus by an 
arbitrary finite group, then $B\Gamma\pcom$ need not be the classifying 
space of a $p$-compact group (nor the $p$-completion of $BG$ for a compact 
Lie group $G$), but it is always the classifying space of a $p$-local 
compact group. For example, if $p$ is an odd prime and $r\ge2$, and 
$\Gamma=(\Z/p^\infty)^r\rtimes C_2$ where $C_2$ acts by inverting 
all elements of $(\Z/p^\infty)^r$, 
then $\Omega(B\Gamma\pcom)$ is not a $p$-compact group since its mod $p$ 
homology is nonvanishing in arbitrarily large degrees (see Example 
\ref{ex:T2:2} for the case $r=2$). 

What we want to do now is to give some explicit examples of such 
$\Omega$-resolutions. We focus on $p$-local compact groups associated to 
extensions of discrete $p$-tori by finite groups, especially by those of 
order prime to $p$.

\begin{Prop} \label{Osys-T.H}
Let $T\nsg\Gamma$ be a pair of groups such that $T\cong(\Z/p^\infty)^r$ 
for some $r\ge1$ and $\Gamma/T$ is finite. Let 
$O^p(\Gamma)\nsg\Gamma$ be the smallest normal subgroup containing $T$ 
and of $p$-power index in $\Gamma$, and set $\pi=\Gamma/O^p(\Gamma)$. 
Then the following hold.
\begin{enuma} 

\item The subgroup $O^p(\Gamma)$ is $p$-perfect, the 
spaces $BT\pcom$, $BO^p(\Gamma)\pcom$, and 
$B\Gamma\pcom$ are all $p$-complete, and $BT\pcom\simeq K(\Z_p,2)^r$ 
and $BO^p(\Gamma)\pcom$ are simply connected. The sequence 
$BO^p(\Gamma)\pcom\too B\Gamma\pcom \too B\pi$ is a homotopy fibration 
sequence, and so $\pi_1(B\Gamma\pcom)\cong\pi$.

\item There is a $p$-local compact group $\SFL$ associated to $\Gamma$, 
where $T\le S\le\Gamma$ and $S/T\in\sylp{\Gamma/T}$, and where 
$|\call|\pcom\simeq B\Gamma\pcom$. 

\item If $\Omega(B\Gamma\pcom)$ is a $p$-compact group, then 
$O^p(\Gamma/T)$ has order prime to $p$.

\end{enuma}
\end{Prop}

\begin{proof} \textbf{(a) } A subgroup $H\le\Gamma$ containing $T$ 
has $p$-power index in $\Gamma$ if and only if $H/T$ has $p$-power 
index in the finite group $\Gamma/T$. So $O^p(\Gamma)/T=O^p(\Gamma/T)$, 
and in particular, $O^p(\Gamma)/T$ is $p$-perfect. Since $T$ is also 
$p$-perfect (being $p$-divisible), $O^p(\Gamma)$ is $p$-perfect.

Since $T$ and $O^p(\Gamma)$ are $p$-perfect, $BT\pcom$ and 
$BO^p(\Gamma)\pcom$ are $p$-complete and simply connected by 
\cite[Proposition VII.3.2]{BK}. Also, $B(\Z/p^\infty)\pcom\simeq 
K(\Z_p,2)$ 
by \cite[VI.2.1--2.2]{BK}. 

Now, $BO^p(\Gamma)\pcom\too B\Gamma\pcom \too B\pi$ is a homotopy 
fibration sequence by \cite[Example II.5.2(iv)]{BK} (applied with 
$R=\F_p$) and since $\pi$ is a finite $p$-group, and so 
$\pi_1(B\Gamma\pcom)\cong\pi$. By the same argument applied to the 
completed sequence, $(B\Gamma\pcom)\pcom\simeq B\Gamma\pcom$, and so 
$B\Gamma\pcom$ is $p$-complete.

\smallskip

\noindent\textbf{(b) } Embed $T$ in $\GL_r(\C)$ as the subgroup of 
diagonal matrices of $p$-power order. Then via induction, $\Gamma$ 
embeds as a subgroup of $\GL_{r\cdot|\Gamma/T|}(\C)$, and hence is a 
linear torsion group in the sense of \cite[\S\,8]{BLO3}. So by 
\cite[Theorem 8.10]{BLO3}, it has an associated $p$-local compact group 
$(S,\calf,\call)$, where $T\le S\le\Gamma$, $S/T\in\sylp{\Gamma/T}$, 
and $B\Gamma\pcom\simeq|\call|\pcom$. Also, $\calf=\calf_S(\Gamma)$: 
the fusion system over $S$ whose morphisms are those homomorphisms 
between subgroups of $S$ induced by conjugation in $\Gamma$.

\smallskip

\noindent\textbf{(c) } Now assume that $\Omega(B\Gamma\pcom)$ is a 
$p$-compact group; i.e., that $H_*(\Omega(B\Gamma\pcom);\F_p)$ is 
finite. Let $O^p(\Gamma)\nsg\Gamma$ and $\pi=\Gamma/O^p(\Gamma)$ be 
as in (a). Then $BO^p(\Gamma)\pcom$ is the homotopy fibre 
of the natural map $B\Gamma\pcom\too B\pi$, and hence 
is equivalent to the covering space of $B\Gamma\pcom$ with covering 
group $\pi$. So $\Omega(BO^p(\Gamma)\pcom)$ also has finite mod $p$ 
homology. We can thus assume that $\Gamma=O^p(\Gamma)$, and hence is $p$-perfect 
by (a). In particular, $\Omega(B\Gamma\pcom)$ is connected. If we now show 
that $S=T$, then $\Gamma/T$ has order prime to $p$ since 
$S/T\in\sylp{\Gamma/T}$.

For a finite $p$-group $Q$, set $\Rep(Q,\call)=\Hom(Q,S)/{\sim}$, where 
$\rho_1\sim\rho_2$ if $\rho_1=\alpha\rho_2$ for some 
$\alpha\in\Iso_\calf(\rho_2(Q),\rho_1(Q))$. In other words, it is the set 
of $\Gamma$-conjugacy classes in $\Hom(Q,S)$. Let $[BQ,B\Gamma\pcom]$ 
be the set of homotopy classes of unpointed maps $BQ\too 
B\Gamma\pcom$. By \cite[Theorem 6.3(a)]{BLO3}, there is a bijection 
$\Rep(Q,\call)\too[BQ,B\Gamma\pcom]$ that sends the conjugacy class of a 
homomorphism $\rho$ to the homotopy class of $B\rho$. 

By \cite[Proposition 5.6]{DW} and since $\Omega(B\Gamma\pcom)$ is 
connected, for each $n\ge1$ and each $f\:BC_{p^n}\too B\Gamma\pcom$, $f$ 
extends (up to homotopy) to a map from $BC_{p^{n+1}}$ to $B\Gamma\pcom$. 
Hence each $\rho\in\Hom(C_{p^n},S)$ extends, up to 
$\Gamma$-conjugacy, to some $\4\rho\in\Hom(C_{p^{n+1}},S)$. Since 
no element of $S\sminus T$ is infinitely $p$-divisible (and they all 
have $p$-power order), this shows that $S=T$, and thus that 
$\Gamma/T$ has order prime to $p$.
\end{proof}

In fact, whenever $T\nsg\Gamma$ are such that $T$ is a discrete $p$-torus 
and $\Gamma/T$ is finite, $\Omega(B\Gamma\pcom)$ is a $p$-compact group if 
and only if $O^p(\Gamma/T)$ has order prime to $p$ and 
$\Aut_{O^p(\Gamma/T)}(T)$ is generated by pseudoreflections on $T$. The 
necessity of this last condition was shown by Dwyer and Wilkerson 
\cite[Theorem 9.7.ii]{DW}. Conversely, Clark and Ewing \cite[Corollary, p. 
426]{CE} showed that if $O^p(\Gamma/T)$ has order prime to $p$ and its 
action is generated by pseudoreflections, then $H^*(B\Gamma\pcom;\F_p)$ is 
a polynomial algebra over $\F_p$, and hence the (co)homology of its loop 
space is finite. 

\begin{Rmk} \label{Osys-T.H2}
Assume $T\nsg\Gamma$ are as in Proposition \ref{Osys-T.H}, where in 
addition, $p\nmid|\Gamma/T|$ and the conjugation action of $\Gamma/T$ on 
$T$ is faithful (i.e., $C_\Gamma(T)=T$). Then $S=T$, $\Ob(\call)=\{T\}$, 
and $\Aut_\call(T)=\Gamma$, so that $\call\cong\calb(\Gamma)$, 
$\pi_1(|\call|)\cong\Gamma$, and $\pi_1(|\call|\pcom)=1$ since $\Gamma$ is 
$p$-perfect. Thus by Proposition \ref{ex4} and Theorem \ref{Omega_C}(b), the 
$\Omega$-system associated to $\SFL$ is $\Osys{k\Gamma\mod}{k\mod}$, where 
$\theta_*(M)=\colim_{\Gamma}(M)$ for a $k\Gamma$-module $M$ and $\theta^*$ 
sends a $k$-module $N$ to the corresponding $k\Gamma$-module with trivial 
action; and $H_*(\Omega(|\call|\pcom);k)$ is the homology of an 
$\Omega$-resolution of $k$.
\end{Rmk}

Note also, in the situation of Remark \ref{Osys-T.H2}, that since 
$|\Gamma/T|$ has order prime to $p$, the group $T$ is uniquely 
$|\Gamma/T|$-divisible. Hence $H^i(\Gamma/T;T)=0$ for all $i>0$, and 
$\Gamma$ must be a semidirect product: $\Gamma\cong T\rtimes H$ where 
$H\cong\Gamma/T$.

We also note the following:

\begin{Rmk} \label{Q12b}
Let $\Gamma$ be a linear torsion group: a subgroup of $GL_n(K)$, for some 
field $K$ of characteristic different from $p$, all of whose elements have 
finite order. By \cite[Theorem 8.10]{BLO3}, there is a $p$-local compact 
group $(S,\calf,\call)$, where $S\le\Gamma$ is a maximal discrete $p$-toral 
subgroup and $|\call|\pcom\simeq B\Gamma\pcom$. Set 
$\pi=\pi_1(B\Gamma\pcom)\cong\pi_1(|\call|\pcom)$ (a finite $p$-group by 
\cite[Proposition 4.4]{BLO3}), and let 
$\theta\:\calb(\Gamma)\too\calb(\pi)$ and $\eta\:\call\too\calb(\pi)$ be 
functors that induce these isomorphisms. By Theorem \ref{Omega_C}(c), 
$H^\Omega_*(\calb(\Gamma),\theta;k)\cong H^\Omega_*(\call,\eta;k)$; i.e., 
$\Omega$-resolutions with respect to these two different $\Omega$-systems 
have the same homology. 
\end{Rmk}

Throughout the rest of the section, whenever $\Gamma$ and 
$\pi=\Gamma/O^p(\Gamma)$ are as in Proposition \ref{Osys-T.H}, we write 
``$\Omega$-resolution of $k\pi$ with respect to $\Gamma$'' to mean an 
$\Omega$-resolution of $k\pi$ with respect to the category 
$(\calb(\Gamma),\theta)$ over $\pi$ or the $\Omega$-system 
$\Osys{k\Gamma\mod}{k\pi\mod}$, where $\theta\:\calb(\Gamma)\too\calb(\pi)$ 
is the natural projection. 

\subsection{$\Omega$-resolutions with respect to discrete $p$-tori}
\label{s:tori}
\leavevmode

Let $T\nsg\Gamma$ be a pair of groups, where $T\cong(\Z/p^\infty)^r$ is a 
discrete $p$-torus of rank $r\ge1$ and $\Gamma/T$ is finite of order prime 
to $p$. Thus $\Gamma=T\rtimes H$ for some finite subgroup $H\le\Gamma$ of 
order prime to $p$ (see the paragraph after Remark \ref{Osys-T.H2}). We regard the group ring $kT$ as a left 
$k\Gamma$-module, where for $t\in T$, $h\in H$, and $x\in kT$, $t(x)=tx$ 
and $h(x)=hxh^{-1}$. We will construct complexes of projective 
$k\Gamma$-modules which, as complexes of $kT$-modules, are 
$\Omega$-resolutions of $k$ with respect to $T$. The $k\Gamma$-module 
structure on these complexes will be used in the next two subsections.

Set $V=\Omega_1(T)\cong(C_p)^r$. For each $n\ge1$, set 
$T_n=\Omega_n(T)\cong(C_{p^n})^r$ (thus $V=T_1$), and regard $kT_n$ as a 
subring of $kT$. Let $I(kT)\le kT$ and $I(kT_n)\le kT_n$ be the 
augmentation ideals. 

For each $n\ge1$, and each $kT_n$-module $M$ and proper submodule $M_0<M$, 
$M/M_0$ has a nontrivial quotient module with trivial $T_n$-action. 
Hence $I(kT_n)\cdot(M/M_0)<M/M_0$, and so $M_0+I(kT_n)\cdot M<M$. 

For each $n\ge0$, let $\4\varphi_n\:V\too I(kT_n)/I(kT_n)^2$ be 
the map $\4\varphi_n(t)=[t-1]$. This is a homomorphism of groups, and extends 
to a $kH$-linear isomorphism $k\otimes_{\F_p}V\cong 
I(kT_n)/I(kT_n)^2$. Lift $\4\varphi_n$ to an $\F_pH$-linear homomorphism 
$\til\varphi_n\:V\too I(kT_n)$ (the ring $\F_pH$ is semisimple since 
$p\nmid|H|$), and extend that to a $k\Gamma$-linear homomorphism 
	\[ \varphi_n\: kT\cdot V \defeq kT \otimes_{\F_p} V \Right5{} kT \]
by setting $\varphi_n(\xi\otimes v)=\xi\cdot\til\varphi_n(v)$ for $\xi\in 
kT$ and $v\in V$. Since $\varphi_n$ induces an isomorphism $k\otimes V\cong 
I(kT_n)/I(kT_n)^2$ by assumption, $\varphi_n(kT_n\cdot 
V)+I(kT_n)^2=I(kT_n)$, and so 
	\beqq \varphi_n(kT_n\cdot V)=I(kT_n) \label{e:H<Gamma-1} \eeqq
by the last paragraph (applied with $\varphi_n(kT_n\cdot V)$ and 
$I(kT_n)$ in the roles of $M_0$ and $M$). Hence $\Im(\varphi_n)=kT\cdot 
I(kT_n)$. 

In particular, for each $n\ge1$, $\Im(\varphi_n)\le 
I(kT)\cdot\Im(\varphi_{n+1})$. Since $kT\cdot V$ is projective as a 
$k\Gamma$-module, there is a $k\Gamma$-linear homomorphism 
$\psi_n\:kT\cdot V\too kT\cdot V$ such that 
	\beqq \varphi_{n+1}\circ\psi_n=\varphi_n. 
	\label{e:H<Gamma-2} \eeqq
Note that 
	\beqq \psi_n(kT\cdot V)\le I(kT)\cdot V. \label{e:Im(psi)<IV} \eeqq

Let $\Lambda^m_R(M)$ denote the $m$-th exterior power over a commutative 
ring $R$ of an $R$-module $M$. Define, for each $0\le m\le r$ and each $n\ge1$, 
	\[ D_m = \Lambda^m_{kT}\bigl( kT\cdot V\bigr)
	\cong kT \otimes_{\F_p} \Lambda^m_{\F_p}(V) \,, \]
regarded as a $k\Gamma$-module. In particular, $D_0=kT$. 
For each $n\ge1$ and each $1\le m\le r$, define a boundary map 
$\2_m^{(n)}\:D_m\too D_{m-1}$ by setting 
	\[ \2_m^{(n)}\bigl(v_1\wedge v_2\wedge\cdots\wedge v_m\bigr) = 
	\sum_{i=1}^m(-1)^{i-1}\varphi_n(v_i)\cdot v_1\wedge\cdots 
	\5{v_i}\cdots\wedge v_m \]
for $v_1,\dots,v_m\in kT\cdot V$. Then 
	\[ \bfD^{(n)} \defeq (D_*,\2_*^{(n)}) = 
	\Bigl( 0 \too D_r \Right2{\2_r^{(n)}} D_{r-1} 
	\Right2{} \cdots \Right2{} D_1 \Right2{\2_1^{(n)}} kT 
	\too 0 \Bigr) \]
is a chain complex of projective $k\Gamma$-modules, and 
$\bigl\{\Lambda^m(\psi_n)\bigr\}_{m=0}^r$ defines a morphism of chain 
complexes $\Psi^{(n)}\:\bfD^{(n)}\too\bfD^{(n+1)}$ where 
$\Lambda^0(\psi_n)=\Id_{kT}$. 

Fix $n\ge1$, $0\le m\le r$, and $x\in 
D_m=\Lambda^m_{kT}(kT\cdot V)$ such that $\2_m^{(n)}(x)=0$. 
For each $v\in V$, 
	\[ \2_{m+1}^{(n)}(x\wedge v) = \2_m^{(n)}(x)\wedge v + (-1)^m 
	\varphi_n(v)\cdot x = (-1)^m \varphi_n(v)\cdot x. \]
Thus $\varphi_n(v)\cdot[x]=0$ in $H_m(\bfD^{(n)})$ for each $v\in V$, so 
$\Im(\varphi_n)$ annihilates $H_*(\bfD^{(n)})$. Since $\Im(\varphi_n)\ge 
I(kT_n)$, we now conclude that 
	\beqq \textup{for each $n\ge1$, $T_n$ acts trivially on 
	the homology of $\bfD^{(n)}$.} 
	\label{e:I.H*=0} \eeqq

As usual, whenever $\psi_*\:(C_*,\2_*)\too(C'_*,\2'_*)$ is a morphism of 
chain complexes, the \emph{mapping cone} of $\psi$ is the chain complex 
	\[ C_\psi = \Bigl( \cdots 
	\Right6{\mxtwo{\2'_4}{-\psi_3}0{\2_3}} C'_3\oplus C_2 
	\Right5{\mxtwo{\2'_3}{\psi_2}0{\2_2}} C'_2\oplus C_1 
	\Right6{\mxtwo{\2'_2}{-\psi_1}0{\2_1}} C'_1\oplus C_0 
	\Right4{(\2'_1,\psi_0)} C'_0 \Bigr). \]
The signs are chosen so that $C'_*\le C_\psi$ as chain complexes, and 
$C_\psi/C'_*\cong\Sigma C_*$. See \cite[\S\,1.5]{Weibel} for more details. 

Let $\bfD$ be the mapping cone of the chain map
	\[ \Psi = \bigl(\Id-\oplus\Psi^{(n)}\bigr) \: 
	\bigoplus_{n=1}^\infty \bfD^{(n)} \Right7{} 
	\bigoplus_{n=1}^\infty \bfD^{(n)}  \]
(see \cite[\S\,1.5]{Weibel}). More explicitly, 
	\[ \Psi(x_1,x_2,x_3,\dots)=(x_1,x_2-\Psi^{(1)}(x_1), 
	x_3-\Psi^{(2)}(x_2),\dots). \]
Since $\Psi$ is injective, $H_*(\bfD)\cong H_*(\Coker(\Psi))$, where 
$\Coker(\Psi)\cong\colim(\bfD^{(n)},\Psi^{(n)})$. So 
	\[ H_*(\bfD) \cong \colim \bigl( H_*(\bfD^{(1)}) 
	\Right4{H_*(\Psi^{(1)})} H_*(\bfD^{(2)}) 
	\Right4{H_*(\Psi^{(2)})} 
	H_*(\bfD^{(3)}) \Right3{} \cdots \bigr) \,, \]
since colimits are exact. Also, $T$ acts trivially on $H_*(\bfD)$ by 
\eqref{e:I.H*=0}.

Now, $H_*(k\otimes_{kT}\bfD)$ is isomorphic to the homology of the cokernel 
of the chain map 
	\[ \4\Psi = \bigl(\Id-\oplus\4\Psi^{(n)}\bigr) \: 
	\bigoplus_{n=1}^\infty k\otimes_{kT}\bfD^{(n)} \Right7{} 
	\bigoplus_{n=1}^\infty k\otimes_{kT}\bfD^{(n)} \,. \]
Recall that $\bfD^{(n)}=\bigl(\Lambda^m_{kT}(kT\cdot 
V),\2^{(n)}_m\bigr)_{m=0}^r$ and 
$\Psi^{(n)}=\bigl\{\Lambda^m(\psi_n)\bigr\}_{m=0}^r$. Here, 
$\Lambda^0(\psi_n)=\Id$, while by \eqref{e:Im(psi)<IV},  
$(\Lambda^m(\psi_n))(\Lambda^m_{kT}(kT\cdot V))\le 
I(kT)\cdot\Lambda^m_{kT}(kT\cdot V)$ for $m>0$. So each 
$\4\Psi^{(n)}$ is zero in positive degrees and the identity in degree 0, 
and hence
the quotient complex $\Coker(\4\Psi)$ is 
zero in positive degrees and isomorphic to $k$ in degree $0$. Thus 
$k\otimes_{kT}\bfD$ is acyclic with $H_0(k\otimes_{kT}\bfD)\cong k$. Also, 
	\beqq H_0(\bfD) \cong k\otimes_{kT}H_0(\bfD) \cong 
	H_0(k\otimes_{kT}\bfD) \cong k \,: \label{e:H0} \eeqq
the first isomorphism since $T$ acts trivially on $H_0(\bfD)$ and the 
second since $(k\otimes_{kT}-)$ is right exact.

We have now proven:

\begin{Prop} \label{torus-res}
For $k$, $\Gamma$, and $\bfD$ as above, $\bfD$ is a chain complex of length 
$r+1$ of projective $k\Gamma$-modules, and $\bfD\xto{~\gee~}k\too0$ is an 
$\Omega$-resolution of $k$ with respect to $T$. 
\end{Prop}

\subsection{$\Omega$-resolutions with respect to the Sullivan spheres}
\label{s:rank1}
\leavevmode

We now restrict to the special case of Remark \ref{Osys-T.H2} where $p$ is 
odd and $r=1$. Thus $T\nsg\Gamma$ where $T\cong\Z/p^\infty$, 
$p\nmid|\Gamma/T|$, and $C_\Gamma(T)=T$. Then 
$\Aut(T)\cong(\Z_p)^\times\cong C_{p-1}\times\Z_p$, and since $\Gamma/T$ is 
finite and acts faithfully on $T$, it must be cyclic of order dividing 
$p-1$. Again, we will construct explicit $\Omega$-resolutions of $k$ with 
respect to $\Gamma$.

Spaces $B\Gamma\pcom$ for $\Gamma$ of this form are the simplest and oldest 
examples constructed of $p$-compact groups (other than compact Lie groups). 
They were originally constructed from the space $K(\Z_p, 2)$ ($\simeq 
B(\Z/p^\infty)\pcom$), by taking the Borel construction $B(p,m)$ of the 
faithful action of a cyclic group $C_m$ (for $m|(p-1)$) on $K(\Z_p,2)$. The 
$p$-completion of $B(p,m)$ is a classifying space for the $p$-completed 
sphere $(S^{2m-1})\pcom$, and hence $B(p,m)\pcom$ is the classifying space 
of a $p$-compact group that is often referred to as a ``Sullivan sphere''. 
What we will show is that not only do these spaces have finite 
dimensional homology, but also that the associated $\Omega$-systems have 
$\Omega$-resolutions of finite length.

Write $\Gamma=T\rtimes H\cong\Z/p^\infty\rtimes H$, where 
$H$ is cyclic of order $m|(p-1)$. Let $\chi\:H\too\F_p^\times$ be the 
injective homomorphism such 
that $hth^{-1}=t^{\chi(h)}$ for all $h\in H$ and all $t\in V=\Omega_1(T)$. 
Set $kT_{(1)}=kT\cdot V=kT\otimes_kV$ as a $k\Gamma$-module, which we 
identify with $kT$ but with $H$-action $h(x)=\chi(h)\cdot hxh^{-1}$ for 
$h\in H$ and $x\in kT$. More generally, for arbitrary $j\ge0$, we write 
	\[ k_{(j)} = V^{\otimes j} 
	\qquad\textup{and}\qquad
	kT_{(j)}=kT\otimes_kk_{(j)} \] 
as $k\Gamma$-modules. Thus $k_{(j)}\cong k$ and $kT_{(j)}\cong kT$ as 
$kT$-modules, but $h\in H$ acts on the first via multiplication by 
$\chi(h)^j$ and on the second via that and conjugation. 
We also write $kT=kT_{(0)}$ and $k=k_{(0)}$ for short.

Let $\varphi_n\in\Hom_{k\Gamma}(kT_{(1)},kT)$ and 
$\psi_n\in\Hom_{k\Gamma}(kT_{(1)},kT_{(1)})$ be as in Section 
\ref{s:tori}, and set $\mu_n=\varphi_n(1)$ and $\nu_n=\psi_n(1)$. Then 
for all $n\ge1$, 
	\beqq \mu_{n+1}\nu_n=\mu_n \qquad\textup{and}\qquad
	kT_n\cdot\mu_n=I(kT_n), \label{e:mu.nu=mu} \eeqq
the first since $\varphi_{n+1}\circ\psi_n=\varphi_n$ by 
\eqref{e:H<Gamma-2} and the second since 
$\varphi_n(kT_n\cdot V)=I(kT_n)$ by \eqref{e:H<Gamma-1}. Also, 
	\beqq h(\mu_n)=\chi(h)\cdot\mu_n \label{e:mu.nu=mu2} \eeqq
for all $n\ge1$ and $h\in H$ since $\varphi_n$ is $k\Gamma$-linear.

The complex $\bfD$ of Proposition \ref{torus-res} has the form
	\[ \bfD = \Bigl( 0 \Right2{} \bigoplus_{n=1}^\infty kT_{(1)}\cdot \bfa_n 
	\Right3{\2_2} \bigoplus_{n=1}^\infty kT_{(1)}\cdot \bfa_n 
	\oplus\bigoplus_{n=1}^\infty kT\cdot \bfb_n 
	\Right3{\2_1} \bigoplus_{n=1}^\infty kT\cdot \bfb_n \Right2{} 0 \Bigr), \]
where 
	\[ \2_2(\bfa_n)=-(\bfa_n-\nu_n\bfa_{n+1}) + \mu_n\bfb_n \,,
	\qquad \2_1(\bfa_n)=\mu_n\bfb_n \,, \qquad\textup{and}\qquad 
	\2_1(\bfb_n)=\bfb_n-\bfb_{n+1}. \]
Since 
	$\bigl(\bigoplus_{n=1}^\infty kT\cdot \bfb_n 
	\Right{10}{(\bfb_n\mapsto \bfb_n-\bfb_{n+1})} \bigoplus_{n=1}^\infty 
	kT\cdot \bfb_n \bigr)$ 
is injective with cokernel $kT$, the complex $\bfD$ is equivalent to 
	\[ \4\bfD = \Bigl( 0 \Right2{} \bigoplus_{n=1}^\infty kT_{(1)}\cdot 
	\bfa_n \Right5{\2_2} \bigoplus_{n=1}^\infty kT_{(1)}\cdot \bfa_n 
	\Right5{\2_1} kT_{(0)}\cdot \bfa_0 \Right2{} 0 \Bigr), \]
where this time 
	\[ \2_2(\bfa_n)=\bfa_n-\nu_n\bfa_{n+1} 
	\qquad\textup{and}\qquad \2_1(\bfa_n)=\mu_n\bfa_0. \]
By \eqref{e:mu.nu=mu} and \eqref{e:mu.nu=mu2}, $\2_1$ and $\2_2$ are 
$k\Gamma$-linear and $\2_1\circ\2_2=0$. 

Recall that $\bfD$ is an $\Omega$-resolution with respect to $T$. 
We now want to identify $H_1(\bfD)$ more precisely, and use this complex to 
construct an $\Omega$-resolution with respect to $\Gamma$. To do this, 
define elements $\sigma_n$ (all $n\ge1$) and $\nu_0$ in $kT_n$ by setting 
	\beqq \sigma_n=\sum_{t\in T_n}t \quad
	\textup{(all $n\ge1$)}\qquad\textup{and}\qquad
	\nu_0=\sigma_1. \label{e:eta-n} \eeqq
To better understand the relation between the $\mu_n$, $\nu_n$, and 
$\sigma_n$, fix $n\ge2$ and 
a generator $t_{n}\in T_{n}$, and set $X=t_{n}-1\in I(kT_{n})$. 
Then $X^{p^{n}}=t_n^{p^n}-1=0$ and $\{1,X,X^2,\dots,X^{p^{n}-1}\}$ is a basis for 
$kT_{n}$, so $kT_{n}\cong k[X]/(X^{p^n})$ as rings, and each ideal in 
$kT_{n}$ is a power of $I(kT_{n})=(X)$. Thus $(\mu_{n})=(X)$, 
$(\mu_{n-1})=(t_n^p-1)=(X)^p$, and hence $(\nu_{n-1})=(X)^{p-1}$ 
by \eqref{e:mu.nu=mu}. Also, $(\sigma_n)=(X)^{p^n-1}$ and 
$(\sigma_{n-1})=(X)^{p(p^{n-1}-1)}$ (see \eqref{e:eta-n}), so 
$(\nu_{n-1}\sigma_{n-1})=(X)^{p^n-1}=k\cdot\sigma_n$. Thus 
for each $n\ge2$, $\nu_{n-1}\sigma_{n-1}=a_n\cdot\sigma_n$ for some 
$0\ne a_n\in k$. To simplify notation, we can replace the $\mu_n$ (all 
$n\ge2$) and $\nu_n$ (all $n\ge1$) by appropriate scalar multiples, and 
arrange that 
	\beqq \textup{for each $n\ge1$,} \quad 
	\nu_{n-1}\sigma_{n-1}=\sigma_n\,. \label{e:sigma} \eeqq

Each element in $\Coker(\2_2)$ is the class of $\xi\cdot\bfa_n$ for some 
$n$ and some $\xi\in kT_m$, and we can always arrange (modulo $\Im(\2_2)$) 
that $m=n$. If in addition, $\2_1(\xi\cdot \bfa_n)=0$, then $\mu_n\xi=0$, 
so $I(kT_n)\cdot\xi=0$ by \eqref{e:mu.nu=mu}, and hence 
$\xi=a\cdot\sigma_n$ for some $a\in k$. Thus $H_1(\bfD)$ is generated by 
the classes $[\sigma_n\bfa_n]$ for $n\ge1$, where for each $n$, 
$[\sigma_n\bfa_n]=[\sigma_n\nu_n\bfa_{n+1}]=
[\sigma_{n+1}\bfa_{n+1}]$ by the definition of $\2_2$ and \eqref{e:sigma}. Also, 
$h(\sigma_n)=\chi(h)\sigma_n$ in $kT_{(1)}$, so $H_1(\bfD)\cong k_{(1)}$ as 
$k\Gamma$-modules. To summarize, 
	\beqq H_0(\bfD)=\Gen{[\bfa_0]}\cong k \qquad\textup{and}\qquad 
	H_1(\bfD)=\Gen{[\sigma_1\bfa_1]}\cong k_{(1)} \,. 
	\label{e:H*(D)} \eeqq

Thus $\bfD$ is an $\Omega$-resolution of $k$ with respect to $\Gamma$ 
if $H=1$, but is not an $\Omega$-resolution if $H\ne1$ since $\Gamma$ acts 
nontrivially on $H_1(\bfD)$ (i.e., condition ($\Omega$-3) fails). In this 
case, we construct an $\Omega$-resolution by ``pasting together'' several 
copies of the above sequence.

Define a complex $\bfC_\infty$ of projective $k\Gamma$-modules of 
infinite length 
	\begin{multline*} 
	\bfC_\infty = \Bigl( \cdots \Right2{} 
	\bigoplus_{n=0}^\infty kT_{(3)}\cdot \bfa_n \Right2{\2_6} 
	\bigoplus_{n=1}^\infty kT_{(3)}\cdot \bfa_n \Right2{\2_5} 
	\bigoplus_{n=0}^\infty kT_{(2)}\cdot \bfa_n \Right2{\2_4} \\
	\bigoplus_{n=1}^\infty kT_{(2)}\cdot \bfa_n \Right2{\2_3}
	\bigoplus_{n=0}^\infty kT_{(1)}\cdot \bfa_n \Right2{\2_2}
	\bigoplus_{n=1}^\infty kT_{(1)}\cdot \bfa_n \Right2{\2_1}
	kT_{(0)}\cdot \bfa_0 \Right2{} 0 \Bigr) ,
	\end{multline*}
where $\2_1(\bfa_n)=\mu_n\bfa_0$ (as in $\bfD$), and for $i\ge2$, 
	\[ \2_i(\bfa_n) = \begin{cases} 
	\bfa_n - \nu_n\bfa_{n+1} & \textup{if $i$ is even}\\
	\mu_n(\bfa_0+\sigma_1\bfa_1+\sigma_2\bfa_2
	+\dots+\sigma_{n-1}\bfa_{n-1}) & \textup{if $i$ is odd.}
	\end{cases} \]
Here, it is understood that $\bfa_0=0$ in the terms of odd degree. By 
\eqref{e:mu.nu=mu}, \eqref{e:mu.nu=mu2}, \eqref{e:eta-n}, and 
\eqref{e:sigma}, all boundary maps are $k\Gamma$-linear 
and $\2_{i-1}\circ\2_{i}=0$ for all $i\ge2$. 

For each $j\ge1$, let $\bfC_j\subseteq\bfC_\infty$ be the subcomplex 
consisting of all terms in $\bfC_\infty$ of degree at most $2j-1$ together 
with the summands $\bigoplus_{n=1}^\infty kT_{(j)}\cdot \bfa_n$ in degree 
$2j$ (thus omitting only the summand $kT_{(j)}\cdot \bfa_0$). Thus 
$\bfC_1\cong\4\bfD$. More generally, if we set $\bfC_0=0$, then for each 
$j\ge0$, $\bfC_{j+1}/\bfC_j$ is isomorphic to the $2j$-fold suspension 
of $\4\bfD$ tensored by $k_{(j)}$, and hence by \eqref{e:H*(D)} has 
homology isomorphic to $k_{(j+1)}$ in degree $2j+1$ and $k_{(j)}$ 
in degree $2j$. If $j\ge1$, then the homology of $\bfC_{j+1}/\bfC_j$ in 
degree $2j$ is represented by the class of $\bfa_0$ in that degree, 
$\2_{2j}(\bfa_0)=-\nu_0\bfa_1=-\sigma_1\bfa_1$, and by \eqref{e:H*(D)} again, this 
represents the homology class in $\bfC_j/\bfC_{j-1}$ of degree $2j-1$. 
Together, these observations imply that $\bfC_\infty$ is acyclic, and 
that for each $j\ge1$, 
	\[ H_0(\bfC_j)\cong k, \qquad H_{2j-1}(\bfC_j)\cong k_{(j)}, 
	\qquad\textup{and}\qquad H_i(\bfC_j)=0 
	\textup{ for $i\ne0,2j-1$.} \]

Set $\bfR=\bfC_m$ (recall $m=|H|$). We claim that $\bfR\too k\too0$ is an 
$\Omega$-resolution with respect to $\Gamma$. Condition ($\Omega$-1) clearly holds 
(each of the terms is projective), and $\Gamma$ acts trivially on 
$H_*(\bfR)$ since $\chi^m=1$. It remains to show ($\Omega$-2): that 
$k\otimes_{k\Gamma}\bfR$ is acyclic. Since $k\otimes_{k\Gamma}kT_{(i)}\cong 
k$ whenever $m|i$ and is zero otherwise, 
	\[ k\otimes_{k\Gamma}\bfR \cong \Bigl( 0 \too 
	\bigoplus_{n=1}^\infty k\cdot \bfa_n \xto{\4\2_{2m}=\Id} 
	\bigoplus_{n=1}^\infty k\cdot \bfa_n \too 0 \too \dots \too 0 \too 
	k \too 0 \Bigr) \]
and is acyclic.

We have now shown:

\begin{Prop} \label{Ores:Sullivan}
Let $\Gamma=T\rtimes H$, where $T\cong\Z/{p^\infty}$ and $H$ acts on 
$\Omega_1(T)$ via an 
injective character $\chi\:H\too\F_p^\times$. Set $m=|H|$. Then the complex 
$\bfR\too k\too0$ defined above is an $\Omega$-resolution of $k$ with 
respect to the $\Omega$-system $\Osys{k\Gamma\mod}{k\mod}$, and 
$H_*(\Omega(B\Gamma\pcom);k)\cong H_*(S^{2m-1};k)$.
\end{Prop}

Note that $\bfC_j$ is not an $\Omega$-resolution when $j>m$ since 
$k\otimes_{k\Gamma}\bfC_j$ has nonzero homology in degree $2m$ (the class 
of $\bfa_0$), and is not an $\Omega$-resolution when $1\le j<m$ since 
$\Gamma$ acts nontrivially on $H_{2j-1}(\bfC_j)$. Note also that since 
$\Omega(B\Gamma\pcom)$ has nonzero homology in degree $2m-1$, $\bfR$ has 
the shortest possible length of any $\Omega$-resolution of $k$ with respect 
to $\Osys{k\Gamma\mod}{k\mod}$. The equivalence 
$\Omega(B\Gamma\pcom)\simeq(S^{2m-1})\pcom$ follows from \cite[pp. 
103--105]{Sullivan} (from the proof of the proposition), and since 
$BT\pcom\simeq K(\Z_p,2)$ by Proposition \ref{Osys-T.H}(a).

We will see later (Proposition \ref{finite-res} and Remark 
\ref{filtered-res}) that there are similar constructions of 
$\Omega$-resolutions when $T\nsg\Gamma$ are such that 
$T\cong(\Z/p^\infty)^r$ for $r>1$ and $p\nmid|\Gamma/T|$.

\begin{Rmk} 
When $H\ne1$, the parameters $\mu_n$ and $\nu_n$ can be defined 
more explicitly as follows. Fix generators $t_n\in T_n$ for each $n\ge1$, 
chosen so that $(t_n)^p=t_{n-1}$ when $n\ge2$, and set $\mu_n=\sum_{h\in 
H}\chi(h)^{-1}t_n^{\chi(h)}\in kT_n$ for all $n\ge1$. It is straightforward 
to check that $h(\mu_n)=\chi(h)\mu_n$ for $h\in H$, that 
$kT_n\cdot\mu_n=I(kT_n)$, and that $(\mu_1)^p=0$ while 
$(\mu_n)^p=\mu_{n-1}$ for $n\ge2$. Also, 
$(\mu_n)^{p-1}\sigma_{n-1}=\sigma_n$, so we can set 
$\nu_{n-1}=(\mu_{n})^{p-1}$ for each $n$, and use these parameters to 
define the $\Omega$-resolution $\bfR$.
\end{Rmk}

\subsection{Groups with discrete $p$-tori of index prime to $p$}
\leavevmode

We now make some more computations of $\Omega$-homology in the situation of 
Proposition \ref{Osys-T.H}: this time by using the existence of 
$\Omega$-resolutions without constructing them explicitly. The key to doing 
this is the following spectral sequence.

\begin{Lem} \label{H_*(C_T)-sp.seq.} Let $T\nsg \Gamma$ be a pair of groups 
such that $T$ is $p$-perfect. Let $(C_*,\2_*)$ be a positively graded  
chain complex of $k\Gamma$-modules that are projective as $kT$-modules, and 
assume that $T$ acts trivially on $H_*(C_*,\2_*)$. Then there is a first 
quadrant spectral sequence of $k[\Gamma/T]$-modules of the form 
	\[ E^2_{ij} = H_i(k\otimes_{kT}C_*) \otimes_k 
	H_j(\Omega(BT\pcom);k) \quad\implies\quad H_{i+j}(C_*) \] 
where the action of $\Gamma/T$ on $H_*(\Omega(BT\pcom);k)$ is that induced 
by conjugation on $T$. \end{Lem}

\begin{proof} Since $BT\pcom$ is simply connected by 
Proposition \ref{Osys-T.H}(a), the space $\Omega(BT\pcom)$ is 
connected.  Consider the following diagram of spaces 
	\[ \xymatrix@C=40pt{
	L^pT \ar[r] \ar[d] & A^pT \ar[r] \ar[d]^{\nu} & \calp(BT\pcom) \ar[d] \\
	ET \ar[r] & BT \ar[r] & BT\pcom } \]
where $BT=|\calb(T)|$ and $ET$ is its universal covering space, 
$\calp(BT\pcom)$ is the space of paths in $BT\pcom$ originating at the 
image of the (unique) vertex in $BT$, both squares are pullbacks, and thus 
the vertical maps are all fibrations with fibre $\Omega(BT\pcom)$. 
In particular, $A^pT$ is the homotopy fibre of the completion map, and 
is mod $p$ acyclic since $H_*(BT;\F_p)\cong H_*(BT\pcom;\F_p)$ 
and $BT\pcom\simeq K(\Z_p,2)^r$ is simply connected. Also, 
$L^pT\simeq\Omega(BT\pcom)$ since $ET\simeq*$, and $T$ acts freely on 
$L^pT$ with orbit space $A^pT$.

Now, $\Gamma$ acts on the right on all of these spaces via the 
conjugation action. More precisely, we identify the vertices in $ET$ with 
$T$, and let $\Gamma$ act on $ET$ by setting $x*g=g^{-1}xg$ for $x\in 
T$ and $g\in\Gamma$, in contrast to the free right action of $T$ defined by 
$x\cdot t=xt$. This induces actions of $\Gamma$ on $BT$, $BT\pcom$, and 
$\calp(BT\pcom)$, and hence on $L^pT$; and the actions of $T$ and 
$\Gamma$ on $ET$ and on $L^pT$ satisfy the relation 
$((x*g)\cdot t)*g^{-1}=x\cdot(gtg^{-1})$ for $g\in\Gamma$ and 
$t\in T$. 

In particular, $C_*(L^pT;k)$ is a complex of $k[T\sd{}\Gamma]$-modules that 
are free as $kT$-modules, and $C_*(A^pT;k)\cong C_*(L^pT;k)\otimes_{kT}k$ 
is an acyclic complex of $k\Gamma$-modules. The action of $\Gamma$ on 
$H_*(L^pT;k)$ restricts to the conjugation action of $T=\pi_1(BT)$ on the 
fibre of $\nu$ over the basepoint, and this action is trivial since $\nu$ 
is pulled back from the simply connected space $BT\pcom$. Thus the action 
of $\Gamma$ on $H_*(L^pT;k)\cong H_*(\Omega(BT\pcom);k)$ factors through 
$\Gamma/T$. 

Consider the complex $C_*(L^pT;k)\otimes_{kT} C_*$. This is a double 
complex of $k\Gamma$-modules, where $g(x\otimes y)=x*g^{-1}\otimes gy$ for 
$g\in\Gamma$. This action of $\Gamma$ is well defined on the tensor product 
over $kT$, since for $g\in\Gamma$ and $t\in T$,
	\begin{align*} 
	g(xt\otimes t^{-1}y) &= (xt)*g^{-1}\otimes gt^{-1}y 
	= (x*g^{-1})\cdot(gtg^{-1})\otimes gt^{-1}y \\
	&= x*g^{-1} \otimes (gtg^{-1})gt^{-1}y 
	= x*g^{-1}\otimes gy = g(x\otimes y) 
	\end{align*}
by the relation shown above. As usual, we consider the two spectral 
sequences of $k\Gamma$-modules induced by this double complex.

If we first take homology in the left-hand factor, we obtain 
	\[ E^1_{i,j} \cong H_j(\Omega(BT\pcom);k)\otimes_{kT} C_i
	\qquad\text{and}\qquad 
	E^2_{i,j} \cong H_j(\Omega(BT\pcom);k) \otimes_k 
	H_i(k\otimes_{kT}C_*), \]
where the first isomorphism holds since each $C_i$ is projective as a 
$kT$-module, and the second since $T$ acts trivially on $\Omega(BT\pcom)$. 
On the other hand, if we first take homology of the right-hand factor, we 
obtain
	\[\bar{E}^1_{i,j}\cong C_j(L^pT;k)\otimes_{kT} H_i(C_*) \cong 
	\bigl(C_j(L^pT;k)\otimes_{\kk T} k\bigr) \otimes H_i(C_*),\]
the first isomorphism since each $C_j(L^pT;k)$ is free as a $kT$-module, 
and the second since the action of $T$ on the homology of $C_*$ is 
trivial. Then 
	\[ \bar{E}^2_{i,j} \cong \begin{cases} 
	H_i(C_*) & \textup{if $j=0$} \\ 0 & \textup{if $j>0$}
	\end{cases} \]
since $C_*(L^pT;k)\otimes_{kT}k$ is isomorphic to $C_*(A^pT;k)$ and hence 
is acyclic.
\end{proof}

In the following lemma, when we say that a graded vector space over $\kk$ 
is ``finite dimensional'', we mean that it is finite dimensional in each 
degree and is nonzero in only finitely many degrees.

\begin{Lem} \label{fin<=>fin}
Fix a pair of groups $T\nsg \Gamma$, where $T\cong(\Z/p^\infty)^r$ for 
some $r\ge1$ and $\Gamma/T$ is finite. Let 
$O^p(\Gamma)\nsg\Gamma$ and $\pi=\Gamma/O^p(\Gamma)$  be as in 
Proposition \ref{Osys-T.H} (thus $O^p(\Gamma)\ge T$ and $\pi$ is a finite 
$p$-group), set $\pi=\Gamma/O^p(\Gamma)$, and let 
$\theta\:\calb(\Gamma)\too\calb(\pi)$ be the natural functor. Let 
$(C_*,\2_*)$ be an $\Omega$-resolution of $k\pi$ with respect to the 
$\Omega$-system $\Osys{k\Gamma\mod}{k\pi\mod}$. Then $H_*(C_*,\2_*)$ is 
finite dimensional if and only if $H_*(C_*\otimes_{kT}k)$ is finite 
dimensional. 
\end{Lem}

\begin{proof} Since $H_*(\Omega(BT\pcom);\kk)$ is finite dimensional, Lemma 
\ref{H_*(C_T)-sp.seq.} implies that if $H_*(C_*\otimes_{kT}\kk)$ is finite 
dimensional, then so is $H_*(C_*;\kk)$. Thus it remains to prove the 
converse. Since $H_*(-;\kk)\cong H_*(-;\F_p)\otimes_{\F_p}\kk$, it 
suffices to show this when $\kk=\F_p$. 

Consider the following diagram:
	\[ \xymatrix@C=40pt{
	L^p\Gamma \ar[r] \ar[d] & (L^p\Gamma)/T \ar[r] \ar[d] & A^p\Gamma 
	\ar[r] \ar[d] & \calp(B\Gamma\pcom) \ar[d] \\ 
	E\Gamma \ar[r] & (E\Gamma)/T \ar[r] & B\Gamma \ar[r] & B\Gamma\pcom 
	} \]
where the squares are pullback squares. Since $O^p(\Gamma)$ is 
$p$-perfect by Proposition \ref{Osys-T.H}(a) (and since $\pi$ is a 
finite $p$-group), Lemma \ref{|X|pcom->Bpi}(c) implies that 
$B\Gamma\pcom$ is an $\F_p$-plus construction for 
$(B\Gamma,O^p(\Gamma))$. Also, $L^p\Gamma$ 
has a free action of $\Gamma$ induced by that on $E\Gamma$. Set 
$\calc=\calb(\Gamma)$ and let $\theta\:\calc\too\calb(\pi)$ be the 
natural functor, so that $E\calc(\Bobj[\Gamma])=E\Gamma$. Then by 
Proposition \ref{C*Enu}, $C_*(L^p\Gamma;\F_p)$ is an 
$\Omega$-resolution of $\F_p\pi$ with respect to the given 
$\Omega$-system. Hence by Proposition \ref{p2:functoriality}, we may 
assume $C_*=C_*(L^p\Gamma;\F_p)$, and so $C_*\otimes_{\F_pT}\F_p\cong 
C_*(L^p\Gamma/T;\F_p)$. From the above pullback diagram, we see that 
$L^p\Gamma/T$ is the homotopy fibre of the map 
$BT\simeq{}E\Gamma/T\Right2{}B\Gamma\pcom$. Since 
$\pi_1(B\Gamma\pcom)\cong\pi$ is a finite $p$-group by 
Proposition \ref{Osys-T.H}(a), it acts nilpotently on 
$H_*(L^p\Gamma/T;\F_p)$, and hence by the  mod-$R$ fibre lemma of 
Bousfield and Kan \cite[Lemma II.5.1]{BK}, $(L^p\Gamma/T)\pcom$ is the 
homotopy fibre of the map $BT\pcom\too B\Gamma\pcom$ induced by the 
inclusion. Since $BT\pcom$ and $B\Gamma\pcom$ are $p$-complete 
by Proposition \ref{Osys-T.H}(a), the space $(L^p\Gamma/T)\pcom$ is 
also $p$-complete by the mod-$R$ fibre lemma again.

Now assume that $H_*(C_*;\F_p)$ is finite dimensional, and hence that 
$\Omega(B\Gamma\pcom)$ is a $p$-compact group. By Proposition 
\ref{Osys-T.H}(b), there is a $p$-local compact group $\SFL$ with $T\le 
S\le\Gamma$ and $|\call|\pcom\simeq B\Gamma\pcom$. So for each $1\ne t\in 
T$, \cite[Theorem 6.3]{BLO3} implies that the map $B\gen{t}\too 
B\Gamma\pcom$ induced by the inclusion is not null homotopic. In the 
terminology of \cite[\S\,7]{DW}, this means that the map $BT\pcom\too 
B\Gamma\pcom$, regarded as a map of $p$-compact groups, has trivial kernel. 
So by \cite[Theorem 7.3]{DW}, $BT\pcom\too B\Gamma\pcom$ is a monomorphism 
of $p$-compact groups, and by definition \cite[3.2]{DW}, the 
$\F_p$-homology of its homotopy fibre $(L^p\Gamma/T)\pcom$ is finite 
dimensional. 
\end{proof}

We now give some consequences of Lemmas \ref{H_*(C_T)-sp.seq.} and 
\ref{fin<=>fin}. The first example can also be carried out using the 
Serre spectral sequence for the path-loop fibration of $B\Gamma\pcom$, 
but the argument given here is slightly easier, and it illustrates 
nicely how the action of $\Gamma/T$ on the spectral sequence of Lemma 
\ref{H_*(C_T)-sp.seq.} can be exploited. Note that the space 
$\Omega(B\Gamma\pcom)$ considered in the lemma is not a $p$-compact 
group, since it has unbounded homology.

\begin{Ex} \label{ex:T2:2}
Set $\Gamma=T\sd{}C_2$, where $p$ is odd, $T\cong(\Z/p^\infty)^2$, and 
$\Gamma/T\cong C_2$ acts by inverting elements of $T$. Then
	\[ H_i(\Omega(B\Gamma\pcom);\F_p) \cong \begin{cases} 
	\F_p & \textup{if $i=0$}\\
	\F_p^3 & \textup{if $i=3$}\\
	\F_p^4 & \textup{if $i>3$ and $i\in3\Z$}\\
	0 & \textup{otherwise.}
	\end{cases} \]
\end{Ex}

\begin{proof} To simplify notation, we do this over an arbitary field $\kk$ 
of characteristic $p$. Fix an $\Omega$-resolution $(C_*,\2_*)$ of $\kk$ 
with respect to $\Gamma$, and let $E$ be the spectral sequence of 
Lemma \ref{H_*(C_T)-sp.seq.}. Recall that this is a spectral sequence of 
$\kk[\Gamma/T]$-modules: the action of $\Gamma/T$ plays a central role in 
the argument here. 

Since $H_i(C_*\otimes_{\kk\Gamma}\kk)=0$ for $i>0$ by \omres2, 
$\Gamma/T\cong C_2$ acts via $-\Id$ on $E^2_{i,0}\cong H_i(C_*\otimes_{\kk 
T}\kk)$ for each $i>0$. Also, since $\Gamma$ acts trivially on $H_i(C_*)$ 
for all $i$ by \omres3, $\Gamma/T$ acts trivially on $E^\infty_{i,j}$ for 
all $i,j$. 

We claim that $E^2_{i,j}$ takes the following form:
	\[ \underline{\left| 
	\vcenter{\xymatrix@C=15pt@R=8pt{ 
	\kk_+ & 0 & \kk^2_- & 0 & 0 & \kk^2_- & 0 & 0 & \kk^2_- & \dots \\
	\kk^2_- & 0 & \kk^4_+ \ar@{-->}[ull] & 0 & 0 & \kk^4_+ & 0 & 0 & \kk^4_+ & \dots \\
	\kk_+ & 0 & \kk^2_- \ar[ull]_(0.4){\cong} & 0 & 0 & \kk^2_- 
	\ar[uulll]_(0.6){\cong} & 0 & 0 & \kk^2_- \ar[uulll]_(0.6){\cong} & \dots 
	}}\right.} \] 
where the subscripts ($\pm$) describe the action of $\Gamma/T\cong C_2$, 
and where the pattern continues with $E^2_{3k+2,0}\cong\kk^2_-$ for $i\ge0$ 
and $E^2_{i,0}=0$ when $0<i\not\equiv2$ (mod $3$). To see this, note first 
that $E^2_{1,0}\cong E^\infty_{1,0}=0$ since $\Gamma/T$ acts by 
inverting elements in $E^2_{1,0}$ and fixing those in 
$E^\infty_{1,0}$, and that the differential sends $E^2_{2,0}$ 
isomorphically to $E^2_{0,1}$ since the action on $E^\infty$ is trivial. 
Hence $E^2_{i,j}$ is as described for $i\le2$. Also, $E^2_{3,0}=0$ and 
$E^2_{4,0}=0$ since there are no terms which their 
differentials could hit upon which $\Gamma/T$ acts by 
inverting elements, and $E^2_{5,0}\cong\kk^2_-$ must be sent 
isomorphically to $E^2_{2,2}$. 

Upon continuing in this way, we see inductively that 
$E^2_{i,0}\cong\kk^2_-$ for $i=3j+2$ (all $j\ge0$) and is zero in other 
positive degrees, and the differentials are as shown with one possible 
exception. We claim that the differential from $E^2_{2,1}$ to $E^2_{0,2}$ is 
surjective: this holds since 
$H_i(B\Gamma\pcom;k)\cong(H_i(BT\pcom;k))^{\Gamma/T}=0$ for $i\le3$ implies 
that $H_i(\Omega(B\Gamma\pcom);k)=0$ for $i\le2$. Thus $E^\infty_{i,j}$ is zero in 
positive total degrees except $E^\infty_{21}\cong\kk^3$ and 
$E^\infty_{3i+2,1}\cong\kk^4$ (all $i\ge1$). 
\end{proof}

The next proposition shows that in the situation of 
Lemma \ref{fin<=>fin}, at least, whenever the homology of the loop 
space is bounded, we get an $\Omega$-resolution of finite length. 

\begin{Prop} \label{finite-res}
Let $T\nsg \Gamma$ be such that $T\cong(\Z/p^\infty)^r$ for some $r\ge1$ 
and $\Gamma/T$ is finite. If in addition, $\Omega(B\Gamma\pcom)$ is a 
$p$-compact group, then there is an $\Omega$-resolution of $k$ with respect 
to $\Gamma$ of finite length. More precisely, if $N\ge1$ is maximal such 
that $H_N(\Omega(B\Gamma\pcom);k)\ne0$, then there is an 
$\Omega$-resolution of $k$ with respect to $\Gamma$ of length $N+1$. 
\end{Prop}

\begin{proof} Since $\Omega(B\Gamma\pcom)$ is a $p$-compact group, 
$O^p(\Gamma/T)$ has order prime to $p$ by Proposition 
\ref{Osys-T.H}(c). 
Let $O^p(\Gamma)\nsg\Gamma$ and $\pi=\Gamma/O^p(\Gamma)$ be as 
in Proposition \ref{Osys-T.H}. Since $\Omega(B(O^p(\Gamma))\pcom)$ is 
homotopy equivalent to a connected component of $\Omega(B\Gamma\pcom)$ 
(as shown in the proof of Proposition \ref{Osys-T.H}(c)), 
$\Omega(BO^p(\Gamma)\pcom)$ is also a $p$-compact group. If $\bfC\too 
k\too0$ is an $\Omega$-resolution of $k$ with respect to $O^p(\Gamma)$, 
then $\Ind_{O^p(\Gamma)}^\Gamma(\bfC)\too k\pi\too0$ is an 
$\Omega$-resolution of $k\pi$ with respect to $\Gamma$. So it suffices 
to prove the proposition when $\Gamma=O^p(\Gamma)$; i.e., when 
$\Gamma/T$ has order prime to $p$ and $\pi=1$.

Recall that $r$ is the rank of $T$. By Proposition \ref{torus-res}, 
there is an $\Omega$-resolution $\bfD\too k\too0$ of $k$ with respect to $T$ of length $r+1$ which 
is also a chain complex of projective $k\Gamma$-modules.

\smallskip

\noindent\textbf{Step 1: } By Proposition \ref{exists-res}, there is an 
$\Omega$-resolution $\bfC\too k\too0$ of $k$ with respect to $\Gamma$ 
(possibly of infinite length). We will use this as a template for 
constructing an $\Omega$-resolution of finite length $N+1$.

By Lemma \ref{fin<=>fin}, the homology of $k\otimes_{kT}\bfC$ is finitely 
generated. Let 
	\[ 0=m_1 < m_2 < \cdots < m_\ell=m \]
be the degrees in which $H_*(k\otimes_{kT}\bfC)$ is nonzero. Thus $\ell$ is 
the number of distinct degrees in which this homology is nonzero. By the 
spectral sequence $\{E^r_{*,*}\}$ of Lemma \ref{H_*(C_T)-sp.seq.}, 
$H_{m+r}(\bfC)\cong E^2_{m,r}\ne0$, and so $N=m+r$ by Corollary 
\ref{Omega_G}.

\smallskip

\noindent\textbf{Step 2: } By Proposition \ref{p2:functoriality}, and 
since $\bfC$ satisfies condition ($\Omega$-3) (Definition 
\ref{d:Omega-resolution}) as a complex of $kT$-modules, there is a 
$kT$-linear chain map 
	\[ \psi^{(0)}_*\: \bfD = H_{0}(k\otimes_{kT}\bfC)\otimes_k 
	\Sigma^{0}\bfD \Right5{} \bfC \]
that induces a $k\Gamma$-linear isomorphism 
$H_0(\bfD)\xto{~\cong~}H_0(\bfC)$. By averaging, i.e., by replacing 
$\psi^{(0)}_*$ by the map 
$x\mapsto\frac{1}{|\Gamma/T|}\sum_{gT\in\Gamma/T}g(\psi^{(0)}_*(g^{-1}x))$ for $x\in\bfD$, we 
can arrange that $\psi^{(0)}_*$ is $k\Gamma$-linear without changing 
$H_0(\psi^{(0)}_*)$. Let $\bfC^{(1)}$ be 
the mapping cone of $\psi^{(0)}_*$ \cite[\S\,1.5]{Weibel}; again a chain 
complex of projective $k\Gamma$-modules. Since the $p$-perfect group $T$ 
acts trivially on the homology of $\bfD$ and of $\bfC$, it also acts 
trivially on $H_i(\bfC^{(1)})$ for each $i$ (Lemma \ref{l:p-perf}). Also, 
the homology of $k\otimes_{kT}\bfC^{(1)}$ is isomorphic to that of 
$k\otimes_{kT}\bfC$ (as $k$-vector spaces), except in degree $0=m_1$.

\smallskip

\noindent\textbf{Step 3: } Let $t$ be the minimum of all $i$ such that 
$H_i(\bfC^{(1)})\ne0$. If $t=\infty$ (i.e., if $\bfC^{(1)}$ 
is exact), then the sequence splits, so $k\otimes_{kT}\bfC^{(1)}$ is also 
exact, and $\ell=1$. 

Assume that $\bfC^{(1)}$ is not exact; i.e., that $t<\infty$. 
Then the exact sequence $C^{(1)}_{t}\too C^{(1)}_{t-1}\too \cdots\too 
C^{(1)}_0\too0$ of projective $k\Gamma$-modules splits. By this splitting, and 
since $(k\otimes_{kT}-)$ is right exact and $T$ acts trivially on 
$H_t(\bfC^{(1)})$, we have $H_t(k\otimes_{kT}\bfC^{(1)})\cong 
H_t(\bfC^{(1)})\ne0$, while $H_i(k\otimes_{kT}\bfC^{(1)})=0$ for all $i<t$. 
Thus $t=m_2$ and $\ell\ge2$. By Proposition \ref{p2:functoriality} again 
(and averaging), there is a $k\Gamma$-linear chain map 
	\[ \psi^{(1)}_*\: H_{m_2}(k\otimes_{kT}\bfC)\otimes_k \Sigma^{m_2}\bfD 
	\Right5{} \bfC^{(1)} \]
that induces an isomorphism in $H_{m_2}(-)$. In other words, we shift $\bfD$ by 
degree $m_2$, tensor each term by the $k$-module 
	\[ H_{m_2}(k\otimes_{kT}\bfC)\cong 
	H_{m_2}(k\otimes_{kT}\bfC^{(1)})\cong H_{m_2}(\bfC^{(1)}), \] 
and then map the resulting complex into $\bfC^{(1)}$. 

Let $\bfC^{(2)}$ be the mapping cone of $\psi^{(1)}_*$. By the 
arguments used in Step 2, $T$ acts trivially on $H_*(\bfC^{(2)})$, 
and $H_i(k\otimes_{kT}\bfC^{(2)})\cong H_i(k\otimes_{kT}\bfC)$ for all 
$i>m_2$ while $H_i(k\otimes_{kT}\bfC^{(2)})=0$ for $i\le m_2$. 

\smallskip

\noindent\textbf{Step 4: } We now repeat this procedure to obtain an 
increasing sequence 
	\[ \bfC \le \bfC^{(1)} \le \bfC^{(2)} \le \cdots \le \bfC^{(\ell)} \]
of chain complexes of projective $k\Gamma$-modules, where for each $1\le 
r\le\ell$, $T$ acts trivially on $H_*(\bfC^{(r)})$, and 
$H_i(k\otimes_{kT}\bfC^{(r)})\cong H_i(k\otimes_{kT}\bfC)$ for all $i>m_r$ 
while $H_i(k\otimes_{kT}\bfC^{(r)})=0$ for $i\le m_r$. Also, by the argument 
at the start of Step 3, $H_i(\bfC^{(r)})=0$ for $i<m_{r+1}$, while 
$H_{m_{r+1}}(\bfC^{(r)})\cong H_{m_{r+1}}(k\otimes_{kT}\bfC^{(r)})\cong 
H_{m_{r+1}}(k\otimes_{kT}\bfC)\ne0$. 

In particular, $\bfC^{(\ell)}$ is an exact sequence of 
projective $k\Gamma$-modules. Set 
	\[ \bfR = \Sigma^{-1}
	\bigl(\bfC^{(\ell)}\big/\bfC\bigr). \]
Then $\bfC^{(\ell)}$ is the mapping cone of a 
$k\Gamma$-linear chain map $\bfR\too\bfC$, and $\Gamma$ 
acts trivially on $H_*(\bfR)\cong H_*(\bfC)$. Also, 
$k\otimes_{k\Gamma}\bfR$ is acyclic since $k\otimes_{k\Gamma}\bfC$ is (and 
since the sequence $k\otimes_{k\Gamma}C^{(\ell)}_*$ is exact), and so 
$\bfR$ is an $\Omega$-resolution of $k$ with respect to $\Gamma$ of length 
$m+r+1=N+1$.
\end{proof}

Note that the converse of Proposition \ref{finite-res} also holds: 
$\Omega(B\Gamma\pcom)$ is a $p$-compact group if there is an 
$\Omega$-resolution of finite length. As usual, this can be reduced to the 
case where $\Gamma/T$ is $p$-perfect; i.e., where $B\Gamma\pcom$ is simply 
connected. If there is an $\Omega$-resolution of $k$ with respect to 
$\Gamma$ of finite length, then $H_i(\Omega(B\Gamma\pcom);\F_p)=0$ for $i$ 
large enough by Corollary \ref{Omega_G}. Also, $H_i(B\Gamma\pcom;\F_p)\cong 
H_i(B\Gamma;\F_p)$ is finite for each $i$ by the Serre spectral sequence 
for the fibration sequence $BT\too B\Gamma\too B(\Gamma/T)$ (and since 
$H_*(BT;\F_p)\cong H_*(B(S^1)^r;\F_p)$ and $H_*(B(\Gamma/T);\F_p)$ are 
finite in each degree), and hence $H_i(\Omega(B\Gamma\pcom);\F_p)$ is 
finite for each $i$ by \cite[Proposition 7]{Serre2} and since 
$B\Gamma\pcom$ is simply connected. So $\Omega(B\Gamma\pcom)$ is a 
$p$-compact group.

\begin{Rmk} \label{filtered-res}
By a closer inspection, one can say more about the $\Omega$-resolution 
constructed in the proof of Proposition \ref{finite-res}. By construction, 
$\bfR$ has a filtration whose successive quotients are the 
suspended complexes $H_{m_r}(k\otimes_{kT}\bfR)\otimes_k\Sigma^{m_r}\bfD$ 
for $1\le r\le\ell$, where $\bfD$ is the complex constructed in Proposition 
\ref{torus-res} and $0=m_1<m_2<\cdots<m_\ell=m$ are the degrees in which 
$H_*(k\otimes_{kT}\bfR)$ is nonzero. 
\end{Rmk}

\appendix

\section{$\RR$-perfect groups and $\RR$-plus constructions}

Recall that for a group $G$, we write $G\ab=G/[G,G]\cong H_1(G;\Z)$ for 
short. For a commutative ring $\RR$, we say that $G$ is $\RR$-perfect if 
$H_1(G;\RR)=0$; equivalently, if $G\ab\otimes_{\Z}\RR=0$. When $p$ is a prime, 
$G$ is $p$-perfect if it is $\F_p$-perfect.

\begin{Lem} \label{l:p-perf}
Fix a commutative ring $\RR$ and an $\RR$-perfect group $G$. Let 
$M_0\subseteq M$ be $\RR G$-modules such that $G$ acts trivially on 
$M_0$ and on $M/M_0$. Then $G$ also acts trivially on $M$.
\end{Lem}

\begin{proof} Let $\Aut_\RR^0(M)$ be the group of all $\RR$-linear 
automorphisms of $M$ that induce the identity on $M_0$ and on $M/M_0$. Then 
$\Aut_\RR^0(M)\cong\Hom_\RR(M/M_0,M_0)$ is abelian and has the structure of 
an $\RR$-module, and $G$ acts on $M$ via a homomorphism 
$G\too\Aut_\RR^0(M)$. Each homomorphism from $G$ to an $\RR$-module factors 
through $G\ab\otimes_{\Z}\RR=0$ and hence is trivial, so $G$ acts trivially 
on $M$.
\end{proof}

We now turn our attention to plus constructions (see Definition 
\ref{d:plus}). The following lemma will be needed when checking the 
condition in the definition about homology with twisted coefficients.

\begin{Lem} \label{cover->twisted}
Fix a commutative ring $\RR$ and a group $\pi$. Let $f\:X\too Y$ be a map 
between connected spaces, and let $\eta\:\pi_1(Y)\too\pi$ be a homomorphism 
such that $\eta$ and $\eta\circ\pi_1(f)$ are both surjective. Let $\til{X}$ 
and $\til{Y}$ be the covering spaces of $X$ and $Y$ with fundamental groups 
$\Ker(\eta\circ\pi_1(f))$ and $\Ker(\eta)$, respectively, and assume that a 
covering map $\til{f}\:\til{X}\too\til{Y}$ is an $\RR$-homology 
equivalence. Then $H_*(f;M)$ is an isomorphism for each $\RR\pi$-module 
$M$. 
\end{Lem}

\begin{proof} Let $\5C_*$ be the mapping cone of the chain map 
$C_*(\til{f})\:C_*(\til{X};\RR)\too C_*(\til{Y};\RR)$ (see the remark just 
before Proposition \ref{finite-res}). Since $C_*(X;M)\cong 
C_*(\til{X};R)\otimes_{\RR\pi}M$ and similarly for $Y$ (see \cite[Theorem 
VI.3.4]{Whitehead}), we get that $\5C_*\otimes_{\RR\pi}M$ is the mapping 
cone of $C_*(f;M)$. Since $\5C_*$ is an exact sequence of free 
$\RR\pi$-modules and is bounded below, $\5C_*\otimes_{\RR\pi}M$ is also 
exact, and hence $H_*(f;M)$ is an isomorphism. 
\end{proof}

As was defined in the introduction (see also Section \ref{s:loops}), 
a group $G$ is strongly $\RR$-perfect if it is $\RR$-perfect 
($H_1(G;\Z)\otimes\RR=0$) and $\Tor(H_1(G;\Z),\RR)=0$. Thus if $\RR$ 
is flat as a $\Z$-module (in particular, if $\RR=\Z$), then $G$ is strongly 
$\RR$-perfect if it is $\RR$-perfect. By the universal coefficient 
theorem, all $\RR$-superperfect groups are strongly $\RR$-perfect, where a 
group $G$ is $\RR$-superperfect if $H_2(G;\RR)\cong H_1(G;\RR)=0$. 

The next lemma gives necessary and sufficient conditions for a group to be 
$\RR$-perfect or strongly $\RR$-perfect. 

\begin{Lem} \label{p:axb=0}
Let $R$ be a commutative ring, and let $A$ be an abelian group. 
\begin{enuma} 

\item If $\chr(R)=n\ne0$, then $A$ is $R$-perfect if and only if $A$ is 
$n$-divisible (i.e., $A=nA$), and $A$ is strongly $R$-perfect if and only 
if $A$ is uniquely $n$-divisible.

\item If $\chr(R)=0$, then $A$ is strongly $R$-perfect if and only if 
each element of $A$ is annihilated by an integer that is invertible in $R$. 

\item If $\chr(R)=0$ and $A$ is $R$-perfect, then each element of $A$ is 
annihilated by an integer that is invertible in $R/\tors(R)$, where 
$\tors(R)\subseteq R$ is the ideal of torsion elements in $R$. 

\end{enuma}
\end{Lem}

\begin{proof} Since $\Tor$ sends monomorphisms to monomorphisms and 
$\Tor(\Z/p,\Z/p)\ne0$ for each prime $p$, we have
	\beqq \Tor(A,R)=0 ~\implies~ \textup{there is no prime $p$ 
	for which $R$ and $A$ both have $p$-torsion.} 
	\label{e:Tor(A,R)=0} \eeqq

\smallskip

\noindent\textbf{(a) } Assume $\chr(R)=n\ne0$. Then $n\cdot(A\otimes R)=0$ 
and $n\cdot\Tor(A,R)=0$. If $A$ is $n$-divisible, then multiplication by 
$n$ induces a surjection from $A\otimes R$ onto itself, and hence $A\otimes 
R=0$. If $A$ is uniquely $n$-divisible, then multiplication by 
$n$ induces an isomorphism from $\Tor(A,R)$ to itself, and hence 
$\Tor(A,R)=0$. 

Conversely, assume $A\otimes R=0$, and fix $x\in A$. Let 
$R_0\le R$ be a finitely generated subgroup such that $1\in R_0$ and 
$x\otimes1=0$ in $A\otimes R_0$, and let $\gen1\le R_0$ be the 
subgroup generated additively by $1$. Then $\gen1\cong\Z/n$ and $R_0$ has 
exponent $n$, so $\gen1$ is a direct factor of $R_0$ as an additive group 
(see, e.g., \cite[Proposition 3, p. 382]{MB}). Let $\psi\:R_0\too\Z/n$ be 
such that $\psi(1)=1$. Then $\Id_A\otimes\psi$ sends $0=x\otimes1\in 
A\otimes R_0$ to $0=x\otimes1\in A\otimes\Z/n\cong A/nA$, and so $x\in 
nA$.

If $A$ is strongly $R$-perfect, then it is $R$-perfect, and hence 
$n$-divisible. By \eqref{e:Tor(A,R)=0}, $A$ is $p$-torsion free for all 
primes $p\mid n$, and hence is uniquely $p$-divisible.

\smallskip

\noindent\textbf{(c) } Assume $\chr(R)=0$ and $A$ is $R$-perfect. Fix $x\in 
A$. Since $x\otimes1=0$ in $A\otimes R$ and hence in 
$A\otimes(R/\tors(R))$, there is a finitely generated additive subgroup 
$R_0\le R/\tors(R)$ containing $1$ such that $x\otimes1=0$ in $A\otimes 
R_0$. Then $R_0$ is a free abelian group, so we can choose a basis 
$\{b_1,\dots,b_k\}$ for $R_0$ and set $1=\sum_{i=1}^kn_ib_i$ ($n_i\in\Z$). 
Thus $n_ix=0$ in $A$ for each $i$ since $x\otimes1=0$, and if we set 
$n=\gcd\{n_i\}$, then $nx=0$ and $1\in nR_0$. So $x$ is $n$-torsion for 
some $n$ invertible in $R/\tors(R)$. 

\smallskip

\noindent\textbf{(b) } Assume $\chr(R)=0$. If $nA=0$ for some integer 
$n>0$, then $n\cdot(A\otimes R)=0$ and $n\cdot\Tor(A,R)=0$. If, in 
addition, $\frac1n\in R$, then both of these groups are $n$-torsion free, 
and thus $A\otimes R=0=\Tor(A,R)$. More generally, if each element of $A$ 
is annihilated by some integer invertible in $R$, then $A$ is the colimit 
(or union) of its $n_i$-torsion subgroups for some increasing sequence 
$n_1\,|\,n_2\,|\,n_3\,|\,\cdots$ of integers invertible in $R$, and $A$ is 
strongly $R$-perfect since $(-\otimes R)$ and $\Tor(-,R)$ commute with such 
colimits. 

Assume conversely that $A$ is strongly $R$-perfect. By (c), each element of 
$A$ is $n$-torsion for some $n$ invertible in $R/\tors(R)$. By 
\eqref{e:Tor(A,R)=0} and since $\Tor(A,R)=0$, if $A$ has $n$-torsion, then 
$R$ has $m$-torsion only for $m$ prime to $n$. Let $r\in R$ be such that 
$nr=1+t$ for $t\in\tors(R)$. Then $mt=0$ for some $m$ prime to $n$, and 
$n\cdot(mr)=m$. Thus $m\cdot1\in nR$, $n\cdot1\in nR$, and so $1\in nR$ since 
$(m,n)=1$. Hence $n$ is invertible in $R$. 
\end{proof}

When $X$ is a connected CW complex and $H\nsg\pi_1(X)$, the usual plus 
construction for $X$ with respect to $H$ (the case $\RR=\Z$) exists if and 
only if $H$ is perfect. In the more general situation with which we are 
working, the conditions are slightly more complicated.

\begin{Prop} \label{plus}
Let $\RR$ be a commutative ring, let $X$ be a connected CW complex, 
and let $H\nsg\pi_1(X)$ be a normal subgroup. Then there is an $R$-plus 
construction for $(X,H)$ if and only if either
\begin{itemize} 
\item $\chr(R)\ne0$ and $H$ is $R$-perfect, or 
\item $\chr(R)=0$ and $H$ is strongly $R$-perfect.
\end{itemize} 
\end{Prop}

\begin{proof} Set $\pi=\pi_1(X)/H$. Let $\til{X}$ be the covering space of 
$X$ with fundamental group $H$: a space with a free $\pi$-action. To 
shorten notation, we write $H_*(Y)=H_*(Y;\Z)$ when $Y$ is a space or a 
group.

\smallskip

\noindent\boldd{($\,\Longleftarrow$\,) } This is essentially Quillen's 
construction. Assume $H$ is $R$-perfect, and is strongly $R$-perfect if 
$\chr(R)=0$. Attach 2-cells to $\til{X}$ in free $\pi$-orbits to obtain a 
free, simply connected $\pi$-space $\til{X}^+_0\supseteq\til{X}$. 

The homology sequence for the pair $(\til{X}^+_0,\til{X})$ with integer 
coefficients takes the form 
	\[ \Right3{} H_2(\til{X}^+_0) \Right3{\phi} H_2(\til{X}^+_0,\til{X}) 
	\Right3{} H_1(\til{X}) \Right2{} 0, \]
where $H_2(\til{X}^+_0)\cong\pi_2(\til{X}^+_0)$ by the Hurewicz theorem. 
Let $B=\{b_i\}_{i\in I}$ be a basis for $H_2(\til{X}^+_0,\til{X})$ as a 
free $\Z\pi$-module. 
If $\chr(\RR)=0$, then by Lemma \ref{p:axb=0}(b) and since 
$H_1(\til{X})\cong H_1(H)$ is strongly $R$-perfect, we can replace each 
$b_i\in B$ by $r_ib_i\in\Im(\phi)$ for some $r_i\in\Z$ invertible in $\RR$, 
the set $\{r_ib_i\}_{i\in I}$ forms a basis of 
$H_2(\til{X}^+_0,\til{X})\otimes \RR$ as an $\RR\pi$-module, and thus a 
basis for $H_2(\til{X}^+_0,\til{X};R)$ can be lifted back to 
$\pi_2(\til{X}^+_0)$. 
If $\chr(R)=n>0$, then by Lemma \ref{p:axb=0}(a) and since $H_1(\til{X})$ 
is $R$-perfect, each $b_i\in B$ has the form $b_i=c_i+nd_i$ for 
$c_i\in\Im(\phi)$ and $d_i\in H_2(\til{X}^+_0;\til{X})$, the sets 
$\{c_i\}_{i\in I}$ and $B$ induce the same basis of 
$H_2(\til{X}^+_0,\til{X};R)$, and so we can again lift this back to 
$\pi_2(\til{X}^+_0)$. 

Thus in either case, free $\pi$-orbits of 3-cells can be attached to 
$\til{X}^+_0$ to obtain a free $\pi$-space $\til{X}^+_{\RR}$ with 
$H_*(\til{X}^+_{\RR},\til{X};\RR)=0$. Set $X^+_{\RR}=\til{X}^+_{\RR}/\pi$ 
and let $\kappa\:X\too X^+_{\RR}$ be the inclusion; then $\pi_1(\kappa)$ is 
surjective with kernel $H$. Also, by Lemma \ref{cover->twisted} and since 
$\til\kappa\:\til{X}\too\til{X}^+_{\RR}$ is an $\RR$-homology equivalence, 
$\kappa$ induces an isomorphism in homology with coefficients in any 
$\RR\pi$-module. So $X^+_{\RR}$ is an $\RR$-plus construction for $(X,H)$. 

\smallskip

\noindent\boldd{($\implies$) } Assume $X_R^+\supseteq X$ is an $R$-plus 
construction for $(X,H)$. Let $\til{X}^+_R$ be the universal cover of 
$X^+_R$, and regard $\til{X}$ as a subspace of $\til{X}^+_R$. Thus 
$\til{X}_R^+$ is simply connected, $H_*(\til{X}_R^+,\til{X};R)=0$, and 
$H_i(\til{X}_R^+,\til{X})=0$ for $i=0,1$. Also, $H_2(\til{X}_R^+,\til{X})$ 
surjects onto $H_1(\til{X})\cong H_1(H)$. 

Consider the chain complex 
	\[ \Right3{\2_4} C_3(\til{X}_R^+,\til{X}) \Right3{\2_3} 
	C_2(\til{X}_R^+,\til{X}) \Right2{\2_2} 
	C_1(\til{X}_R^+,\til{X}) \Right2{\2_1} 
	C_0(\til{X}_R^+,\til{X}) \too 0 \]
of free $\Z\pi$-modules. 
Set $F_0=\Ker(\2_2)$, set $F_i=C_{i+2}(\til{X}_R^+,\til{X})$ for 
$i\ge1$, and set $\rho_i=\2_{i+2}$ for $i\ge0$. Thus $F_i$ is a free 
abelian group for each $i$ (since subgroups of free abelian groups 
are free), and $H_2(\til{X}_R^+,\til{X})\cong F_0/\Im(\rho_1)$. So we 
get another chain complex
	\[ \Right3{\rho_3} F_2 \Right3{\rho_2} F_1 \Right3{\rho_1} F_0 
	\Right3{\gee} H_1(\til{X}) \Right2{} 0, \]
where $\gee$ is surjective and is induced by the boundary map from 
$H_2(\til{X}_R^+,\til{X})$ to $H_1(\til{X})$. 
Note that the sequence $(F_*\otimes R,\rho_*\otimes\Id_R)$ is exact since 
$H_*(\til{X}_R^+,\til{X};R)=0$.

Set $M=F_0/\Im(\rho_1)$. Thus $M$ surjects onto $H_1(\til{X})\cong 
H_1(H)$, and $M\otimes R=0$ since $F_1\otimes R$ surjects onto 
$F_0\otimes R$. In particular, $H_1(H)\otimes R=0$ since $M$ surjects onto 
$H_1(H)$, so $H$ is $R$-perfect, and we are done if $\chr(R)\ne0$.

Now assume $\chr(R)=0$. Consider the diagram 
	\[ \xymatrix@C=30pt@R=30pt{
	& F_2\otimes R \ar[r] \ar[d]^{0} & F_1\otimes R \ar[r] 
	\ar@{->>}[d]^{\beta} & F_0\otimes R \ar[r] \ar@{=}[d] & 0 \\
	0 \ar[r] & \Tor(M,R) \ar[r] & \Im(\rho_1)\otimes R \ar[r] & 
	F_0\otimes R \ar[r] & 0 \rlap{\,,}
	} \]
where the top row is exact since $H_*(\til{X}_R^+,\til{X};R)=0$. The bottom 
row is exact since $M=F_0/\Im(\rho_1)$ and $F_0$ is a free abelian group. 
The right hand square commutes by naturality, and the left hand square 
commutes since the composite $F_2\xto{~\rho_2~} 
F_1\xto{~\rho_1~}\Im(\rho_1)$ is zero. Also, $\beta$ is onto since 
$F_1\too\Im(\rho_1)$ is onto. An easy diagram chase (or the snake lemma) 
now shows that $\Tor(M,R)=0$. So $M$ is strongly $R$-perfect. Since $M$ 
surjects onto $H_1(\til{X})\cong H_1(H)$, $H$ is also strongly $R$-perfect by 
the characterization in Lemma \ref{p:axb=0}(b). 
\end{proof}

\begin{Ex} \label{ex:prodring}
Set $R=\Q\times\Z/p$ for some prime $p$. Then $p$ is the only prime not 
invertible in $R$, so $(X,H)$ admits an $R$-plus construction ($H$ is 
strongly $R$-perfect) if and only if $H/[H,H]$ is torsion prime to $p$. The 
group $\Z/p^\infty$ is $R$-perfect, since $\Z/p^\infty\otimes R\cong 
(\Z/p^\infty\otimes\Q)\oplus(\Z/p^\infty\otimes\Z/p)=0$, but it is not strongly 
$R$-perfect. 

As another example, consider $R=\prod_p\Z/p$, with the product taken over 
all primes $p$. Then $\chr(R)=0$, no prime is invertible in $R$, and every 
prime is invertible in $R/\tors(R)$. For each prime $p$, $\Z/p^\infty$ is 
$R$-perfect but not strongly $R$-perfect.
\end{Ex}

The next lemma shows that product rings $R$ as in Example 
\ref{ex:prodring} are essentially the only ones with $\chr(R)=0$ that admit 
$R$-perfect groups that are not strongly $R$-perfect. By Lemma 
\ref{p:axb=0}(b,c), this is possible only if there is a prime $p$ that is 
invertible in $R/\tors(R)$ but not in $R$.

\begin{Lem} \label{R-splits}
Fix a prime $p$ and a commutative ring $R$, and let $\tors(R)\subseteq R$ 
be the ideal of torsion elements. Assume $\chr(R)=0$, $p$ is not invertible 
in $R$, and $p$ is invertible in $R/\tors(R)$. Then $R\cong R_1\times R_2$ 
where $R_1$ and $R_2$ are (nonzero) rings, $p$ is invertible in $R_1$, 
and $\chr(R_2)=p^k$ for some $k$.
\end{Lem}

\begin{proof} Since $p$ is invertible in $R/\tors(R)$, there is $n>0$ such 
that $n\cdot1\in npR$. Let $k\ge0$ be the largest power of $p$ 
dividing $n$. Thus $n\cdot1\in p^{k+1}R$ and $p^{k+1}\cdot1\in p^{k+1}R$, 
so $p^k\cdot1\in p^{k+1}R$ since $p^k=\gcd(n,p^{k+1})$. Let $r\in R$ be 
such that $p^k=p^{k+1}r$. Then 
	\[ p^k = p^k(pr) = p^k(pr)^2 = \dots = p^k(pr)^k 
	\quad\implies\quad (pr)^k=p^kr^k=p^kr^k(pr)^k=(pr)^{2k}. \]
Set $e=(pr)^k$, so that $e^2=e$ and $p^ke=p^k$. Thus $R=eR\times(1-e)R$, 
where $p$ is invertible in $eR$ since $e=e^2=p(p^{k-1}r^ke)$ is the 
identity in $eR$, and $\chr((1-e)R)\mid p^k$ since $p^k(1-e)=0$. Since 
$\frac1p\notin R$ and $\chr(R)=0$, both factors are nonzero.
\end{proof}

As one application of Lemma \ref{R-splits}, one can show that for a 
commutative ring $R$ with $\chr(R)=0$ and torsion ideal $\tors(R)$, an 
abelian group $A$ is $R$-perfect if and only if 
\begin{itemize} 
\item $A$ is a torsion group, 
\item $A$ has $p$-torsion only for primes $p$ invertible in $R/\tors(R)$, and 
\item $A$ is $p$-divisible for all primes $p$ not invertible in $R$.
\end{itemize}
This was the one case missing in Lemma \ref{p:axb=0}, when characterizing 
$R$-perfect and strongly $R$-perfect groups. 

Under certain conditions, completion or fibrewise completion as defined by 
Bousfield and Kan gives another, more functorial way to construct plus 
constructions.

\begin{Lem} \label{|X|pcom->Bpi}
Let $\pi$ be a group, and let $\theta\:X\too B\pi$ be a map of spaces where 
$X$ is connected and $\pi_1(\theta)$ is onto. Set $H=\Ker(\pi_1(\theta))$, 
and let $\til{X}$ be the covering space of $X$ with covering group $\pi$ 
and fundamental group $H$. 

\begin{enuma} 

\item Assume that $\RR$ is a subring of $\Q$ or $\RR=\F_p$ for some prime 
$p$, and also that $H$ is $\RR$-perfect. Let $\5\theta\:X^{\wedge}\too 
B\pi$ be the fibrewise $\RR$-completion of $X$ over $B\pi$. Then 
$\kappa\:X\too X^\wedge$ is an $\RR$-plus construction for $(X,H)$.

\item Assume, for some $\RR\subseteq\Q$, that $H$ is 
$\RR$-perfect, and that $\pi$ is $\RR$-nilpotent and has nilpotent action 
on $H_i(\til{X};\RR)$ for each $i$. Then the $\RR$-completion map 
$\kappa\:X\too X\Rcom$ is an $\RR$-plus construction for $(X,H)$.

\item Assume, for some prime $p$, that $\pi$ is a finite $p$-group and 
$H$ is $p$-perfect. Then the $p$-completion map $\kappa\:X\too X\pcom$ 
is a $\kk$-plus construction for $(X,H)$ for each field $\kk$ of 
characteristic $p$.

\end{enuma}
\end{Lem}

\begin{proof} \noindent\textbf{(a) } Since $H$ is $\RR$-perfect, 
$H_1(\til{X};\RR)\cong H_1(H;\RR)=0$ by definition, and hence 
$\til{X}\Rcom$ is simply connected by \cite[Lemma I.6.1]{BK} (applied with 
$k=1$). Since $\til{X}$ is the homotopy fibre of $\theta\:X\too B\pi$, its 
$R$-completion $\til{X}\Rcom$ is the homotopy fibre of $\5\theta$ by 
\cite[Corollary I.8.3]{BK}. Thus $\5\theta$ induces an isomorphism 
$\pi_1(X^{\wedge})\cong\pi$, and $\til{X}\Rcom$ is the universal cover of 
$X^\wedge$. 

Since $\til{X}\Rcom$ is simply connected, it is $\RR$-good by Proposition 
V.3.4 or VI.5.3 in \cite{BK}, and hence $\til{X}$ is $\RR$-good by 
\cite[Proposition I.5.2]{BK}. So $\kappa_0\:\til{X}\too \til{X}\Rcom$ is an 
$\RR$-homology equivalence. By Lemma \ref{cover->twisted}, 
$\kappa\:X\too X^\wedge$ induces an isomorphism in homology with 
coefficients in arbitrary $\RR\pi$-modules, and hence $\kappa$ is an $\RR$-plus 
construction for $(X,H)$. 

\smallskip

\noindent\textbf{(b) } If $\RR\subseteq\Q$ and $\pi$ is $\RR$-nilpotent, 
then $B\pi$ is $\RR$-complete by \cite[Proposition V.2.2]{BK}. If, in 
addition, $H$ is $\RR$-perfect and the action of $\pi$ on 
$H_i(\til{X};\RR)$ (equivalently, on $H_i(\til{X};\RR)$) is nilpotent for each 
$i$, then fibrewise completion over $B\pi$ is the same as $\RR$-completion by 
the mod-$R$ fibration lemma \cite[II.5.1]{BK}, and the result follows from 
(a).

\smallskip

\noindent\textbf{(c) } If $\pi$ is a finite $p$-group, then $B\pi$ is 
$p$-complete by \cite[VI.3.4 and VI.5.4]{BK}. Hence fibrewise completion 
over $B\pi$ is the same as $p$-completion by the mod-$R$ fibration lemma 
\cite[II.5.1 and II.5.2.iv]{BK}, and $\kappa\:X\too X\pcom$ is an 
$\F_p$-plus construction by (a) (applied with $\RR=\F_p$) when $H$ is 
$p$-perfect. If $\kk$ is an arbitrary field of characteristic $p$, then 
$\kappa$ is also a $\kk$-plus construction since $H_*(-;\kk)\cong 
H_*(-;\F_p)\otimes_{\F_p}\kk$. 
\end{proof}

\end{document}